\documentclass[12pt,centertags,oneside]{amsart}
\usepackage{amstext,amsthm,amscd}
\usepackage{amsmath,amssymb}

\usepackage[a4paper,width=16.2cm,top=3cm,bottom=3cm]{geometry}

\usepackage[mathscr]{eucal}
\usepackage{mathrsfs}
\usepackage{charter}
\usepackage{colonequals}

\usepackage{xcolor}
\usepackage[colorlinks=true,
            linkcolor=blue,
            citecolor=blue,
            urlcolor=blue]{hyperref}

\numberwithin{equation}{section}

\newtheorem{theorem}{Theorem}[section]
\newtheorem{definition}[theorem]{Definition}
\newtheorem{proposition}[theorem]{Proposition}
\newtheorem{corollary}[theorem]{Corollary}
\newtheorem{lemma}[theorem]{Lemma}
\newtheorem{remark}[theorem]{Remark}

\newtheorem{question}[theorem]{Question}
\newtheorem{conjecture}[theorem]{Conjecture}
\newtheorem{problem}[theorem]{Problem}

\newcommand{\Aut}{{\rm Aut}}

\newcommand{\supp}{{\rm supp}}

\newcommand{\alg}{\textnormal{alg}}

\renewcommand{\Re}{\mathop{\mathrm{Re}}}

\newcommand{\hyp}{\mathop{\mathrm{hyp}}\nolimits}
\newcommand{\chyp}{\mathop{\underline{\mathrm{hyp}}}\nolimits}

\newcommand{\sing}{{\rm sing}}

\newcommand{\FS}{{\rm FS}}
\newcommand{\B}{\mathbb{B}}
\newcommand{\D}{\mathbb{D}}
\newcommand{\C}{\mathbb{C}}

\newcommand{\N}{\mathbb{N}}
\newcommand{\Z}{\mathbb{Z}}
\newcommand{\R}{\mathbb{R}}

\renewcommand\P{\mathbb{P}}

\title[]{\quad Quantitative hyperbolicity for complex manifolds \break via numerical invariants}

\author{Tien-Cuong Dinh, Duc-Bao Nguyen, Duc-Viet Vu}

\newcommand{\Addresses}{
{\bigskip
\footnotesize

		\textsc{Tien-Cuong Dinh, National University of Singapore, Department of Mathematics, 10 Lower Kent Ridge Road, 119076, Singapore.}
		\noindent
		\par\nopagebreak
		\noindent
		\textit{E-mail address}: \texttt{matdtc@nus.edu.sg}
}
{
		\bigskip
  \footnotesize
		
  \textsc{Duc-Bao Nguyen, National University of Singapore, Department of Mathematics, 10 Lower Kent Ridge Road, 119076, Singapore.}
		\noindent
		\par\nopagebreak
		\noindent
		\textit{E-mail address}: \texttt{ducbao.nguyen@u.nus.edu}}
        
{
		\bigskip
		\footnotesize
		\textsc{Duc-Viet Vu, University of Cologne, Division of Mathematics, Department of Mathematics and Computer Science, Weyertal 86-90, 50931, K\"oln.}
		\noindent
		\par\nopagebreak
		\noindent
		\textit{E-mail address}: \texttt{dvu@uni-koeln.de}
}}

\begin{document}

\hyphenpenalty=10000

\sloppy

\date{\today}

\begin{abstract}
    We introduce numerical invariants called hyperbolic indices, which measure the hyperbolicity of compact K\"ahler manifolds using directed positive closed currents. We prove that if a manifold $X$ has positive hyperbolic indices, then $X$ is Kobayashi hyperbolic; and if $X$ satisfies Demailly's condition of negative jet curvature, then it has positive hyperbolic indices. In particular, by combining the method of jet differentials and density currents, we can prove that for a general hypersurface $X_d$ of degree $d$ in $\P^{n+1}$, the hyperbolic indices of $X_d$ grows to $\infty$ with at least linear growth in $d$. Finally, we discuss an analytic approach to the Kobayashi conjecture. 
\end{abstract}

\medskip

\maketitle

\noindent {\bf Classification AMS 2020:} 32Q45, 32U15, 32Q15.

\smallskip


\tableofcontents

\section{Introduction}

A complex manifold $X$ is called \textit{Kobayashi hyperbolic} if its Kobayashi pseudo-metric is a metric, and \textit{Brody hyperbolic} if there exist no non-constant entire curves $f:\C\to X$. It is known that Kobayashi hyperbolic implies Brody hyperbolic, and when $X$ is compact, these two notions coincide by Brody's lemma.

Such manifolds are called hyperbolic manifolds and have been studied intensively (see \cite{diverio-survey} for a recent survey). A central problem in complex hyperbolicity is the Kobayashi conjecture, which predicts that a general projective hypersurface $X$ of degree $d$ in $\P^n$ with $d \geq 2n-1$ is hyperbolic. It has been proved in \cite{siu-hyperbolic-idea,siu-hyperbolic-full,brotbek-ihes} that general projective hypersurfaces of degree $d \gg n$ in $\P^n$ are hyperbolic. See also \cite{diverio-merker-rousseau-effective,riedl-yang-grassmannian-technique}. While a number of papers have refined the aforementioned works by providing an effective number $d$ (see \cite{deng-effectivity,berczi-kirwan-non-reductive-hyperbolicity,cadorel-hyperbolicity-viaGG} for results in all dimensions, see \cite{DemaillyElGoul00,paun-dim2-hyperbolicity,diverio-trapani-GGlocus-codim2,HouHuynhMerkerXie26,CuiHouLiuXie26} for results in low dimensions),  the optimal degree in the Kobayashi conjecture has still remained elusive to obtain. A key method employed in these studies has been the use of jet differentials, which reduces the problem to algebraic computations.

The aim of this article is to revisit complex hyperbolicity and the Kobayashi conjecture from a different and more analytic perspective through the theory of positive currents. More precisely, we introduce and study a family of numerical invariants which measure, in a quantitative way, the hyperbolicity of a complex manifold.

One of the main motivations for our theory is Demailly's notion of \textit{algebraic hyperbolicity} for projective manifolds, introduced in \cite{demailly-97-algebraic-hyperbolicity}. Recall that a projective manifold $X$ is algebraically hyperbolic if we have
\begin{equation}\label{eq demailly alg condition}
\hyp_{\alg}(X,\omega)\colonequals \inf_\mathscr{C} \Big\{\frac{-\chi(\widehat{\mathscr{C}})}{\int_\mathscr{C} \omega}\Big\} >0,\end{equation}
where the infimum runs over all irreducible closed curves.
Here, $\widehat{\mathscr{C}}$ is a normalization of $\mathscr{C}$, $\chi(\widehat{\mathscr{C}})$ is the Euler characteristic of $\widehat{\mathscr{C}}$, and $\omega$ is a K\"ahler form on $X$. As shown in \cite{demailly-97-algebraic-hyperbolicity}, Kobayashi hyperbolicity implies this notion.  A version of Kobayashi's conjecture using algebraic hyperbolicity instead of Kobayashi's was proved in \cite{clemens-86-curves-hypersurfaces,ein88-subvarieties-complete-intersect,ein91-subvarieties-complete-intersect-II,voisin-rational-curve-hypersurfaces}, see also \cite[Corollary 2.2.5]{diverio-survey}.

Another motivation comes from Nevanlinna theory. Given a non-constant entire curve $f:\C \to X$ in a compact complex manifold $X$, one can associate to $f$ positive closed currents of bi-dimension $(1,1)$, which are called \textit{Nevanlinna currents} (for an explicit construction, see Section~\ref{sec complex hyperbolicity}). This class of currents has led to the use of pluripotential theory in the study of complex hyperbolicity. However, in general, these currents are not unique and may possess pathological properties (we refer to \cite{Dinh_Vu_flot,HuynhXie-viduAhlfors,xie-imrn,wu-Xie-entirecurves-torus,Chen-Foraess-Xie} for more information on this class of currents). This motivates the need for \textit{softer} objects to make the problem more amenable to analytic techniques.

Motivated by the theory of directed currents in holomorphic dynamics and foliation theory (see \cite{Dinh-suite-laminaire,fornaess-sibony-survet-laminiantion,sibony-pfaff}), we will introduce the notion of \textit{liftable currents} on compact K\"ahler manifolds $X$ that contain both closed curves and Nevanlinna currents on $X$. Recall that the Demailly-Semple jet bundle of $X$ contains a sequence of complex directed manifold $(X_k,V_k)_{k\in \Z_{\geq 0}}$, which can be seen as a natural compactification of jet spaces. Denote by $\pi_{k,0}:X_k\to X$ be the canonical projection. We say a positive closed current $T$ of bi-dimension $(1,1)$ is \textit{$k$-liftable} if there exists a positive closed current $T_{[k]}$ of bi-dimension $(1,1)$ on $X_k$ such that $(\pi_{k,0})_* (T_{[k]}) = T$ and $T_{[k]}$ is directed by $V_k$. $T_{[k]}$ is called a \textit{geometric $k$-lifting} of $T$ if $T_{[k]}$ puts no mass on the singular locus of $X_k$. We refer the reader to Section~\ref{sec demailly-semple jet and liftable currents} for details on lfitable currents. A crucial property is that the class of liftable currents is closed under weak limits of currents. We will consider both liftable positive closed and $dd^c$-closed currents. An advantage of working with liftable $dd^c$-closed currents is that $1$-liftable $dd^c$-closed currents always exist on general compact manifolds (see Proposition~\ref{lemma existence liftable current induction level}).

In order to define the hyperbolic index of a manifold, we will replace the closed curves in the condition~\eqref{eq demailly alg condition} by liftable currents. This allows us to take into account the information from Nevanlinna currents as well. In particular, we can generalize the notion of the  Euler characteristic of curves to liftable currents.

For every liftable current $T$, we denote by ${\chi}_k(T)$ the geometric $k$-Euler characteristic of $T$ (see Section~\ref{sec pluri Euler char} for details). When $T$ is the current of integration over a curve $\mathscr{C}$, $\chi_k(T)$ converges to $\chi(\widehat{\mathscr{C}})$ as $k\to \infty$. However, due to the appearance of the singular locus $X_k^{[\sing]}$ when $k\geq 2$, this geometric notion is not well suited for working with Nevanlinna currents. This difficulty reflects the transcendental feature of entire curves, in contrast with the case of algebraic curves. To bypass this difficulty, we introduce the notion of cohomological $k$-Euler characteristic of liftable current $T$, which is more flexible to work with. This notion is denote by $\overline{\chi}_k(T)$ since it always greater than or equal to $\chi_k$.

The hyperbolic indices of a compact K\"ahler manifold $(X,\omega)$ are defined by 
\[\hyp^{(k)}(X,\omega) \colonequals \inf \Big\{\frac{-{\chi}_k(T)}{\|T\|_\omega}\Big\}, \qquad \chyp^{(k)}(X,\omega) \colonequals \inf \Big\{\frac{-\overline{\chi}_k(T)}{\|T\|_\omega}\Big\},\]
where the infimum runs over all liftable currents. This notion only depends on the cohomology class of $\omega$. The positivity of hyperbolic indices does not depend on the choice of $\omega$. In many cases, we will work with submanifolds of $\P^n$, where the canonical choice of metric is the one induced by the Fubini-Study metric $\omega_{\FS}$. We can also define hyperbolic indices without the reference metric $\omega$.

By the uniformization theorem, if a closed curve $\mathscr{C}$ contains the image of a non-constant entire curve, then $\mathscr{C}$ is a rational or elliptic curve. This implies that the geometric Euler characteristic of $\mathscr{C}$ is non-negative. Moreover, by a result of Duval in \cite{duval-Brody-08} (see also \cite{Duval-singularity,Duval-local-Ahlfors}), a Nevanlinna current only charges rational or elliptic curves. Motivated by this observation, we prove that $\overline{\chi}_\infty$ of a Nevanlinna current is a non-negative number. This implies our first main result (see Theorem~\ref{theorem chyp > 0 implies Kobayashi hyperbolic} and Theorem~\ref{theorem hyp big implies hyperbolic in projective case}).

\begin{theorem}\label{theorem first theorem main}
    If $\chyp^{(k)}(X,\omega) > 0$ for some K\"ahler form $\omega$ and $k\in \Z^+ \cup \{\infty\}$, then $X$ is Kobayashi hyperbolic. Moreover, if $X\subset \P^n$ is a projective manifold such that $\hyp^{(k)}(X,\omega_{\FS}|_X) > 2\cdot (3^{k-1}-1)$, then $X$ is Kobayashi hyperbolic.
\end{theorem}

 An important feature of hyperbolic indices is that they admit a monotonicity property: every submanifold $Y$ of $X$ has bigger hyperbolic indices . This reflects the quantitative aspect of these invariants.

It was proved by Demailly in \cite{demailly-97-algebraic-hyperbolicity} that if $X$ admits a negative jet curvature $k$-jet metric for some $k$, then $X$ is Kobayashi hyperbolic. This curvature condition, in fact, also implies the positivity of hyperbolic indices (see Theorem~\ref{theorem negative jet curvature implies pluripotential hyperbolic}).

\begin{theorem}\label{theorem second theorem main}
     If $X$ admits a negative jet curvature $k$-jet metric for some $k\in \Z^+$, then $\hyp^{(\infty)}(X,\omega) > 0$ for every K\"ahler form $\omega$ on $X$.
\end{theorem}

Motivated by a conjecture by Demailly, we expect that the positivity of hyperbolic indices is equivalent to Kobayashi hyperbolicity (see Conjecture~\ref{conj hyp > 0 equiv Kobayashi}).

We now discuss the hyperbolic indices of projective manifolds, which are related to the Kobayashi conjecture. Since the Euler characteristic of liftable currents is invariant under automorphisms (Lemma~\ref{prop liftable preserved}), we can use the ``moving by automorphisms" technique to compute the hyperbolic indices in some special cases. In particular, we prove that $\hyp^{(k)}(\P^n,\omega_{\FS}) = -2$ for all $k$ (Proposition~\ref{prop hyp of Pn}). This implies that $\hyp^{(k)}(X,\omega_{\FS}|_X) \geq -2$ for any submanifold $X$ of $\P^n$. We will see later that the traditional algebraic method of jet differentials not only implies Kobayashi hyperbolicity but also the positivity of hyperbolic indices. This is illustrated by the following transcendental version of the fundamental vanishing theorem for entire curves (see Theorem~\ref{theorem alpha_km positive implies bound on -chi_k} for a more complete version).

 \begin{theorem}\label{theorem transcendental fundamental vanishing theorem (light version)} Let $X$ be a compact K\"ahler manifold. Fix a K\"ahler class $\alpha$ on $X$.  Suppose there exist $k\in \Z^+$ and $ m \in \R^+$ such that the cohomology class
 \[c_1(\mathcal{O}_{X_k}(m)) - \pi_{k,0}^*(\alpha) \in H^{1,1}(X_k,\R) \]
 is pseudoeffective. Let $S$ be a positive closed $(1,1)$-current in this class and $Z$ is the set of all points $x\in X_k$ such that the Lelong number of $S$ at $x$ is positive. Let $T$ be a liftable current such that $\overline{\chi}_k(T) \geq 0$. Then there exists a $k$-lifting $T_{[k]}$ of $T$ such that $T_{[k]}$ has positive mass on $Z$. Moreover, if the cohomology class $\{T\}$ is movable, we only need $\alpha$ to be a big class.
 \end{theorem}

We note that in the case of complex projective surfaces, the cohomology class of the Nevanlinna current of a transcendental entire curve is a nef (and therefore movable) class. Thus, this theorem may be useful in the study of entire curves on complex surfaces of general type. The technical core of these results is the theory of density currents introduced by the first named author and Sibony in \cite{Dinh_Sibony_density}. We underline that this theorem is true without the directed structure, which motivates the search for more delicate algebraic obstructions (see Remark~\ref{remark search for more algebraic obstruction}).

By employing previous results on the Kobayashi conjecture and the Green-Griffiths conjecture, we can prove effective estimates for hyperbolic indices of general hypersurfaces in $\P^{n+1}$ (compare with the Kobayashi conjecture).

 \begin{theorem}\label{theorem quantitative Kobayashi conjecture}
     Let $X_d$ be a general hypersurface of degree $d$ in $\P^{n+1}$. Then, there exist effective positive functions $\delta(n)$ and $\lambda(n)$ such that
     \[\hyp^{(n)}(X_d,\omega_{\FS}|_{X_d}) \geq \chyp^{(n)}(X_d,\omega_{\FS}|_{X_d}) \geq \delta(n) d - \lambda(n) \text{ for every }d.\]
 \end{theorem}

Finally, motivated by Siu's semicontinuity theorem \cite{Siu}, we propose an analytic approach to the Kobayashi conjecture. The first step involves developing an analogous theory in the logarithmic setting. Next, one could employ deformation methods to construct examples of low-degree hypersurfaces with positive hyperbolic indices. The conjecture would then follow, provided that a semicontinuity theorem in the (countable) Zariski topology for the hyperbolic indices can be established. We expect the analytical softness of liftable currents to play a crucial role in this strategy. See Section~\ref{sec open problem} for further details.

The paper is organized as follows. In Section~\ref{sec complex hyperbolicity}, we recall basic results in complex hyperbolicity and Nevanlinna theory, followed by the necessary background on currents in Section~\ref{sec background currents}. Section~\ref{sec demailly-semple jet and liftable currents} reviews Demailly-Semple jet bundles and employs them to construct liftable currents. In Section~\ref{sec pluri Euler char}, we define the geometric Euler characteristic for liftable currents and establish the basic properties of this notion. The notion of hyperbolic indices will be introduced and studied in Section~\ref{sec hyperbolic indices introduce and basic}. The relationship between these hyperbolic indices and Kobayashi hyperbolicity is explored in Section~\ref{sec relation with Kobayashi hyperbolic}. In Section~\ref{sec quantitative Kobayashi conjecture}, we apply the method of jet differentials to prove our main theorems regarding the hyperbolic indices of general hypersurfaces. Finally, Section~\ref{sec open problem} concludes the paper with a discussion of open problems and further developments.

\hspace{1cm}

\noindent \textbf{Acknowledgment.} T.-C. D has received funding from the National University of Singapore and MOE of Singapore through the
grants A-8002488-00-00 and A-8003576-00-00. D.-B. N is supported by the Singapore International Graduate Award (SINGA) and the Overseas Research Immersion Award (ORIA). D.-V. V is supported by the DFG grant VU 126/1-1. Part of this work was carried out during the visit of the second-named author to the Department of Mathematics and Computer Science at the University of Cologne. He would like to thank the department for their warm welcome and support. He would also like to thank  Dinh Tuan Huynh, Vi\^{e}t-Anh Nguy\^{e}n, Mihai P\u{a}un and Song-Yan Xie for their helpful discussions.

\section{Complex hyperbolicity and Nevanlinna theory}\label{sec complex hyperbolicity}

 \subsection{Notions of hyperbolicity}

 We first recall various notions on the hyperbolicity of complex manifolds (see \cite{kobayashi-book,diverio-survey,Noguchi} for more comprehensive treatments).

 Let $X$ be a complex (not necessarily compact) manifold. We denote by $\D$ the unit disc in $\C$. Given two points $x,y\in X$. A \textit{chain of holomorphic discs from $x$ to $y$} is given by a chain of points $z_0 \colonequals x,z_1,\ldots,z_k\colonequals y$ in $X$, pairs of points $a_1,b_1,\ldots,a_k,b_k$ in $\D$, and holomorphic discs $f_1,\ldots,f_k:\D\to X$ such that $f_j(a_j)= z_{j-1},f_j(b_j) = z_j$ for $j=1,\ldots,k$. Denoting this chain by $\alpha$, we define the length of this chain 
 \[\mathrm{length} (\alpha)\colonequals \rho(a_1,b_1)+\cdots +\rho(a_k,b_k),\] where $\rho$ is the Poincar\'e distance on $\D$. The \textit{Kobayashi pseudometric $\mathbf{d}_X$} on $X$ is defined by $\mathbf{d}_X\colonequals \inf \mathrm{length}(\alpha)$, where the infimum is taken over all the chains of holomorphic discs. When $\mathbf{d}_X$ is actually a metric, $X$ is called \textit{Kobayashi hyperbolic}. One has an infinitesimal version of the Kobayashi pseudometric as follows. Fix an arbitrary holomorphic tangent vector $v\in T_{x}X, x\in X$. Define
\[\mathbf{k}_{X}(v)\colonequals \inf \Big\{\lambda > 0: \exists f:\D \to X \text{ holomorphic}, f(0) = x, \lambda f'(0) = v \Big\}.\]
It is known that 
\[\mathbf{d}_X(x,y) = \inf_{\gamma} \int_\gamma \mathbf{k}_{X}(\gamma'(t))dt,\]
where the $\inf$ is taken over all piecewise smooth curves $\gamma$ joining $x$ and $y$. A manifold $X$ is called \textit{infinitesimally Kobayashi hyperbolic} if for every compact subset $K\subset X$ we have $\mathbf{k}_X(v)\geq \epsilon \|v\|_{\omega}$, where $v\in T_xX,x\in K$, $\omega$ is a Hermitian metric on $X$, and $\epsilon $ is a positive constant. A manifold $X$ is called \textit{Brody hyperbolic} if there exists no non-constant entire curve $f:\C \to X$. In general, if $X$ is infinitesimally Kobayashi hyperbolic, then $X$ is Kobayashi hyperbolic; if $X$ is Kobayashi hyperbolic, then $X$ is Brody hyperbolic. When $X$ is compact, by Brody's reparametrization lemma, these three notions are equivalent (see \cite{kobayashi-book}). When $X$ is a projective manifold, we have the notion of algebraic hyperbolicity, which has been mentioned earlier. We refer the reader to \cite{gromov-kahler-hyperbolic,diverio-weakly-hyperbolic-GGL} for some related notions on the hyperbolicity of K\"ahler manifolds.

Let $M$ be a relatively compact, locally closed complex manifold immersed in a complex manifold $X$. We say that $M$ is \textit{hyperbolically embedded} in $X$ if $M$ is Kobayashi hyperbolic, and for any two distinct points $x,y \in \partial M$, there are neighborhoods $U$ of $x$ and $V$ of $y$ in $X$ satisfying
\[\inf \Big\{ \mathbf{d}_M(x',y'): x' \in M \cap U, y'\in M \cap V \Big\} > 0.\]
The logarithmic version of the Kobayashi conjecture predicts that $\P^n\setminus D$ is hyperbolically embedded in $\P^n$, where $H$ is a general smooth hypersurface of $\P^n$ with degree $d\geq 2n-1$. We refer the reader to \cite{siuYeung-defects-97,ElGoul03,Rousseau09,brotbek-deng} for results on this conjecture.

Let $X$ be a projective manifold, and $D$ be an effective divisor. Then, the pair $(X,D)$ is called \textit{algebraically hyperbolic} if there exists $\epsilon > 0$ such that
\[-\chi(\widehat{\mathscr{C}}) + i(\mathscr{C},D) \geq \epsilon \int_\mathscr{C} \omega\]
for every irreducible closed curve $\mathscr{C} \subset X$ such that $\mathscr{C}\not\subset D$, where $\nu:\widehat{\mathscr{C}} \to X$ is a normalization on $\mathscr{C}$ and $i(\mathscr{C},D)$ is the number of distinct points in $\nu^{-1}(D)$. This notion was introduced and studied by Chen in \cite{chen-algebraic-hyperbolic-for-log-varieties}. By a result of Pacienza-Rousseau in \cite{pacienza_rousseau_logarithmic_Kobayashi}, if $X\setminus D$ is hyperbolically embedded into $X$, then $(X,D)$ is algebraic hyperbolic.

\subsection{Basics on Nevanlinna theory}

Next, we recall some facts on Nevanlinna theory. The readers are referred to \cite{Ru-book,Noguchi} for more information.

We start with some basic definitions. Let $X$ be a complex manifold, and $f:\C \to X$ be a non-constant entire curve. Let $\D_r$ denote the disc of radius $r>0$ with center $0$ in $\C$.  For a smooth $(1,1)$-form $\eta $ and $r>1$, we define the \textit{characteristic function} \[T_f(r,\eta)\colonequals \int_1^r \frac{dt}{t} \int_{\D_t} f^*\eta.\]
Let $D \subset X$ be a divisor, and denote by $\Theta_D$ a curvature form of a smooth Hermitian metric on $L$ where $L$ is the line bundle associated to $D$. Then, there exists a negative quasi-plurisubharmonic function $\varphi_D$ such that $[D] = \Theta_D + dd^c \varphi_D$ in the sense of currents. Define the \textit{proximity function} of $f$ with respect to $D$ by
\[m_f(r,D)\colonequals -\frac{1}{2\pi}\int_0^{2\pi} \varphi_D \circ f(re^{i\theta})\ d\theta,\]
and the \textit{truncated counting function} of $f$ at level $k$ with respect to $D$ by
\[N_f^{[k]}(r,D)\colonequals \int_1^r \frac{dt}{t}\sum_{|t|< r } \min (k,\mathrm{ord}_t (f^*D)).\]
When $k = \infty$, we write $N_f(r,D)$ instead of $N_f^{[\infty]}(r,D)$. The \textit{defect of truncation $k$} of $D$ with respect to $f$ is defined by
\[\delta^{[k]}_f(D) \colonequals \liminf_{r\to \infty} \Big(1 - \frac{N_f^{[k]}(r,D)}{T_f(r,L)}\Big),\]
where $T_f(r,L)\colonequals T_{f}(r,\Theta_D)$, which is well-defined up to a bounded term as $r \to \infty$.\\
Let $\omega$ be a Hermitian metric on $X$. For $r>1$, we define
\[S_f(r,\omega)\colonequals  \int_1^{r} \frac{dt}{t} \ \mathrm{length}_\omega (f(\partial \D_{t})).\]

We now recall some useful lemmas. The following lemma is known as Nevanlinna's first main theorem (see, e.g., \cite[Theorem A2.3.1]{Ru-book} or \cite[Theorem 2.3.31]{Noguchi}).

\begin{lemma}(Nevanlinna's first main theorem)\label{lemma Nevanlinna first main theorem}
    Suppose that $f(\C)\not\subset \supp (D)$. Then we have $T_f(r,\Theta_D)=N_f(r,D) + m_f(r,D) - m_f(1,D)$.
\end{lemma}

For a proof, see \cite[Theorem A2.3.1]{Ru-book} or \cite[Theorem 2.3.31]{Noguchi}.

\begin{lemma}(Ahlfors' lemma)\label{lemma Ahlfors} Let $\delta > 0$. Then, there exists a measurable set $E(\delta)$ with $\int_{E(\delta)} dr/(r\log r) < \infty$ such that $S_f(r,\omega) \leq \delta T_f(r,\omega)$ for $ r\notin E(\delta)$. In particular, there exists a sequence $r_m \to \infty$ such that $\lim\limits_{m\to \infty} S_f(r_m,\omega)/T_f(r_m,\omega) = 0$.
\end{lemma}

For a proof, see \cite[Theorem 0]{brunella-99-foliations}.

\begin{lemma}\label{lemma logarithmic}(Lemma on the Logarithmic Derivative) There exists a measurable set $E$ of finite Lebesgue measure such that 
\[\int_0^{2\pi} \log^+ |f'(re^{i\theta}| d\theta \leq O(\log T_f(r,\omega)+ \log r) \text{ for } r\notin E.\]
\end{lemma}

For a proof, see \cite[Theorem A1.2.5]{Ru-book} or \cite[Theorem 1.2.2]{Noguchi}

\begin{lemma}\label{lemma Demailly transcendental}
    Suppose that $T_f(r,\omega) = O(\log r)$. Then $f$ is a rational curve.
\end{lemma}

For a proof, see \cite{demailly-algcurve-charfunction} or \cite{duval-aroundBrody}.

\subsection{Positive current associated to entire curves}

We end this section by recalling some constructions of positive currents associated with entire curves. Let $(X,\omega)$ be a \textit{compact} Hermitian manifold and $f:\C \to X$ be a non-constant entire curve. Consider the sequence of currents 
\[S_r\colonequals \frac{1}{T_{f}(r,\omega)} \int_1^r \frac{dt}{t} f_*([\D_t]),\]
which are positive currents of mass $1$ on $X$. By direct computations, we have 
\begin{equation}\label{eq compute ddc for S_r entire curve case}dd^c S_r= \frac{1}{T_f(r,\omega)} (f_* (\mu_r)- f_*(\delta_0)),\end{equation}
where $\mu_{r}$ is the Haar measure on the circle $\partial\D_{r}$ and $\delta_0$ is the Dirac mass at $0$. Therefore, if we take an arbitrary weak limit of the sequence $(S_r)_r$, we obtain a positive $dd^c$-closed current since $T_f(r,\omega)\to \infty$ as $r\to \infty$. We call such limits \textit{Nevanlinna currents associated with $f$}. By Stokes' theorem, for any sequence $(r_m)$ constructed in Lemma~\ref{lemma Ahlfors}, any weak limit of $(S_{r_m})_m$ as $m\to \infty$ is a positive \textit{closed} current.

If we consider the sequence of positive currents $\frac{f_*([\D_r])}{\mathrm{area} \ (f(\D_r))}$,
where area is computed with respect to $\omega$. Then, one can deduce from Lemma~\ref{lemma Ahlfors} that there exists a sequence $r_m\to \infty$ such that any weak limit of $\frac{f_*([\D_{r_m}])}{\mathrm{area} \ (f(\D_{r_m}))}$ as $m\to \infty$ is a positive closed current. We call such limits \textit{Ahlfors currents associated with $f$}. In this article, we will mainly focus on Nevanlinna currents.

 Nevanlinna and Ahlfors currents have found important applications in complex hyperbolicity and Nevanlinna theory (see e.g. \cite{McQuillan,duval-Brody-08,paun-sibony-value,ru-sibony-second-main-hyperbolic}) and in foliation theory (see \cite{fornaess-sibony-survet-laminiantion,Dinh-Nguyen-Sibony-heat-equation-ergodic,DS-uniqueergodicity,DNS-uniqueergodicity}). These constructions can be generalized to the situation of maximal rank holomorphic maps from $\C^p$ to a compact Hermitian manifold $(X,\omega)$, as shown in \cite{gasbarri-pacienza-rousseau-higherdim-tautological-ineq}. See also \cite{dethelin-ahlfors-current-highdim,Burns_Sibony}.

\section{Background on positive currents}\label{sec background currents}

The main tools for our work are positive currents, which are either closed or $dd^c$-closed. We refer the reader to \cite{demailly-complex-book} for basic results on the theory of positive closed currents. We will mainly focus on positive $dd^c$-closed currents in this section.

\subsection{Support theorem}

The first ingredient we need is the support theorem for currents. For positive closed currents, we have the following theorem, which is a consequence of Federer's support theorem for flat currents (see e.g.\cite[Chapter III. Section 2.]{demailly-complex-book}).

\begin{lemma}\label{lemma support lemma for closed current}
    Let $X$ be a complex manifold, and $T$ be a positive closed current of bi-dimension $(p,p)$ on $X$. Let $A$ be an analytic subset of $X$ with irreducible components $A_j$ of pure dimension $k$. Then the following statements hold
    \begin{itemize}
        \item[(i)] $\mathbf{1}_A T$ and $\mathbf{1}_{X\setminus A} T$ are positive closed currents. 
        \item[(ii)] Suppose that $T$ has support in $A$. If $k < p$, then $T =0 $. If $k = p$, then $T = \sum_j c_j[A_j]$ where $c_j \in \R_{\geq 0}$. If $k > p$, then $T$ is a current on $A$. This means that there exists a positive closed current $S$ on $A$ such that $i_*(S) = T$, where $i:A\hookrightarrow X$ is the embedding map.
    \end{itemize}
\end{lemma}

We have a similar support theorem for positive $dd^c$-closed currents, provided some extra conditions on the manifold $X$ and the analytic set $A$. This theorem follows from the support theorem for $\C$-flat currents developed by Bassanelli in \cite{Bassanelli}. Although this result is well-known, we provide a proof here as we cannot find a suitable reference.

\begin{lemma}\label{lemma support lemma for ddc closed current}
    Let $T$ be a positive $dd^c$-closed current of bi-dimension $(p,p)$ with compact support on a K\"ahler manifold $X$. Let $A$ be an analytic subset of $X$ with irreducible components $A_j$ of pure dimension $k$. Then the following statements hold
    \begin{itemize}
        \item[(i)] $\mathbf{1}_A T$ and $\mathbf{1}_{X\setminus A} T$ are positive $dd^c$-closed currents. If $p=1$, then the aforementioned property also holds without the K\"ahler condition on $X$.
        \item[(ii)] Suppose that $A$ is compact and $T$ has support in $A$. If $k < p$, then $T = 0$. If $k > p$, then $T$ is a current on $A$. If $k = p$, then $T = \sum_j c_j[A_j]$ where $c_j \in \R_{\geq 0}$. 
    \end{itemize}
\end{lemma}

\begin{proof}
    We first prove (i). We consider the current $S\colonequals\mathbf{1}_{X\setminus A} T$. This is a positive $dd^c$-closed current on $X\setminus A$ with locally finite mass near $A$. Let $\widetilde{S}$ be the trivial extension of $S$ through $A$ on $X$. We show that $dd^c \widetilde{S} \leq 0$. This is a local problem. We assume for simplicity that $X$ is the ball of center $0$ and radius $2$ in $\C^n$ and will show that $dd^c\widetilde{S} \leq 0$ on the unit ball in $\C^n$.

    Consider the sequence of smooth psh functions $(u_k)$, such that $0\leq u_k\leq 1$, vanishing near $A$ and increasing to $\mathbf{1}_{X\setminus A}$. By using the proof of \cite[Theorem 1.3]{DinhSibony_pullback} (for earlier results, we refer the reader to \cite{Sibony_duke,Bassanelli-extension-analyticsubset,Bassanelli,dabbek-Elkhadhra-ElMir-extension}) for $S$, we get 
    \[\partial u_k \wedge S \to 0, \bar{\partial} u_k \wedge S \to 0.\] Moreover, since $u_kdd^c S -dd^c (u_k S) = dd^c u_k\wedge S - d(d^c u_k \wedge S) + d^c(du_k \wedge S)$ and $dd^c S = 0$ on $X\setminus A$, we have $dd^c \widetilde{S} = \lim\limits_{k\to \infty} dd^c(u_kS) = -\liminf\limits_{k\to \infty}  dd^c u_k \wedge S\leq 0$ as desired.

    We now back to the global setting. Pick a K\"ahler form $\omega$ on $X$. Since $T$ has compact support, by Stokes' theorem, for $\Omega \colonequals \omega^{n-1}$, we have
    \begin{equation}\label{eq ddc tilde S = 0}\int_X dd^c \widetilde{S} \wedge \Omega = \int_X \widetilde{S} \wedge dd^c \Omega = 0.\end{equation}
    As $dd^c \widetilde{S} \leq 0$ , we deduce that $dd^c \widetilde{S} = 0$, which means $\mathbf{1}_{X\setminus A} T = \widetilde{S}$ is a positive $dd^c$-closed current on $X$. Since $\mathbf{1}_{A} T = T - \mathbf{1}_{X\setminus A} T$, we also conclude that $\mathbf{1}_A T$ is $dd^c$-closed. When $p=1$, we do not need a K\"ahler form to get \eqref{eq ddc tilde S = 0}.

    We now prove (ii). The case $k <p$ and $k > p$ has been proven in \cite{Bassanelli}. We only consider the case where $k=p$. By \cite[Theorem 4.10]{Bassanelli}, $T = f[A]$ for some positive weakly psh function on $A$. Since $A$ is compact, $f$ is bounded from above. By taking a desingularization $\pi: \widetilde{A}\to A$ of $A$, we get a bounded from above psh function $f\circ \pi$ on the manifold $\widetilde{A} \setminus E$ where $E$ is the exceptional locus. This psh function can be extended uniquely to a psh function on the compact manifold $\widetilde{A}$. By the maximum principle, $f\circ \pi$ is constant. Hence, $f$ is constant and the result follows.
\end{proof}

\subsection{Density currents}

The next tool in our study is the density theory of currents, which was first introduced for positive \textit{closed} currents by the first named author and Sibony in \cite{Dinh_Sibony_density}. This theory has been extended to positive \textit{closed} currents on non-K\"ahler manifolds in \cite{Vu_density-nonkahler}, positive \textit{$dd^c$-closed} currents on complex K\"ahler surfaces in \cite{DS-uniqueergodicity,DNS-uniqueergodicity,Dinh-Nguyen-Siuforddc},  and to a large class of positive \textit{plurisubharmonic} currents on general complex manifolds in a recent series of works by Nguy\^{e}n \cite{VANguyen-densitycurrent,VA-generalizedLelong}. We will recall some results here and refer the reader to these papers for more details.

\subsubsection{Basic notions}

We start with some definitions. Let $X$ be a K\"ahler manifold of dimension $n$, and $V$ be a complex submanifold of dimension $k$. Let $T$ be a positive $dd^c$-closed current of bi-degree $(p,p)$ on $X$ such that $V \cap \supp (T)$ is compact. Let $\pi : E \to V$ be the normal bundle of $V$ in $X$, and $\overline{E}\colonequals \P(E \oplus \C)$ be the natural compactification of $E$. We also denote by $\pi$ the projection $\pi:\overline{E}\to V$.

A \textit{strongly admissible map along $V$} is a \textit{smooth} diffeomorphism $\tau$ from an open subset $U$ of $X$ to an open neighborhood of $V$ in $E$ such that for every point $x\in V\cap U$, for every local chart $y=(z,w)$ on a neighborhood $W$ of $x$ in $U$ with $V\cap W = \{z= 0\}$, we have 
\[(z+zAz^T + O(\|z\|^3), w + Bz + O(\|z\|^2)),\]
where $A$ is a $(k-l)\times (k-l)$-matrix and $B$ is an $l\times (k-l)$-matrix whose entries are smooth functions in $w$. By \cite[Proposition 3.8]{DS-uniqueergodicity} and \cite[Theorem 1.19 (i)]{VANguyen-densitycurrent}, strongly admissible maps along $V$ always exist.

For $\lambda \in \C^*$, let $A_\lambda : E\to E$ be the multiplication by $\lambda$ on the fibers of $E$. Let $\tau$ be a strongly admissible map defined on a tubular neighborhood of $V$. By \cite[Theorem 4.6]{Dinh_Sibony_density} (in the case where $T$ is closed) and \cite[Theorem 1.18 and Corollary 1.21]{VANguyen-densitycurrent} (in the general case), the  currents 
\[(A_\lambda)_* (\tau)_* (T)\]
has uniformly bounded masses on compact subsets of $E$. Any limit current of this family in $E$ is called a \textit{tangent current to $T$ along $V$}. Such a current is positive $dd^c$-closed, invariant by $A_\lambda$, and can be extended to a current in $\overline{E}$. A tangent current does not depend on the choice of admissible maps. In general, the tangent currents to $T$ along $V$ are not unique.

Suppose that $T$ is closed. Then by \cite[Theorem 4.6]{Dinh_Sibony_density}, the cohomology classes of the tangent currents to $T$ along $V$ in $H^*_c(\overline{E},\C)$ are unique. This class is denoted by $\kappa^{V}(T)$ and is called the \textit{total tangent class to $T$ along $V$}. Let $h_{\overline{E}}$ be the Chern class of $\mathcal{O}_{\overline{E}}(1)$. By the Leray-Hirsch theorem, we can write $\kappa^V(T)$ uniquely as
\begin{equation}\label{eq Leray-Hirsch for closed case}
    \kappa^V(T) = \sum_{j= \max (0,k-p)}^{\min(k,n-p-1)} \pi^*( \kappa_j) \cdot h_{\overline{E}}^{p-(k-j)},
\end{equation}
where $\kappa_j$ is a cohomology class in $H_c^{2k-2j}(V,\C)$.

Suppose that $X$ is a compact K\"ahler manifold. Let $S$ be a $dd^c$-closed current of bi-degree $(p,q)$. Let $\alpha$ and $\alpha'$ be two smooth closed $(n-p,n-q)$ forms in the same cohomology class in $H^{n-p,n-q}(X,\C)$. By the $dd^c$-lemma, there exists a smooth form $f$ of bi-degree $(n-p-1,n-q-1)$ such that $\alpha-\alpha' = dd^c f$. It follows from Stokes' formula that
\[\langle S,\alpha\rangle - \langle S,\alpha'\rangle = \langle S,dd^c f\rangle = \langle dd^c S, f\rangle = 0.\]
So $S$ defines an element in the dual space of $H^{n-p,n-q}(X,\C)$. By Poincar\'e's duality, $S$ defines a class in $H^{p,q}(X,\C)$.

In our situation, the tangent currents to $T$ along $V$ are positive $dd^c$-closed currents on the compact K\"ahler manifold $\overline{E}$. By \cite[Theorem 2.14]{VA-generalizedLelong}, these currents belong to the same cohomology class in $H^{p,p}(\overline{E},\R)$, which we also denote by $\kappa^{V}(T)$. Similar to \eqref{eq Leray-Hirsch for closed case}, we can decompose this class uniquely as
\begin{equation}\label{eq Leray-Hirch for ddc closed case}\kappa^V(T) = \sum_{j= \max (0,k-p)}^{\min(k,n-p-1)} \pi^*( \kappa_j) \cdot h_{\overline{E}}^{p-(k-j)},
\end{equation}
where $\kappa_j$ is a cohomology class in $H^{k-j,k-j}(V,\C)$ (note that in this situation $V$ is also a compact K\"ahler manifold).

Let $S$ be a non-zero positive current of bi-degree $(p,p)$ on $\overline{E}$. Let $V_0$ be an open subset of $V$. Then the \textit{horizontal dimension} (h-dimension for short) of $S$ over $V_0$ is the largest number $j$ such that $S\wedge \pi^*(\omega_V^j) \neq 0$ on $\pi^{-1}(V_0)$, where $\omega_V$ is some K\"ahler form on $V$. By dimension reason, the h-dimension of $S$ belongs to $[\max (0,k-p),\min(k,n-p-1)]$.

Suppose that $S$ has compact support on $\overline{E}$. When $S$ is closed, due to \cite[Lemma 3.8]{Dinh_Sibony_density}, the h-dimension of $S$ over $V$ depends only on the cohomology class $\{S\}$ and equals the largest $j$ such that $\kappa_j \neq 0$ in the decomposition \eqref{eq Leray-Hirsch for closed case}. When $X$ is a compact K\"ahler manifold and $S$ is a general positive $dd^c$-closed current, a similar conclusion holds by \cite[Theorems 2.12 and 2.14]{VA-generalizedLelong}.

Suppose that $p<k$ and $S$ is a positive $dd^c$-closed current of bi-degree $(p,p)$ on $\overline{E}$ such that $S\wedge \pi^*(\omega_V^j) = 0$ on $\pi^{-1}(V_0)$ for every $j\geq 1$ (we call such $S$ has \textit{minimal $h$-dimension}).  Then, by \cite[Lemma 3.4]{Dinh_Sibony_density} (in the case where $S$ is closed) and \cite[Lemma 3.7]{DS-uniqueergodicity} (in the general case), there is a unique positive $dd^c$-closed current $S^h$ of bi-degree $(p,p)$ on $V_0$ such that $S=\pi^*(S^h)$ on $\pi^{-1}(V_0)$. We say $S^h$ is the \textit{shadow} of $S$ on $V_0$. We note that although \cite[Lemma 3.7]{DS-uniqueergodicity} is only proven for dimension $2$, the proof is similar in arbitrary dimensions, which relies on the fact that positive plurisubharmonic functions on $\P^{n-k}$ are constant functions.

Let $T_1$ be a positive $dd^c$-closed current of bi-degree $(p,p)$ on $X$, $T_2$ be a positive \textit{closed} current of bi-degree $(q,q)$ on $X$. Assume that the intersection of their supports is compact. Define $\mathbb{T}\colonequals T_1 \otimes T_2$. Then $\mathbb{T}$ is a positive $dd^c$-closed on $X^2$. Denote by $\Delta$ the diagonal of $X^2$. Then $\supp( \mathbb{T}) \cap \Delta$ is compact. Let $\pi_2:\mathbb{E}_2 \to \Delta$ denotes the normal bundle to $\Delta$ in $X^2$. A \textit{density current} of $T_1$ and $T_2$ is a tangent current of $\mathbb{T}$ along $\Delta$. Density currents of $T_1$ and $T_2$ are positive $dd^c$-closed current on $\mathbb{E}_2$ which can be extended trivially to $\overline{\mathbb{E}_2}$. The \textit{total density class} of $T_1,T_2$ is the total tangent class of $\mathbb{T}$ along $\Delta$, and the \textit{density h-dimension} is the h-dimension of density currents of $T_1$ and $T_2$ along $\Delta$.

\subsubsection{Technical lemmas}

We now recall some results on density currents. An important property of density currents is that they can be localized. Let $T_\infty$ be a tangent current to $T$ along $V$. Then there exists a sequence $(\lambda_m)_m$ such that $\lambda_m \to \infty$ and 
\[T_\infty = \lim_{m\to \infty} (A_{\lambda_m})_* \tau_* (T).\]
As we mentioned before, $T_\infty$ does not depend on the choice of 
the admissible map $\tau$. It only depends on the sequence $(\lambda_m)$. This property can be viewed through the following Proposition (see \cite[Proposition 4.4]{Dinh_Sibony_density} in the case where $T$ is closed and \cite[Theorem 15.2]{VANguyen-densitycurrent} in the general case).

\begin{proposition}\label{prop local global density}
Let $U$ be a local chart of $X$ such that $U\cap V \neq \varnothing$. Let $\tau'$ be a strongly admissible map along $U\cap V$ in $U$. Then, over $U$, we have
\[T_\infty = \lim_{m\to \infty} (A_{\lambda_m})_* \tau'_* (T).\]
\end{proposition}

In our later applications, we only need to consider the case where $T_1$ is a positive $dd^c$-closed current of bi-dimension $(1,1)$ and $T_2$ is a positive closed current of bi-degree $(1,1)$. We have the following criterion for $T_1$ and $T_2$ to admit minimal density h-dimension.

\begin{lemma}\label{lemma puts no mass implies positive cup product}
     Suppose that $X$ is a compact K\"ahler manifold. Suppose that $T_1$ puts no mass on the set $\{x:\nu(T_2,x) > 0 \}$. Then the density h-dimension of $T_1$ and $T_2$ is minimal. Moreover, if we let $T_\infty$ be a density current of $T_1$ and $T_2$, then $T_\infty = (\pi_2)^* (S_\infty)$ for some positive measure $S_\infty$ on $\Delta$ (which we can identify with $X$) such that $\{T_1\} \cdot \{T_2\} =  \|S_\infty\|_X$. In particular, we have $\{T_1\}\cdot \{T_2\}\geq 0$.  
\end{lemma}

We refer the reader to \cite[Theorem 3.6]{SuVu-volume-lelong-components-II} for a proof of this lemma in the case where $T_1$ is closed. The key point in the proof is to show that the density $h$-dimension is minimal. This is done in \cite{SuVu-volume-lelong-components-II} by using \cite[Proposition 5.7]{Dinh_Sibony_density}. In the general case, this argument can be replaced by \cite[Theorem 2.14]{VA-generalizedLelong}. See also \cite{NSV-cohom-nonpluripolar} for an application of this lemma in the study of positive cones in K\"ahler geometry.

The theory of density currents has been proven to cover previous notions on the intersection theory of currents. We will recall a simple fact here and refer the reader to \cite{DNV,Viet_Lucas,VietTuanLucas,Viet-density-nonpluripolar} for results in this direction.

\begin{proposition}\label{prop density = classical local case}
    Let $U$ be an open subset of $X$. Suppose that $T_2$ is smooth on $U$. Let $T_\infty$ be a density current of $T_1$ and $T_2$. Then, we have
    \[T_\infty = (\pi_2)^* (T_1\wedge T_2)\qquad \text{on}\qquad \pi_2^{-1}(U).\]
\end{proposition}

\begin{proof}
    By Proposition~\ref{prop local global density}, we can identify $U$ with an open subset of $\C^n$, and $\mathbb{E}_2$ with $U\times \C^n$ over $U$. Arguing similarly to \cite[Proposition 2.1]{Viet_Lucas} will conclude the proof.
\end{proof}

\subsection{Directed currents}

We finish this section by recalling the notion of directed currents, which play an important role in our study. Let $X$ be a complex manifold of dimension $n$. Let $V$ be a holomorphic rank $k$ sub-bundle of $T_X$. Such a pair $(X,V)$ is called a \textit{complex directed manifold}. We can also view $V$ as a family of complex subspaces $(V_p)_{p \in  X}$ such that the following conditions are satisfied:
\begin{itemize}
\item[(i)] $V_p$ is a complex subspace of dimension $k$ in $T_p X$ for every $p \in  X$;

\item[(ii)] There exists a finite open cover of $X$ by local charts $U_1, \ldots, U_m$ and holomorphic $1$-forms $\eta_{j,l}$ on $U_j$ for $1 \le j \le m, 1\leq l \leq n-k $ such that 
\begin{align}\label{eq def of directed manifold}
V_p = \{v \in T_p  X: \eta_{j,l}(v)=0 \text{ for all } j,l\}
\end{align}
for every $p \in U_j$.
\end{itemize}
We note that in \eqref{eq def of directed manifold}, at every point $p\in U_j$, $(\eta_{j,l}|_p)_l$ are linearly independent in $T_p^* X$.

Let $T$ be a positive current of bi-degree $(p,p)$ on $X$. We said that $T$ is \textit{directed by $V$} if 
\[T \wedge \eta_{j,l} = 0\qquad \text{for}\qquad1\leq j\leq m, 1\leq l \leq n-k.\]
Note that this definition does not depend on the choice of open local charts $U_j$ as well as holomorphic forms $\eta_{j,l}$. The following property is useful in the study of directed currents.

\begin{proposition}\label{prop dominated directed currents}
    Let $T$ be a positive current that is directed by $V$. Suppose that $T \geq S$ where $S$ is a positive current. Then $S$ is also directed by $V$.
\end{proposition}

\begin{proof}
    Since $i \eta_{j,l} \wedge \overline{\eta_{j,l}}$ is a positive form for all $j,l$, we have 
    \[0\leq S \wedge (i \eta_{j,l} \wedge \overline{\eta_{j,l}}) \leq T \wedge (i \eta_{j,l} \wedge \overline{\eta_{j,l}}) = 0. \]
    This implies that $S \wedge (i\eta_{j,l} \wedge \overline{\eta_{j,l}}) = 0$ for all $j,l$. Let $\gamma$ be a test form. By Cauchy-Schwarz inequality, we have
    \[|\langle S \wedge \eta_{j,l}, \gamma\rangle | \leq \langle S,i\eta_{j,l} \wedge \overline{\eta_{j,l}} \rangle ^{1/2} \cdot \langle S, i\gamma \wedge \overline{\gamma} \rangle^{1/2} = 0.\]
    We infer that $S\wedge \eta_{j,l} = 0$ for all $j,l$. Thus, $S$ is directed by $V$ as desired. 
\end{proof}

Directed currents are useful tools in the study of foliations and the Levi problem. We refer the reader to\cite{fornaess-sibony-survet-laminiantion,sibony-levi-problem} for more information.

\section{Demailly-Semple tower and liftable currents}\label{sec demailly-semple jet and liftable currents}

In this section, we will recall the construction of the Demailly-Semple tower (we refer the reader to \cite{demailly-97-algebraic-hyperbolicity} for a complete treatment) and introduce the notion of liftable currents, which contain both closed curves and currents associated with entire curves mentioned in Section~\ref{sec complex hyperbolicity}).

\subsection{Projectivized $1$-jet bundle}\label{subsec projectivized 1 jet}

Given a complex directed manifold $(X,V)$, where $X$ is a complex manifold $X$ of dimension $n$ and $V$ is a rank $r$ holomorphic subbundle of $T_X$. We note that in general $V$ is \textit{not} integrable. One can produce a new complex directed manifold $(\widetilde{X},\widetilde{V})$ that plays the role of a space of $1$-jets over $(X,V)$. We put $\widetilde{X} = \mathbb{P}(V)$ and let $\pi:\widetilde{X} \to X$ be the natural projection. Denote by $\pi_* : T_{\widetilde{X}} \to T_X$ the differential map. Let $\widetilde{V}$ be a holomorphic subbundle of $T_{\widetilde{X}}$ defined by
\begin{equation}\label{eq define tilde V}
\widetilde{V}_{(x,[v])} \colonequals \Big\{\xi \in T_{(x,[v])} \widetilde{X} : \pi_*(\xi) \in \C v \Big\} \text{ for }[v] \text{ such that } \C v \subset V_x \subset T_xX.
\end{equation}
By definition, we have $\dim \widetilde{X} = n + r -1$ and $\mathrm{rank}(\widetilde{V}) = r$. On $\widetilde{X} = \P (V)$, we have the tautological line bundle $\mathcal{O}_{\widetilde{X}}(-1) \subset \pi^* (V)$ such that $\mathcal{O}_{\widetilde{X}}(-1)|_{(x,[v])} = \C v$. By definition, the bundle $\widetilde{V}$ satisfies the short exact sequence:
\begin{equation}\label{eq short exact sequence for tilde V}
    0\longrightarrow T_{\widetilde{X}/X} \longrightarrow \widetilde{V} \xrightarrow{\ \pi_* \ }  \mathcal{O}_{\widetilde{X}}(-1) \longrightarrow 0,
\end{equation}
where $T_{\widetilde{X}/X}$ is the relative tangent bundle of $\pi:\widetilde{X} \to X$.

We have the following observation about positive currents directed by $\widetilde{V}$.

\begin{proposition}\label{prop k liftable is k-1 liftable}
    Let $T$ be a positive current directed by $\widetilde{V}$. Then $\pi_* (T)$ is directed by $V$.
\end{proposition}

\begin{proof}
    Let $\alpha_1,\ldots,\alpha_m$ be holomorphic $1$-forms that define $V$. Let $\xi \in T_{(x,[v])}\widetilde{X}$ such that $\pi_*(\xi) \in \C v$. Since $v\in V$, we have $\pi_*(\xi) \in V_x$. Thus, we have $\alpha_1(\pi_*(\xi))=\cdots =\alpha_m(\pi_*(\xi))  = 0$ by the definition of $\alpha_1,\ldots,\alpha_m$. This implies that $\pi^*(\alpha_1) (\xi) = \cdots = \pi^*(\alpha_m)(\xi) = 0$. Let $T$ be a positive current such that $T$ is directed by $\widetilde{V}$. Then, we have $T\wedge \pi^*(\alpha_j ) = 0$ for $j=1,\ldots,m$ and deduce that $\pi_*(T) \wedge \alpha_j = 0$ for $j=1,\ldots,m$. Hence, $\pi_*(T)$ is directed by $V$.
\end{proof}

We now consider the lifting of holomorphic curves to $\widetilde{X}$. Let $f:(\C,0) \to X$ be the germ of a non-constant holomorphic curve. Suppose that $f$ is a \textit{tangent trajectory} of $(X,V)$, i.e., $f'(t) \in V_{f(t)}$. One can define the lifting operator
\begin{equation}\label{eq define lifting of germ of hol curve}
\widetilde{f}:(\C,0) \to \widetilde{X}, \qquad t\mapsto (f(t),[f'(t)]).
\end{equation}
This is well-defined if $f'(t) \neq 0$. At the point $t_0$ where $f'(t_0)= 0$, we write $f'(t) = (t-t_0)^s u(t)$ where $s\in \Z^+$ and $u(t_0) \neq 0$, and define the tangent line at $t_0$ by $[u(t_0)]$. Then, near $t_0$, we have $\widetilde{f}(t)= (f(t),[u(t)])$. For simplicity, we still denote $[f'(t_0)] = [u(t_0)]$. By definition, we have $f'(t) \in \mathcal{O}_{\widetilde{X}}(-1)_{\widetilde{f}(t)} = \C u(t)$. Since $\pi \circ \widetilde{f} = f$, we have $\pi_* (\widetilde{f}'(t)) = f'(t) \in \C u(t)$. Therefore, $\widetilde{f}'(t) \in \widetilde{V}_{(f(t),[f'(t)])}$ and $\widetilde{f}$ is a tangent trajectory of $(\widetilde{X},\widetilde{V})$. We call $\widetilde{f}$ the \textit{canonical lifting} of $f$ to $\widetilde{X}$.

For any point $x_0\in X$, there exists a local coordinate $(z_1,\ldots,z_n)$ on a neighborhood $\Omega$ of $x_0$ such that for every $x\in \Omega$, the fibers $(V_z)_{z\in \Omega}$ can be defined by
\begin{equation}\label{eq define Vz locally}V_z= \Big\{ \xi : \xi = \sum_{1\leq j \leq n} \xi_j \frac{\partial}{\partial z_j} \text{ where } \xi_j = \sum_{1\leq k \leq r} a_{jk}(z)  \xi_k \text{ for } j > r\Big\},\end{equation}
where $(a_{jk}(z))$ is a holomorphic $(n-r) \times r$ matrix . Thus, a vector $\xi \in V_z$ is completely determined by its first $r$ components $(\xi_1,\ldots,\xi_r)$.
Therefore, for $1\leq j \leq r$, the affine chart $\xi_j\neq 0$ of $\P V|_\Omega$ can be described by the coordinate system
\begin{equation}\label{eq local system of PV}
    \Big(z_1,\ldots,z_n ; \frac{\xi_1}{\xi_j},\ldots,\frac{\xi_{j-1}}{\xi_j},\frac{\xi_{j+1}}{\xi_j},\ldots,\frac{\xi_r}{\xi_j}\Big)
\end{equation}
Let $f=(f_1,\ldots,f_n)$ be the components of $f$ in the coordinate $(z_1,\ldots,z_n)$. We assume that $f(0) = x \in \Omega$. Since $f'(t)\in V_{f(t)}$, we have
\begin{equation}\label{eq relation of components of f'(t)}
    f_j'(t) = \sum_{1\leq k \leq r} a_{jk}(f(t)) f_k'(t) \text{ for } j > r
\end{equation}
in a small neighborhood of $0$. We denote by $m = m(f,t_0)$ the \textit{multiplicity} of $f$ at the point $t_0$, which is the smallest $m\in \Z^+$ such that $f_j^{(m)}(t_0) \neq 0$ for some $j$ (it is clear that this number does not depend on the choice of local coordinate). By \eqref{eq relation of components of f'(t)}, we can assume that $1\leq j \leq r$. Suppose that $f_r^{(m)}(t_0)\neq 0$. Around $t_0$, we can write $f'(t) = (t-t_0)^{m-1} u(t)$ with $u_r(t_0)\neq 0$. Then, in the coordinate \eqref{eq local system of PV} of the affine chart $\xi_r\neq 0$ of $\P V|_\Omega$, the lifting $\widetilde{f}$ can be described by
\begin{equation}\label{eq describe tilde f on affine chart of PV}
    \widetilde{f} = \Big (f_1,\ldots,f_n; \frac{f_1'}{f_r'},\ldots, \frac{f_{r-1}'}{f_r'}\Big).
\end{equation}

In the last part of this subsection, we recall some curvature computations. Let $h$ be a smooth Hermitian metric on $V$. Denote by $D_h $ the associated Chern connection and by $\Theta_h(V) = \frac{i}{2\pi} D_h^2$ its Chern curvature tensor. The curvature tensor $\Theta_h(V)$ is a $(1,1)$-form with values in the Hermitian endomorphisms of $V$. Let $(e_1,\ldots,e_r)$ be a $\mathscr{C}^{\infty}$ orthonormal local frame of $E$ over a local chart $\Omega$ with the coordinate system $(z_1,\ldots,z_n)$. On $\Omega$, the Chern curvature tensor can be written 
\[\Theta_h(V) = \frac{i}{2\pi} \sum_{1\leq j,k\leq n} \sum_{1\leq \lambda,\mu \leq r} c_{jk\lambda\mu} dz_j \wedge d\overline{z}_k \otimes e_\lambda^* \otimes e_\mu,\]
for some coefficients $c_{jk\lambda\mu} \in \C$. This curvature tensor form can be viewed as a Hermitian form on $T_X \otimes V$ defined by
\begin{equation}\label{eq Griffiths curvature}\theta_{V,h}(\zeta \otimes v) \colonequals \sum_{1\leq j,k\leq n} \sum_{1\leq \lambda,\mu \leq r} c_{jk\lambda\mu} \zeta_j\overline{\zeta}_k v_\lambda \overline{v}_\mu.
\end{equation}
By \cite[Proposition (12.10)]{demailly-complex-book}, for every point $x_0 \in X$ and every local coordinate system $(z_1,\ldots,z_n)$ around $x_0$, there exists a holomorphic frame $(e_\lambda)_{1\leq \lambda \leq r}$ in a neighborhood of $x_0$ such that
\begin{equation}\label{eq local rep of chern curvature tensor}
    \langle e_\lambda,e_\mu \rangle_h = \delta_{\lambda \mu} - \sum_{1\leq j,k\leq n} c_{jk\lambda\mu} z_j \overline{z}_k + O(|z|^3)
\end{equation}
for every $1\leq \lambda,\mu \leq r$.

Since $\mathcal{O}_{\P V}(-1) \subset \pi^* V$, the metric $h$ induces a Hermitian metric $h_1$ on $\mathcal{O}_{\P V}(-1)$. Let $(x_0,[v_0]) \in \widetilde{X}$ such that $|v_0|_h = 1$. We choose a holomorphic local frame $(e_\lambda)_{1\leq \lambda \leq r}$ around $x_0$ such that $(e_\lambda)_\lambda$ satisfies \eqref{eq local rep of chern curvature tensor} and $[e_{r}] = [v_0]$. Consider the holomorphic local coordinates $(z_1,\ldots,z_n;\xi_1,\ldots,\xi_{r-1})$ on a neighborhood of $(x_0,[v_0])$ defined by
\begin{equation*}\label{eq coordinate around x_0,v_0 to compute curvature}
    (z_1,\ldots,z_n;\xi_1,\ldots,\xi_{r-1}) \mapsto (z, [\xi_1 e_1 +\cdots + \xi_{r-1} e_{r-1} + e_r]).
\end{equation*}
Then the map
\begin{equation*}\label{eq section of O(-1) over PV}
    \eta(z,\xi) \colonequals   \xi_1 e_1(z) + \cdots + \xi_{r-1} e_{r-1}(z) + e_r(z)
\end{equation*}
defines a holomorphic section of $\mathcal{O}_{\P V}(-1) \subset \pi^* V$. By the expansion \eqref{eq local rep of chern curvature tensor}, we have
\begin{equation}\label{eq compute norm of eta}
|\eta|_{h_1}^2 = |\eta|_h^2 = 1 + |\xi|^2 - \sum_{1\leq j,k\leq n}  c_{jkrr} z_j \overline{z}_k  + O((|z|+|\xi|)^3).
\end{equation}
By the Lelong-Poincar\'e formula, using \eqref{eq compute norm of eta}, we get
\begin{align}\label{eq compute curvature of O(-1)}\Theta_{h_1}(\mathcal{O}_{\P V}(-1))_{(x_0,[v_0])} &=- \frac{i}{2\pi} \partial \overline{\partial} \log |\eta|^{2} \\ \nonumber
&= \frac{i}{2\pi} \Big(\sum_{1\leq j,k\leq n} c_{jkrr}  d z_j \wedge d\overline{z}_k - \sum_{1\leq \lambda\leq r-1} d\xi_\lambda \wedge d\overline{\xi}_\lambda\Big).
\end{align}
Let $\xi \in T_{(x_0,[v_0])} \widetilde{X}$. We can decompose $\xi = \xi_{\mathrm{hor}} + \xi_{\mathrm{vert}}$, where $\xi_{\mathrm{vert}} \in T_{[v_0]}(\P(V|_{x_0}))$ is the vertical part and $\xi_{\mathrm{hor}}$ is the horizontal part. We also have $\pi_*(\xi) = \pi_*(\xi_{\mathrm{hor}})$ and $\pi_*(\xi_{\mathrm{vert}}) = 0$. By \eqref{eq Griffiths curvature} and \eqref{eq compute curvature of O(-1)}, for $\xi \in T_{(x_0,[v_0])} \widetilde{X}$, we have
\begin{equation}\label{eq formula for curvature on tilde X}
    \Theta_{h_1}(\mathcal{O}_{\widetilde{X}}(-1))(\xi,\xi) = \theta_{V,h}(\pi_* (\xi)\otimes v_0) - \|\xi_{\mathrm{vert}}\|_{\FS}^2,
\end{equation}
where $\|\cdot\|_{\FS}$ is the Fubini-Study metric induced by $h$ on the fiber $\P(V|_{x_0})$. Note that here $|v_0|_h = 1$.

Suppose that $X$ is a compact complex manifold. Fix a Hermitian metric $\omega$ on $X$. Then $\omega$ induces a metric on $V \subset T_X$, which we still denote by $\omega$. The metric $\omega$ induces a metric $h_1$ on $\mathcal{O}_{\widetilde{X}}(-1)$. Let $\mathcal{O}_{\widetilde{X}}(1)$ denote the dual of $\mathcal{O}_{\widetilde{X}}(-1)$. Since $X$ is compact, we can pick a sufficiently large positive constant $C$ such that
\begin{equation}\label{eq control Griffiths curvature by omega}
C\cdot\omega_{x}(\zeta,\zeta) \geq   \theta_{V,\omega}(\zeta \otimes v) \text{ for every } \zeta \in T_x X \text{ and } v\in V \text{ with }|v|_{\omega} = 1.
\end{equation}
By formula \eqref{eq formula for curvature on tilde X} and inequality \eqref{eq control Griffiths curvature by omega}, we get a Hermitian metric on $\widetilde{X}$ defined by
\begin{equation}\label{eq define Kahler form on lifting base case}
    \widetilde{\omega}\colonequals C\cdot \pi^*(\omega) + \ \Theta_{h_1^{-1}} (\mathcal{O}_{\widetilde{X}}(1)).
\end{equation}
Combining \eqref{eq formula for curvature on tilde X} and \eqref{eq define Kahler form on lifting base case}, for every $\xi \in T_{(x,[v])} \widetilde{ X}$, we have
\begin{equation}\label{eq compute metric on tilde X}
    \widetilde{\omega}(\xi,\xi) =C\cdot \omega (\pi_*(\xi),\pi_*(\xi)) +  \|\xi_{\mathrm{vert}}\|_\FS^2 -  \ \theta_{V,\omega}(\pi_* (\xi)\otimes v).
\end{equation}

Suppose that $X\subset \P^n$ is a projective manifold and $\iota: X\hookrightarrow \P^n$ is the embedding map. We denote by $\mathcal{O}_{X}(1)$ the line bundle $\iota^* \mathcal{O}_{\P^n}(1)$.  Since $\mathcal{O}_{\P(V)}(1)$ is very ample relative to
$\pi$, there exists a constant $l_0$ such that for all $l\geq l_0$, we have 
\begin{equation}\label{eq very ample line bundle 1-jet}\mathcal{O}_{\P(V)}(1) \otimes \pi^*(\mathcal{O}_X(l)) \text{ is a very ample line bundle on } \P(V)\end{equation} 
(see \cite[Section 1.7]{Lazarsfeld-Positivity-I}). Put $\widetilde{L}\colonequals \mathcal{O}_{\P(V)}(1) \otimes \pi^*(\mathcal{O}_X(l))$ for some $l\geq l_0$. We obtain the embedding $\varphi: \P(V) \hookrightarrow \P(H^0(\P(V), \widetilde{L} )^*)$ such that $\widetilde{L} =  \varphi^* \mathcal{O}_{\P(H^0(\P(V), \widetilde{L} )^*)}(1)$.

\subsection{Demailly-Semple tower}\label{subsec Demailly-Semple tower}

We now define the Demailly-Semple tower (in the absolute case) inductively by putting 
\begin{equation}\label{eq define X_k V_k}
(X_0,V_0) \colonequals (X,T_X), \qquad (X_k,V_k) = (\widetilde{X}_{k-1},\widetilde{V}_{k-1}) \text{ for } k \geq 1.
\end{equation}
We have $\dim X_k = n+k(n-1)$ and $\mathrm{rank} (V_k) = n$. Let $\pi_{k,k-1}: X_k \to X_{k-1}$ be the natural projection and $\pi_{k,j}: X_k \to X_j$ be the composition of projections for $j=0,\ldots,k-1$. Denote by $\mathcal{O}_{X_k}(-1)$ the tautological bundle of $X_k=\mathbb{P}(V_{k-1})$. We have the canonical injection $\mathcal{O}_{X_k}(-1) \hookrightarrow \pi_{k,k-1}^*(V_{k-1})$. For $k\geq 2$, we have the surjection $(\pi_{k-1,k-2})_* : V_{k-1} \to \mathcal{O}_{X_{k-1}} (-1)$ between vector bundles over $X_{k-1}$. Pulling back these vector bundles to $X_k$ by $\pi_{k,k-1}$ and composing with the map $\mathcal{O}_{X_{k}}(-1) \hookrightarrow \pi_{k,k-1}^*(V_{k-1})$ gives a bundle morphism 
\begin{equation}\label{eq exact sequence define Dk}
\mathcal{O}_{X_k}(-1) \hookrightarrow \pi_{k,k-1}^*(V_{k-1}) \xrightarrow{\pi_{k,k-1}^* \circ (\pi_{k-1,k-2})_*} \pi_{k,k-1}^* \mathcal{O}_{X_{k-1}}(-1).\end{equation}
This is a bundle morphism between two line bundles $\mathcal{O}_{X_k}(-1)$ and $\pi_{k,k-1}^* \mathcal{O}_{X_{k-1}}(-1)$ on $X_k$. We view this morphism as a section of the line bundle 
\[\mathcal{O}_{X_k}(-1) \otimes (\pi_{k,k-1}^* \mathcal{O}_{X_{k-1}}(-1))^{-1}\]
over $X_k$ and denote its zero divisor by $D_k$. Then we have 
\[D_k= \P(T_{X_{k-1}/X_{k-2}}) \subset \mathbb{P}(V_{k-1}) = X_k.\]
Here, $T_{X_{k-1}/X_{k-2}}$ is the kernel of $(\pi_{k-2,k-1})_*$. Denote by $u_k$ the cohomology class $c_1(\mathcal{O}_{X_k}(1))$. Formula~\eqref{eq exact sequence define Dk} gives us
\begin{equation}\label{eq formula for Xk and Xk-1}
    \mathcal{O}_{X_k}(1) = \pi_{k,k-1}^*\mathcal{O}_{X_{k-1}}(1) \otimes \mathcal{O}(D_k).
\end{equation}
Considering the first Chern class of line bundles in the formula \eqref{eq formula for Xk and Xk-1}, we obtain
\begin{equation}\label{eq formula for chern class Xk}
    u_k = \pi_{k,k-1}^* (u_{k-1} ) + \{D_k\}.
\end{equation}
We also define the \textit{singular $k$-jet locus} \[X_k^{[\sing]} : = \bigcup_{2\leq j\leq k}\pi_{k,j}^{-1}( D_j),\] which is a divisor of $X_k$ and the \textit{regular $k$-jet locus} 
\[X_k^{[\mathrm{reg}]} \colonequals \bigcap_{2\leq j \leq k} \pi_{k,j}^{-1} (X_j \setminus D_j) = X_k \setminus X_k^{[\sing]}.\]

 Let $f:(\C,0) \to X$ be a germ of a holomorphic curve. By applying \eqref{eq define lifting of germ of hol curve} inductively, we can define the \textit{$k$-lifting} $f_{[k]} : (\C,0) \to X_k$ by 
 \begin{equation}
     f_{[1]}(t) \colonequals \widetilde{f}(t), \qquad f_{[k]}(t) \colonequals \widetilde{f}_{[k-1]}(t) \text{ for } k\geq 2.
 \end{equation}
Consider a local coordinate $(z_1,\ldots,z_n)$ around $x = f(0)$. In this coordinate, we can write $f= (f_1,\ldots,f_n)$. As in \eqref{eq describe tilde f on affine chart of PV}, we can assume that $f_n'(t) \neq 0$ for $t$ is close enough $0$. Then, in the inhomogeneous coordinate, we can represent
\begin{equation}\label{eq coordinate represent for f_[1]}
f_{[1]}(t) = \Big(f_1,\ldots,f_n; \frac{f_1'}{f_n'},\ldots,\frac{f_{n-1}'}{f_n'}\Big).
\end{equation}
For the coordinate description of higher liftings, we refer the reader to \cite[Theorem 6.8]{demailly-97-algebraic-hyperbolicity}. When $k\geq 2$, we have the following relation between the $k$-lifting $f_{[k]}(t)$ and the divisor $D_k$ in the singular $k$-jet locus (see \cite[Proposition 5.11]{demailly-97-algebraic-hyperbolicity}).

\begin{proposition}\label{prop monotonicity of multiplicity}
     For all $t$, we have $m(f_{[k-2]},t)\geq m(f_{[k-1]},t)$. The inequality is strict if and only if $f_{[k]}(t) \in D_k$ or, equivalently $f'_{[k-1]}(t) \in T_{X_{k-1}/X_{k-2}}$.
\end{proposition}

Suppose that $X$ is a compact complex manifold. Fix a Hermitian metric $\omega$ on $X$. Then, by \eqref{eq define Kahler form on lifting base case}, for a sufficiently large positive constant $C_{[1]}$, we have
\[\omega_{[1]} \colonequals C_{[1]} \pi_{0,1}^*(\omega) + \Theta_{h_{[1]}}(\mathcal{O}_{X_1}(1))\]
is a Hermitian metric on $X_1$, where $h_{[1]}$ is the induced metric on $\mathcal{O}_{X_1}(1)$ by $\omega$. By induction, we can pick a sequence of sufficiently large positive constants $C_{[k]}$ such that
\begin{equation}\label{eq define Kahler form on lifting}
    \quad \omega_{[k]} \colonequals C_{[k]}\pi_{k,k-1}^*(\omega_{[k-1]}) +  \Theta_{h_{[k]}}(\mathcal{O}_{X_k}(1)) \text{ for } k \geq 2,
\end{equation}
are Hermitian metrics on $X_k$. Here, $h_{[k]}$ is the induced metric on $\mathcal{O}_{X_k}(1)$ by $\omega_{[k-1]}$. We note that if we let $X$ be a compact K\"ahler manifold and $\omega$ be a K\"ahler metric, then $\omega_{[k]}$ is a K\"ahler metric on $X_k$.

Let $f: \D_r \to X$ be a non-constant holomorphic disc. Assume that $r$ is small enough that $f_{[1]}(\D_r)$ belongs to some appropriate affine chart. For simplicity, we assume that $f_n(t) = t$. By \eqref{eq coordinate represent for f_[1]}, we have $f_{[1]}(t) = (f_1,\ldots,f_n; f_1',\ldots,f_{n-1}')$. By \eqref{eq compute metric on tilde X}, there exists a sufficiently large positive constant $C_{[1]}$ such that
\begin{equation*}\label{eq estimate for vector field compute by omega_1}
    \omega_{[1]}(\xi,\xi) \leq C_{[1]} \cdot \omega ((\pi_{1,0})_* \xi,(\pi_{1,0})_* \xi) + \|((\pi_{1,0})_*\xi)_{\mathrm{vert}}\|_{\FS}
\end{equation*}
for every $\xi \in T_{X_{1}}$. Recall also that in the inhomogeneous coordinates, the Fubini-Study metric is of the form
    \begin{equation*}\label{eq represent of fubini study inhomogeneous coordinate}\omega_{\FS} = \sum_{j,l} \frac{(1+|z|^2)\delta_{j\overline{l}} - z_j \overline{z}_l}{(1+|z|^2)^2} dz_j \wedge d\overline{z}_l.\end{equation*}
Then, we can control the area of the holomorphic disc $f_{[1]}(\D_r)$ with respect to the metric $\omega_{[1]}$ on $X_1$ as follows:
\begin{align}\label{eq estimate area of lifting compare with area of base}
    \mathrm{area}_{\omega_{[1]}} (f_{[1]}(\D_r)) &= \int_{\D_r} \|f_{[1]}'(t)\|_{\omega_{[1]}}^2 dt \wedge d\overline{t} \\ \nonumber &\leq  \int_{\D_r} \Big( C_{[1]}\cdot\|f'(t)\|_{\omega}^2 dt \wedge d\overline{t} + \|(f_{[1]}'(t))_{\mathrm{vert}}\|_{\FS} \Big) dt \wedge d\overline{t} 
    \\ \nonumber
    &= C_{[1]} \cdot \mathrm{area}_{\omega}(f(\D_r)) + \int_{\D_r} \sum_{j,l} \frac{(1+|f'|^2)\delta_{j\overline{l}} - f'_j\overline{f_{l}'}}{(1+|f'|^2)^2} f_j''\overline{f_l''} dt\wedge d\overline{t},
\end{align}
where $|f'|^2\colonequals  |f_1'|^2 + \cdots + |f_{n-1}'|^2$.

Suppose that $X \subset \P^n$ is a projective manifold. Let $L_0\colonequals \mathcal{O}_X(1)$ and define inductively 
\begin{equation}\label{eq construct ample line bundle on X_k}L_k\colonequals \mathcal{O}_{X_k} (1) \otimes (\pi_{k,k-1})^* L_{k-1}^{\otimes C_{[k]}}\end{equation} where $C_{[k]}$ are positive integers. By \eqref{eq very ample line bundle 1-jet}, we can choose $C_{[k]}$ sufficiently large such that $L_k$ is a very ample line bundle over $X_k$. If we only need the nefness, \cite[Lemma 3]{Diverio-Differential-Equations} says that we can control $C_{[k]} $.

\begin{proposition}\label{prop construct nef line bundle on X_k}
    For $k\geq 1$, the line bundle \[\mathcal{O}_{X_k}(1) \otimes \pi_{k,k-1}^* \mathcal{O}_{X_{k-1}}(2) \otimes \pi_{k,k-2}^* \mathcal{O}_{X_{k-2}}(6) \otimes \cdots \otimes \pi_{k,0}^{*} \mathcal{O}_{X}(2\cdot 3^{k-1})\] is a nef line bundle over $X_k$.
\end{proposition}

\subsection{Liftable currents}

In this last subsection, we introduce the notion of \textit{liftable currents}. These currents are positive closed (or $dd^c$-closed) currents of bi-dimension $(1,1)$ that can be lifted to the Demailly-Semple tower and are directed by vector bundles $V_k$. As mentioned in the introduction, this class of currents includes the situation of irreducible closed curves and Nevanlinna currents. We first discuss these special classes and their relation to the Demailly-Semple tower.

Given an irreducible closed curve $\mathscr{C}\subset X$. Let $\nu: \widehat{\mathscr{C}} \to \mathscr{C}$ be the normalization. We can define the $k$-lifting of $\mathscr{C}$ as the image of $\nu_{[k]}:\widehat{\mathscr{C}} \to X_k$ and denote it by $\mathscr{C}_{[k]}$. This is also an irreducible closed curve on $X_k$ and is tangent to $V_k$. Consider the current of integration $[\mathscr{C}_{[k]}]$ on $X_k$. This current is a positive closed current on $X_k$ that is directed by $V_k$ and satisfies
$(\pi_{k,0})_* ([\mathscr{C}_{[k]}]) = [\mathscr{C}]$ as currents, where $[\mathscr{C}]$ is the current of integration over $\mathscr{C}$ on $X$. Moreover, it is clear that $[\mathscr{C}_{[k]}]$ puts no mass on the singular $k$-jet locus $X_k^{[\sing]}$ (this property will play a key role in our theory!).

Given a non-constant entire curve $f:\C \to X$. Define the $k$-lifting $f_{[k]}:\C \to X_k$, which is a non-constant entire curve on $X_k$ that is tangent to $V_k$. Then, as in Section~\ref{sec complex hyperbolicity}, we can associate with $f$ and $f_{[k]}$ some positive closed (or $dd^c$-closed) currents $S$ and $S_{[k]}$. The current $S_{[k]}$ is directed by $V_k$ for all $k$. By using McQuillan's tautological inequality, we can choose $S_{[k]}$ and $S$ such that $(\pi_{k,0})_*(S_{[k]}) = S$. However, it is not clear to us whether $S_{[k]}$ puts no mass on the singular $k$-jet locus $X_k^{[\sing]}$ (this illustrates the transcendental nature of entire curve!). This situation will be discussed in detail in Section~\ref{sec quantitative Kobayashi conjecture}.

We now state the definition of liftable currents.

\begin{definition}\label{def lifable currents}
    Let $T$ be a positive current of bi-dimension $(1,1)$ on $X$.  Let $k\in \Z^+$. We say that $T$ is ``$k$-liftable" if there exists a positive closed current $T_{[k]}$ of bi-dimension $(1,1)$ on $X_k$ such that
\begin{equation}\label{eq define k lifting for liftable}
    T_{[k]} \text{ is directed by }V_k \text{ and } (\pi_{k,0})_* (T_{[k]}) = T.
\end{equation}
Such $T_{[k]}$ is called an ``almost $k$-lifting'' of $T$. If $T$ is $k$-liftable for every $k \in \N$, we say that $T$ is ``$\infty$-liftable''. If there exists a positive $dd^c$-closed current $T_{[k]}$ on $X_k$ such that $T_{[k]}$ satisfies \eqref{eq define k lifting for liftable}, we say that $T$ is ``weakly $k$-liftable". Such $T_{[k]}$ is called an ``almost weak $k$-lifting" of $T$.
If $T$ is weakly $k$-liftable for every $k\in \Z^+$, then we call $T$ a ``weakly $\infty$-liftable'' current.
\end{definition}

By definition, if $T$ is $k$-liftable, then $T$ is weakly $k$-liftable. Using Proposition~\ref{prop k liftable is k-1 liftable}, we see that if $T$ is $k$-liftable, then $T$ is $l$-liftable for every $l<k$. Therefore, the cone of $k$-liftable (resp. weakly $k$-liftable) currents is decreasing in $k$. As discussed above, closed curves are $\infty$-liftable currents, and Nevanlinna currents associated with entire curves are weakly $\infty$-liftable currents. The advantage of working with directed currents is stronger flexibility in comparison to Nevanlinna or Ahlfors currents. On projective manifolds, there always exist $\infty$-liftable currents by considering algebraic curves. For general compact manifolds, the following theorem shows the existence of weakly $1$-liftable currents.

\begin{proposition}\label{lemma existence liftable current induction level}
    There always exist non-zero weakly $1$-liftable currents on compact complex manifolds.
\end{proposition}

\begin{proof}
     Let $\omega$ be a Hermitian metric on $X$. Then we can define a Hermitian metric $\omega_{[1]}$ on $X_1$ as in \eqref{eq define Kahler form on lifting base case}. For $A > 0$, we define
    $\mathcal{C}_A$ to be the set of all positive currents $T$ of bi-dimension $(1,1)$ on $X_1$ such that $T$ satisfies the following conditions:
    \[T \text{ is directed by } V_1, \qquad \langle(\pi_{1,0})_*(T),\omega\rangle =  1, \qquad \langle T,\omega_{[1]} \rangle \leq A.\] 
    It is clear that $\mathcal{C}_A$ is a convex compact set. 
    We also define
    \[\mathcal{V} \colonequals \{ dd^c u : u \text{ is a smooth function on } X_1\}.\]
    We claim that there always exists $A$ (depending only on $X$) sufficient large such that $\mathcal{C}_A \cap \mathcal{V}^\bot \neq \varnothing$, where $\mathcal{V}^\bot$ is the space of $dd^c$-closed currents of bi-dimension $(1,1)$ on $X_1$.

    Suppose that $\mathcal{C}_A \cap \mathcal{V}^\bot = \varnothing$. Then, by the Hahn-Banach theorem, there exist a smooth $(1,1)$-form $\alpha$ and $\delta>0$ such that 
    \begin{equation}\label{eq hahn banach separate}\inf_{T\in \mathcal{C}_A} \langle T,\alpha \rangle \geq \delta > \sup_{S\in \mathcal{V}^\bot} \langle S,\alpha\rangle.\end{equation}
    If $\langle S,\alpha\rangle \neq 0$ for some $S\in \mathcal{V}^\bot$, we can pick $m$ good enough such that $\langle m S,\alpha\rangle > \delta$. Since $\mathcal{V}^\bot$ is a vector space, $mS\in \mathcal{V}^\bot$, this contradicts \eqref{eq hahn banach separate}. Thus, we have $\alpha \in (\mathcal{V}^{\bot})^\bot = \mathcal{V}$, and we can pick a  smooth function $u$ on $X_1$ such that $\alpha= dd^c u$.

    Fix a point $a\in X_1$ where $u$ is maximal. Then $dd^c u$ vanishes at $a$. We choose a special local curve $f:\D \to X$ passing through $x$ such that the $1$-lifting $f_{[1]}$ passes through $a$. Let $x= \pi_{1,0}(a)$. Fix a local coordinate chart around $x$ that we identify to the unit ball in $\C^n$. In this chart, we can write $f=(f_1,\ldots,f_n):\D\to X$ where $f_1,\ldots,f_n$ are holomorphic functions on $\D$. We can assume that $a$ lies inside the affine chart on $X_1$ which was described in \eqref{eq coordinate represent for f_[1]}. In this coordinates, we can represent
    \[a = (0,\ldots,0; (a_1^{(1)},\ldots,a_{n-1}^{(1)})) \]
    We choose $f_j(t) \colonequals  a_j^{(1)}t$ for $1\leq j \leq n-1$ and $f_n(t) = t$. This gives $f_{[1]}(0) = a$ as desired.

    Consider the positive current
    \[T_{[1]}^{(\epsilon)}\colonequals  \frac{[f_{[1]}(\D_\epsilon)]}{\|[f(\D_\epsilon)]\|_{\omega}},\]
    where $\epsilon > 0 $ is a small number. We want to compute $\langle T_{[1]}^{\epsilon},\omega_{[1]}\rangle$. By \eqref{eq estimate area of lifting compare with area of base}, we have
    \begin{equation*}
    \langle T_{[1]}^{(\epsilon)},\omega_{[1]}\rangle \leq C_{[1]} + \frac{1}{\mathrm{area}_{\omega} (f(\D_\epsilon))} \int_{\D_\epsilon} \sum_{j,l} \frac{(1+|f'|^2)\delta_{j\overline{l}} - f'_j\overline{f_{l}'}}{(1+|f'|^2)^2} f_j''\overline{f_l''} dt\wedge d\overline{t}  
    = C_{[1]}.\nonumber\end{equation*}
    Hence, by choosing $A = C_{[1]}$, we have $T_{[1]}^{(\epsilon)} \in \mathcal{C}_A$ for every $\epsilon > 0$. Letting $\epsilon \to 0^+$, we get $\langle T_{[1]}^{(\epsilon)},dd^c u\rangle \to 0$ since $dd^c u$ vanishes at $a$. This is a contradiction and we obtain a weakly $1$-liftable current $T$ as desired. 
\end{proof}

\begin{remark}
We note that our construction still holds if we consider the general case of directed manifolds $(X,V)$. We refer the reader to \cite[Section 2]{sibony-pfaff} for an existence theorem for $dd^c$-closed currents that are directed by $V$.
\end{remark}

\section{Euler characteristic of currents}\label{sec pluri Euler char}

In this section, we will introduce two related notions of Euler characteristic for liftable currents: the geometric one and the cohomological one. We then study basic properties of these notions and compare them in the case where the base manifold is projective. For the first two subsections, the manifold $X$ is a compact K\"ahler manifold of dimension $n$.

\subsection{Basic definitions}

 We start by recalling some observations on the geometric Euler characteristic of irreducible closed curves inside $X$.

\begin{proposition}\label{prop increasing of chi_k to chi_infty for curve}
    Let $\mathscr{C}\subset X$ be a closed irreducible curve and $\nu: \widehat{\mathscr{C}} \to \mathscr{C}$ be a normalization. Let $\nu_{[k]} : \widehat{\mathscr{C}} \to X_k$ be the $k$-lifting of $\nu$ and $\mathscr{C}_{[k]}\colonequals \nu_{[k]}(\widehat{\mathscr{C}})$ the $k$-lifting of $\mathscr{C}$ to $X_k$. Then we have
    \[-\chi(\widehat{\mathscr{C}}) = \sum_{t\in \widehat{\mathscr{C}}} (m(\nu_{[k-1]},t)-1) + \{\mathscr{C}_{[k]}\} \cdot u_k,\]
    where $\{\mathscr{C}_{[k]}\}$ denotes the cohomology class of $\mathscr{C}_{[k]}$. Moreover, $[\mathscr{C}_{[k]}] \cdot u_k$ is a non-decreasing sequence in $k$ of integers such that for $k$ large enough, we have
    \begin{equation}\label{eq geometric genus as limit}
        \{\mathscr{C}_{[k]}\} \cdot u_k = -\chi(\widehat{\mathscr{C}}).
    \end{equation}
\end{proposition}

\begin{proof} The first statement is embedded in the proof of \cite[Theorem 8.1]{demailly-97-algebraic-hyperbolicity}. We recall the proof here for the reader's convenience. Since $\nu_{[k]} = (\nu_{[k-1]},\nu_{[k-1]}'):\widehat{\mathscr{C}} \to X_k$ is the lifting of $\nu_{[k-1]}$, there is a canonical map $\nu_{[k-1]}': T_{\widehat{\mathscr{C}}} \to \nu_{[k]}^*(\mathcal{O}_{X_k}(-1))$. Since $\nu'_{[k-1]}$ vanishes at order $m(\nu_{[k-1]},t)-1$ at $t$ for $t\in \widehat{\mathscr{C}}$, we have
\[T_{\widehat{\mathscr{C}}} = \nu_{[k]}^*(\mathcal{O}_{X_k}(-1)) \otimes \mathcal{O}_{\widehat{\mathscr{C}}}\Big(-\sum_{t\in \widehat{\mathscr{C}}}(m(\nu_{[k-1]},t)-1) t\Big).\]
By taking a Hermitian metric on $\mathcal{O}_{X_k}(-1)$ and using the Gauss-Bonnet formula, we get
\begin{align*}\chi(\mathscr{C}) &= \int_{\widehat{\mathscr{C}}} \nu^*_{[k]} (c_1(\mathcal{O}_{X_k}(-1)) - \sum_{t\in \widehat{\mathscr{C}}}(m(\nu_{[k-1]},t)-1) \\ &= -\{\mathscr{C}_{[k]}\} \cdot u_k- \sum_{t\in \widehat{\mathscr{C}}}(m(\nu_{[k-1]},t)-1). \end{align*}
The first statement follows. For the second statement, we have 
 $\{\mathscr{C}_{[k]}\}\cdot u_k$ is a non-decreasing sequence by Proposition~\ref{prop monotonicity of multiplicity}. Moreover, for $k$ large enough, $\nu_{[k]}$ has no multiplicity at every point of $\widehat{\mathscr{C}}$, and we infer the formula~\eqref{eq geometric genus as limit}.
\end{proof}

Given a closed irreducible curve $\mathscr{C} \subset X$, we define the \textit{$k$-Euler characteristic of $\mathscr{C}$} by \[\chi_k(\mathscr{C})\colonequals -\{\mathscr{C}_{[k]}\} \cdot u_k\] and the \textit{$\infty$-Euler characteristic of $\mathscr{C}$} by \[\chi_{\infty}(\mathscr{C})\colonequals \lim\limits_{k\to \infty} \chi_k(\mathscr{C}).\] Then, by formula \eqref{eq geometric genus as limit}, we see that $\chi_k(\mathscr{C})$ decreases to $\chi_\infty(\mathscr{C})$, which equals the geometric Euler characteristic of $\mathscr{C}$.

Observe that if $\mathscr{C}_{[k]}$ is the $k$-lifting of $\mathscr{C}$, then $\mathscr{C}_{[k]}$ intersects properly with $\pi_{k,j}^*(D_j)$ for $j=2,\ldots,k$. Motivated by this observation and the definition of the regular locus in $X_k$, we call an almost $k$-lifting $T_{[k]}$ of $T$ a \textit{geometric $k$-lifting} if $T_{[k]}$ puts no mass on the singular locus $X_k^{[\sing]}$. Since $X$ is a compact K\"ahler manifold, we can make sense of the cohomology class of $dd^c$-closed currents (see Section~\ref{sec background currents}). Let $T$ be a weakly $k$-liftable current. We also call an almost weak $k$-lifting $T_{[k]}$ of $T$ a \textit{weak geometric $k$-lifting} if $T_{[k]}$ puts no mass on the singular locus $X_{k}^{[\sing]}$.

The notion of geometric $k$-lifting has some limitations. We can see that weak limits of geometric $k$-liftings may not be geometric $k$-lifting, this will lead to several problems (see Remark~\ref{remark open of hyp} and Remark~\ref{remark deform current}). Another difficulty will arises when we deal with the Nevanlinna currents later. Observe that by  Lemma~\ref{lemma puts no mass implies positive cup product}, the cohomology class $\{T_{[k]} \}$ must satisfies the following inequality:
\begin{equation}\label{eq define regular k lift}
    \{T_{[k]}\} \cdot \pi_{k,j}^* \{D_j\} \geq 0\qquad\text{ for }j=2,\ldots,k.
\end{equation}
This suggests us to introduce a refined version of geometric $k$-lifting as follows. An almost $k$-lifting $T_{[k]}$ of $T$ is called \textit{cohomological $k$-lifting} if the cohomology class $\{T_{[k]}\}$ satisfies the inequality~\eqref{eq define regular k lift}. Similarly, we have the notion of \textit{weak cohomological $k$-lifting}. By definition, a geometric $k$-lifting is a cohomological $k$-lifting. When $k=1$, the geometric and cohomological notions coincide. We will call this notion a $1$-lifting (and weak $1$-lifting in the $dd^c$-closed case).

The next proposition ensures the existence of a geometric (therefore cohomological) $k$-lifting for $k$-liftable currents (similar conclusion holds for weakly $k$-liftable currents).

\begin{proposition}\label{prop vanish when restrict to singular part}
    Let $k\in \Z^+, k\geq 2$. Let $S$ be a positive (closed or $dd^c$-closed) current of bi-dimension $(1,1)$ on $X_k$ such that $S$ is directed by $V_k$. Then we have \[(\pi_{k-1,k-2}\circ\pi_{k,k-1})_*(\mathbf{1}_{D_k} S) = 0 \qquad \text{and} \qquad (\pi_{k,0})_*(\mathbf{1}_{X_k^{[\sing]}}S) = 0.\]
    In particular, if $T$ is a $k$-liftable (resp. weakly $k$-liftable) current, then there always exists a geometric $k$-lifting (resp. weak geometric $k$-lifting) of $T$.
\end{proposition}

\begin{proof}
    We consider the case of closed currents, the $dd^c$-closed case can be treated similarly using Lemma~\ref{lemma support lemma for ddc closed current} instead of Lemma~\ref{lemma support lemma for closed current}.
    Recall that $D_k = \P(T_{X_{k-1}/X_{k-2}})$. Put $V' = T_{X_{k-1}/X_{k-2}}$, which is a holomorphic subbundle of $V_{k-1}$, and thus a holomorphic subbundle of $T_{X_{k-1}}$. Then $(X_{k-1},V')$ is also a directed manifold. We see that the couple $(D_k, V_k \cap T_{D_k})$ is the directed manifold obtained by applying formula~\eqref{eq define tilde V} for $(X_{k-1},V')$. By Lemma~\ref{lemma support lemma for closed current}, $\mathbf{1}_{D_k}T$ is a positive closed current on $D_k$. Since $T$ is directed by $V_k$, we have $\mathbf{1}_{D_k}(T)$ directed by $V_k \cap T_{D_k}$.  Therefore, by Proposition~\ref{prop k liftable is k-1 liftable}, $(\pi_{k,k-1})_*(\mathbf{1}_{D_k} T)$ is directed by $V'$. Since $V' = \ker  (\pi_{k-1,k-2})_*$, the push-forward by $\pi_{k-1,k-2}$ of currents that are directed by $V'$ is zero, as they only have vertical components. The second statement follows as $(\pi_{k,j})_*(T)$ is directed by $V_{j}$ by Proposition~\ref{prop k liftable is k-1 liftable}.

    Now, we let $T$ be a $k$-liftable current and pick an almost $k$-lifting $T_{[k]}$ of $T$. By Lemma~\ref{lemma support lemma for closed current}, $\mathbf{1}_{X_k^{[\mathrm{reg}]}} T_{[k]}$ is a positive closed current. It puts no mass on $X_{k}^{ [\sing] }$ and is also directed by $V_k$ by Proposition~\ref{prop dominated directed currents}. Therefore, it is a geometric $k$-lifting of $T$.
\end{proof}

We now define the \textit{geometric Euler characteristic for liftable currents}. Given $k \in \Z^+$. Let $T$ be a $k$-liftable current. We define
\[\chi_k(T)\colonequals \sup \Big\{ -\{T_{[k]} \} \cdot u_k : T_{[k]} \text{ is a geometric } k\text{-lifting of }T \Big \},\]
and call it \textit{geometric $k$-Euler characteristic of $T$}. Similarly, we can define 
\[\chi_{k}^*(T)\colonequals \sup \Big\{ -\{T_{[k]} \} \cdot u_k : T_{[k]} \text{ is a weak geometric } k\text{-lifting of }T\Big\}.\]
 We call this number the \textit{weak geometric $k$-Euler characteristic} of $T$. When $T$ is an $\infty$-liftable current, we define $\chi_\infty(T) \colonequals \inf_{k} \chi_k(T)$. Similarly, we define $\chi_\infty^*(T) \colonequals \inf_{k} \chi_k^*(T)$.

We now prove that these notions recover the $k$-Euler characteristic in the case of irreducible closed curves.

\begin{theorem}\label{theorem coincide with k-Euler characteristic}
    Let $k\in \Z^+ \cup \{\infty\}$. Let $\mathscr{C}$ be an irreducible closed curve inside $X$. Then $\chi_k(\mathscr{C}) = {\chi}_k([\mathscr{C}]) = {\chi}^*_k([\mathscr{C}])$.
\end{theorem}

\begin{proof}
    We first consider the case where $k=1$. Let $T_{[1]}$ be a $1$-lifting of $[\mathscr{C}]$. Since $(\pi_{1,0})_*(T_{[1]}) = [\mathscr{C}]$, we have $\supp (T_{[1]}) \subset \pi_{1,0}^{-1}(\mathscr{C})$. Note that $\pi_{1,0}^{-1}(\mathscr{C})$ is an analytic set of pure dimension $n$ in $X_1$ and $\mathscr{C}_{[1]} \subset \pi_{1,0}^{-1}(\mathscr{C})$. Let $x$ be a regular point of $\pi_{1,0}^{-1}(\mathscr{C})$. Then $T_x(\pi_{1,0}^{-1}(\mathscr{C})) \cap (V_1)_x $ equals $(V_1)_x$ if $x \in \mathscr{C}_{[1]}$. Otherwise, it equals $\{\xi \in T_x (X_1): (\pi_{1,0})_* \xi = 0\}$, which is the vertical part of $T_x(X_1)$. Consider the current $\mathbf{1}_{\pi_{1,0}^{-1}(\mathscr{C})\setminus \mathscr{C}_{[1]}}T_{[1]}$, which is a positive closed current on $\pi_{1,0}^{-1}(\mathscr{C})$ by Lemma~\ref{lemma support lemma for closed current}. This current is directed by the distribution 
    \[\bigcup_{x\in \pi_{1,0}^{-1}(\mathscr{C})} \{\xi \in T_x (X_1): (\pi_{1,0})_* \xi = 0\}\]
    on $\pi_{1,0}^{-1}(\mathscr{C}_\mathrm{reg})$.
    Therefore, since $\mathscr{C}_{\sing}$ is a finite set, we conclude that 
    \[(\pi_{1,0})_*(\mathbf{1}_{\pi_{1,0}^{-1}(\mathscr{C})\setminus \mathscr{C}_{[1]}}T_{[1]}) = 0.\] This implies that 
    \[\{\mathbf{1}_{\pi_{1,0}^{-1}(\mathscr{C})\setminus \mathscr{C}_{[1]}}T_{[1]}\} \cdot u_{1} =  \{\mathbf{1}_{\pi_{1,0}^{-1}(\mathscr{C})\setminus \mathscr{C}_{[1]}}T_{[1]}\} \cdot \{\omega_{[1]}\} \geq 0.\] 
    Moreover, we infer that $\mathbf{1}_{\mathscr{C}_{[1]}} T_{[1]}$ is also a  $1$-lifting of $\mathscr{C}$. By Lemma~\ref{lemma support lemma for closed current}, we have $\mathbf{1}_{\mathscr{C}_{[1]}} T_{[1]} = [\mathscr{C}_{[1]}]$. Hence, we deduce that 
    \begin{align*}\{T_{[1]}\} \cdot u_{1} &= \{\mathbf{1}_{\pi_{1,0}^{-1}(\mathscr{C})\setminus \mathscr{C}_{[1]}}T_{[1]}\} \cdot u_{1} + \{\mathbf{1}_{ \mathscr{C}_{[1]}}T_{[1]}\} \cdot u_{1}\\ &\geq \{ \mathscr{C}_{[1]}\} \cdot u_{1} =-\chi_1(C).\end{align*} 
    Since $[\mathscr{C}_{[1]}]$ is also a $1$-lifting of $\mathscr{C}$, we get ${\chi}_1([\mathscr{C}]) = \chi_1(\mathscr{C})$.

    We prove by induction that every geometric $k$-lifting $T_{[k]}$ of $\mathscr{C}$ is of the form $T_{[k]} = [C_{[k]}] + S_{[k]}$, where $S_{[k]}$ is a positive closed current on $X_k$ that satisfies 
    \[\{S_{[k]}\} \cdot u_k \geq 0,\qquad (\pi_{k,0})_*(S_{[k]}) = 0.\]
    To do this, we assume the statement holds for $k-1$ and write 
    \[T_{[k]} = \mathbf{1}_{\pi_{k,k-1}^{-1}(\mathscr{C}_{[k-1]})} T_{[k]} + S'_{[k]}.\]
    Then we have $(\pi_{k,k-1})_*(S'_{[k]}) = S_{[k-1]}$, where $S_{[k-1]}$ is defined by the induction hypothesis. Since $T_{[k]}$ is a geometric $k$-lifting of $[\mathscr{C}]$, $S_{[k]}'$ puts no mass on $D_k$ (this is the advantage point of the geometric condition!). Then, by the induction hypothesis and Lemma~\ref{lemma puts no mass implies positive cup product}, we have
    \[\{S'_{[k]}\} \cdot u_k = \{S_{[k-1]}\} \cdot u_{k-1} + \{S_{[k']}\} \cdot \{D_k\} \geq 0.\]
    For the remaining part $\mathbf{1}_{\pi_{k,k-1}^{-1}(\mathscr{C}_{[k-1]})} T_{[k]}$, we argue similarly to the case where $k=1$ and infer that $\mathbf{1}_{\pi_{k,k-1}^{-1}(\mathscr{C}_{[k-1]})} T_{[k]} = [\mathscr{C}_{[k]}] + S^{''}_{[k]}$,
    where $S^{''}_{[k]}$ satisfies $\{S^{''}_{[k]}\} \cdot u_k \geq 0$. Now we can choose $S_{[k]} = S'_{[k]} + S_{[k]}^{''}$ and finish the induction process. Similar to the case where $k=1$, this claim implies that ${\chi}_k([\mathscr{C}]) = \chi_k(\mathscr{C})$ as desired. The proof for ${\chi}_k^*([\mathscr{C}])$ is similar, but we use Lemma~\ref{lemma support lemma for ddc closed current} instead of Lemma~\ref{lemma support lemma for closed current}.
\end{proof}

By replacing geometric $k$-lifting by cohomological $k$-lifting, we define 
\[\overline{\chi}_k(T)\colonequals \sup \Big\{ -\{T_{[k]} \} \cdot u_k : T_{[k]} \text{ is a cohomological } k\text{-lifting of }T \Big \},\]
and call it \textit{cohomological $k$-Euler characteristic of $T$}. Similarly, we can define 
\[\overline{\chi}_{k}^*(T)\colonequals \sup \Big\{ -\{T_{[k]} \} \cdot u_k : T_{[k]} \text{ is a weak cohomological } k\text{-lifting of }T\Big\}.\]
 We call this number the \textit{weak cohomological $k$-Euler characteristic} of $T$. When $T$ is an $\infty$-liftable current, we define $\overline{\chi}_\infty(T) \colonequals \inf_{k} \overline{\chi}_k(T)$. Similarly, we define $\overline{\chi}_\infty^*(T) \colonequals \inf_{k} \overline{\chi}_k^*(T)$.

\begin{remark}\label{remark about definition of Euler char}
    By definition, we see that $\chi_1(T) = \overline{\chi}_1(T)$ and $\chi_k(T) \leq \overline{\chi}_k(T)$ for $k\geq 2$. In general, for $k\geq 2$, equality cannot be attained. Consider a smooth rational curve $\mathscr{C}$ on a fiber of $\P (T_{X_1/X})$ and let $\mathscr{C}_{[1]}$ be the lifting of $\mathscr{C}$ to $X_2$. Let $T_{[2]}$ be a geometric $2$-lifting of a $2$-liftable current $T$. Since $\mathscr{C}_{[1]} \subset D_2$, $T_{[2]} + m  [\mathscr{C}_{[1]}]$ is an almost $k$-lifting of $T$ by Proposition~\ref{prop vanish when restrict to singular part}. If $\{T_{[2]}\} \cdot \{D_2\} > 0$, we can choose a small number $m$ such that $T_{[2]} + m [\mathscr{C}_{[1]}]$ is also a cohomological $k$-lifting (in this situation, $\mathscr{C}_{[1]} \subset D_2$). Since $\mathscr{C}$ is a smooth rational curve, we have $\{T_{[2]}+ m [\mathscr{C}_{[1]}]\} \cdot u_2 = \{T_{[2]}\} \cdot u_2 -2m$. This leads to a difference between these two notions when $k\geq 2$. In general, it is not clear how to control this difference. However, fortunately, when $T$ is a liftable current in a projective manifold, we can say something about this problem (see Theorem~\ref{theorem comparison in proj mfd}).
\end{remark}

\subsection{Elementary properties}

We establish some basic properties of $\chi_k$ and $\overline{\chi}_k$. The corresponding statements for $\chi_k^*$ and $\overline{\chi}_k^*$ follow by the same arguments, replacing Lemma~\ref{lemma support lemma for ddc closed current} with Lemma~\ref{lemma support lemma for closed current}.

We first obtain a version of the inequality in Proposition~\ref{prop increasing of chi_k to chi_infty for curve} for liftable currents. This property motivates the definition of cohomological $k$-lifitng.

\begin{proposition}\label{prop increasing of chi_k}
    Let $k\in \Z^+, k\geq 2$. Then $\chi_{k}(T) \geq \chi_{k-1}(T)$ and $\overline{\chi}_{k}(T) \geq \overline{\chi}_{k-1}(T)$ for every $k$-liftable current $T$. In particular, we have $\chi_\infty(T) = \lim\limits_{k\to \infty} \chi_k(T)$ and $\overline{\chi}_\infty(T) = \lim\limits_{k\to \infty} \overline{\chi}_k(T)$ for every $\infty$-liftable current.
\end{proposition}

\begin{proof}
    Let $T_{[k]}$ be a geometric $k$-lifting of $T$. Since $T_{[k]}$ puts no mass on $X_{k}^{[\sing]}$, we have $(\pi_{k,k-1})_* (T_{[k]})$ puts no mass on $X_{k-1}^{[\sing]}$. Thus $(\pi_{k,k-1})_* (T_{[k]})$ is a geometric $(k-1)$-lifting of $T$. Moreover, by Lemma~\ref{lemma puts no mass implies positive cup product}, we have $\{T_{[k]}\} \cdot \{D_k\} \geq 0$. Then, using formula~\eqref{eq formula for chern class Xk}, we obtain $\{T_{[k]}\} \cdot u_k \geq \{(\pi_{k,k-1})_* (T_{[k]})\} \cdot u_{k-1}$. Thus, we get $\chi_k(T) \geq \chi_{k-1}(T)$.

    Let $T_{[k]}$ is a cohomological $k$-lifting. Since $T_{[k]}$ satisfies \eqref{eq define regular k lift}, we have $\{T_{[k]}\} \cdot \{D_k\} \geq 0$. Moreover, we have $\{(\pi_{k,k-1})_* (T_{[k]})\} \cdot \pi_{k-1,j}^* \{D_j\} = \{T_{[k]}\} \cdot \pi_{k,j}^* \{D_j\} \geq 0$. Thus, $\{(\pi_{k,k-1})_* (T_{[k]})\}$ satisfies \eqref{eq define regular k lift} and therefore is a cohomological $(k-1)$-lifting of $T$. We deduce that $\overline{\chi}_k(T) \geq \overline{\chi}_{k-1}(T)$ as desired.
\end{proof}

\begin{proposition}\label{prop finiteness of Euler characteristic}
    Let $k\in \Z^+\cup \{\infty\}$ and $l\leq k, l\in \Z^+$. Then ${\chi}_l(T)$ and $\overline{\chi}_l(T)$ are finite for every $k$-liftable current $T$. 
\end{proposition}

\begin{proof}
    By definition, we see that $\chi_l(T),\overline{\chi}_l(T) > -\infty$. We also have $\chi_l(T) \leq \overline{\chi}_l(T)$. Therefore, we only need to show that $\overline{\chi}_{l}(T) < \infty$. By Proposition~\ref{prop increasing of chi_k}, we only need to prove $\overline{\chi}_1(T) < \infty$. Suppose that there exists a sequence $(T_{[1]}^{(m)})_{m\in \Z^+}$ of $1$-liftings of $T$ such that $\{T_{[1]}^{(m)}\} \cdot u_{1} \to -\infty$ as $m \to \infty$.
    Let $\omega$ be a K\"ahler form on $X$ and $\omega_{[1]}$ be the K\"ahler form on $X_1$ in the construction~\eqref{eq define Kahler form on lifting}. Then we have
    \begin{align*}0 &\leq C_{[1]}  T_{[1]}^{(m)} \wedge \pi_{1,0}^*(\omega) + \Theta_{h_{[1]}}(\mathcal{O}_{X_1}(1)) \\ &= C_{[1]} T \wedge \omega + \{T_{[1]}^{(m)}\} \cdot u_{1} < 0 \qquad\text{for } m \text{ sufficiently large}.\end{align*}
    This is a contradiction, and thus $\overline{\chi}_1(T) < \infty$. The result follows. 
\end{proof}

The next proposition show that the notion of geometric Euler characteristic is intrinsic, in the sense that it does not depend on the choice of the ambient manifold.

\begin{proposition}\label{prop do not depend on ambient manifold}
   Let $Y$ be a submanifold of $X$ with the inclusion map $\iota:Y \hookrightarrow X$. Let $k\in \Z^+ \cup \{\infty\}$. Let $T$ be a $k$-liftable current on $Y$. Then $\iota_* (T)$ is a $k$-liftable current on $X$ with $\chi_l(\iota_*(T)) = \chi_l(T)$ and $\overline{\chi}_l(\iota_*(T)) \geq \overline{\chi}_l(T)$ for all $l\leq k$. 
\end{proposition}

\begin{proof}
    By Proposition~\ref{prop increasing of chi_k}, we can assume $k\in \Z^+$. Denote by $(Y_k,W_k)$ the Demailly-Semple tower of $Y$. By the functoriality of the Demailly-Semple tower, for every $l\leq k$, there exist embedding maps $\iota_l: Y_l\hookrightarrow X_l$ such that $(V_l)|_{Y_l} \cap T_{Y_l} = W_l$ and $\mathcal{O}_{X_l}(1)|_{Y_l} = \mathcal{O}_{Y_l}(1)$. Moreover, we also have $\iota_l^{-1}(D_l) = \P(T_{Y_{l-1}/Y_{l-2}})$ for $l\geq 2$. Then, by this construction, if $T_{[l]}$ is a geometric (resp. cohomological) $k$-lifting of $T$, then $\iota_*(T_{[l]})$ is a geometric (resp. cohomological) $k$-lifting of $\iota_*(T)$. Since $c_1(\mathcal{O}_{Y_l}(1)) = \iota^*(c_1(\mathcal{O}_{X_l}(1)))$, we deduce that $\chi_l(\iota_*(T)) \geq \chi_l(T)$ and $\overline{\chi}_l(\iota_*(T)) \geq \overline{\chi}_l(T)$.

    To prove $\chi_l(\iota_*(T)) = \chi_l(T)$, we use the same idea as in the proof of Theorem~\ref{theorem coincide with k-Euler characteristic}. We prove by induction that for $p\leq l$ and for any geometric $p$-lifting $T_{[p]}$ of $\iota_*(T)$, we have $(\pi_{p,0})_* (\mathbf{1}_{X_p\setminus Y_p}(T_{[p]})) =0 $ and $\{\mathbf{1}_{X_p\setminus Y_p}(T_{[p]})\} \cdot u_p \geq 0$.

    We consider first the case where $p=1$. Let $T_{[1]}$ be a $1$-lifting of $\iota_*(T)$. By Lemma~\ref{lemma support lemma for closed current}, $\mathbf{1}_{\pi_{1,0}^{-1}(Y)} (T_{[1]})$ is a current on $\pi_{1,0}^{-1}(Y) = \P(T_X)|_Y$. Let $x\in \pi_{1,0}^{-1}(Y)$. Then $T_x (\pi_{1,0}^{-1}(Y)) \cap (V_1)_x$ equals $(V_1)_x$ if $x\in Y_1$. Otherwise, it equals $\{\xi \in T_x (X_1): (\pi_{1.0})_*\xi = 0\}$. Thus $(\pi_{1,0})_*(\mathbf{1}_{\pi_{1,0}^{-1}(Y) \setminus Y_1} T_{[1]}) = 0$. Hence, since $\iota_*(T)$ has support in $Y$, we conclude that $(\pi_{1,0})_*(\mathbf{1}_{X_1 \setminus Y_1} T_{[1]}) = 0$. Moreover, this also implies $\{\mathbf{1}_{X_1 \setminus Y_1} T_{[1]}\}\cdot u_1 \geq 0$.

    Suppose that the statement holds for $p-1$ for $p\geq 2$. We define $T_{[p]}'\colonequals \mathbf{1}_{\pi_{p,p-1}^{-1}(Y_{p-1})} T_{[p]}$ and $T_{[p]}''\colonequals\mathbf{1}_{X_p\setminus \pi_{p,p-1}^{-1}(Y_{p-1})} T_{[p]}$. Since $T_{[p]}$ is a geometric $p$-lifting, we have 
    \[\{T_{[p]}''\} \cdot u_p \geq \{(\pi_{p,p-1})_*(T_{[p]}'')\} \cdot u_{p-1} = \{\mathbf{1}_{X_{p-1}\setminus Y_{p-1}} (\pi_{p,p-1})_* (T_{[p]})\} \cdot u_{p-1} \geq 0.  \]
    For $T_{[p]}'$, we argue similarly to the case $p=1$ and conclude the proof of the statement.

    We deduce from the statement that $\chi_l(\iota_*(T)) = \sup \Big\{ -\{\mathbf{1}_{X_l\setminus Y_l} (T_{[l]})\} \cdot u_l\Big\} = \chi_l(T)$. 
\end{proof}

The next proposition shows that both the geometric and cohomological notions are invariant under the action of $\Aut(X)$.

\begin{proposition}\label{prop liftable preserved}
    Let $\varphi:X\to X$ be an automorphism, and $T$ be a $k$-liftable current. Then $\varphi_*(T)$ is also a $k$-liftable current, and we have  ${\chi}_k(T) = {\chi}_k(\varphi_*(T))$ and $\overline{\chi}_k(T) = \overline{\chi}_k(\varphi_*(T))$.
\end{proposition}

\begin{proof}
    Let $(X,V)$ be a complex directed manifold. Let $\varphi:X\to X$ be an automorphism of $X$ such that $\varphi_*(V) = V$. Then $\varphi$ induces an automorphism 
    \[\widetilde{\varphi}:\widetilde{X}\to \widetilde{X}: (x,[v]) \mapsto (\varphi(x),[\varphi_*(v)]).\] Let $\pi:\widetilde{X} \to X$ be the projection. Let $\xi \in T_{(x,[v]) }\widetilde{X}$ such that $\pi_*(\xi) \in \C v$. Then we have $\widetilde{\varphi}_*(\xi) \in T_{(\varphi(x),[\varphi_*(v)])} \widetilde{X}$ and $\pi_* (\widetilde{\varphi}_*(\xi)) = \varphi_*(\pi_*(\xi)) \in \C \varphi_*(v)$. This implies that $\widetilde{\varphi}_* (\widetilde{V}) =\widetilde{V}$. Moreover, since $\widetilde{\varphi}$ is a biholomorphism between $\pi^{-1}(x)$ and $\pi^{-1}(\varphi(x))$, we have $\widetilde{\varphi}^*(\mathcal{O}_{\widetilde{X}}(1) = \mathcal{O}_{\widetilde{X}}(1)$.

     Since $\varphi_*(T_X) = T_X$, we can prove by induction and the above argument that $\varphi$ induces an automorphism $\varphi_{[k]}:X_k\to X_k$ such that $(\varphi_{[k]})_*(V_k) = V_k$. Let $T_{[k]}$ be a geometric (resp. cohomological) $k$-lifting of $T$. Then, we must have $(\varphi_{k})_*(T_{[k]})$ is a $k$-lifting of $\varphi_*(T)$. Since 
    \[\{(\varphi_{[k]})_* (T_{[k]})\} \cdot u_k =\{T_{[k]}\} \cdot \varphi_{[k]}^* (u_k) 
    = \{T_{[k]}\} \cdot u_k,\] we get ${\chi}_k(T) = {\chi}_k(\varphi_*(T))$ (resp. $\overline{\chi}_k(T) =\overline{\chi}_k(\varphi_*(T))$.
\end{proof}

We have seen the role of the geometric notion in Theorem~\ref{theorem coincide with k-Euler characteristic} and Proposition~\ref{prop do not depend on ambient manifold}. The advantage of the cohomological notion is illustrated by the following two propositions.

\begin{proposition}\label{prop pick good k-lifing}
    Let $k\in \Z^+$. Let $T$ be a $k$-liftable current on $X$. Then there always exists a cohomological $k$-lifting $T_{[k]}$ of $T$ such that $-\overline{\chi}_k(T) = \{T_{[k]}\} \cdot u_k$. 
\end{proposition}

\begin{proof}
    By definition, we can pick a sequence $(T_{[k]}^{(m)})_{m\in \Z^+}$ of cohomological $k$-liftings of $T$ such that $\{T_{[k]}^{(m)}\} \cdot u_k \to -\overline{\chi}_k(T)$. By Proposition~\ref{prop finiteness of Euler characteristic}, $-\overline{\chi}_k(T)$ is finite.
    Therefore, we can assume that for all $m$, $\{T_{[k]}^{(m)}\} \cdot u_k \leq -\overline{\chi}_k(T) + 1$. Moreover, since $T_{[k]}^{(m)}$ is a cohomological $k$-lifting, we have $\{T_{[k]}^{(m)}\} \cdot \pi_{k,j}^* (u_j)\leq \{T_{[k]}^{(m)}\} \cdot u_k \leq -\overline{\chi}_k(T) + 1$. Hence, by the construction \eqref{eq define Kahler form on lifting} of $\omega_{[k]}$, we deduce that $(T_{[k]}^{(m)})$ has uniform bounded mass with respect to the metric $\omega_{[k]}$ on $X_k$. Let $T_{[k]}$ be a weak limit of this sequence. Then $T_{[k]}$ is a cohomological $k$-lifting of $T$ and $\{T_{[k]}\} \cdot u_k = -\overline{\chi}_k(T)$.
\end{proof}

\begin{proposition}\label{prop compactness for uniform bound k Euler char}
    Let $k\in \Z^+$. Let $(T^{(m)})_{m\in \Z^+}$ be a sequence of $k$-liftable currents such that $\sup_{m \in \Z^+} -\overline{\chi}_{k}(T^{(m)}) = L < \infty$ and $\|T^{(m)}\|_{\omega} = 1$ for every $m\in \Z^+$. Let $T$ be a weak limit of this sequence. Then $T$ is $k$-liftable and $-\overline{\chi}_k(T) \leq L$.
\end{proposition}

\begin{proof}
    By Proposition~\ref{prop pick good k-lifing}, for each $m$, we can pick $k$-lifting $T_{[k]}^{(m)}$ of $T^{(m)}$ such that
    \[\{T_{[k]}^{(m)}\} \cdot u_k = -\overline{\chi}_k(T^{(m)}) \leq  L.\]
    Similar to the proof of Proposition~\ref{prop pick good k-lifing}, we can pick a weak limit $T_{[k]}$ of this sequence such that $T_{[k]}$ is a cohomological $k$-lifting of $T$. Thus $T$ is $k$-liftable and $-\overline{\chi}_k(T) \leq L$.
\end{proof}

We now prove a decomposition theorem for liftable currents. This theorem will be useful for estimating the Euler characteristic of liftable currents.

\begin{proposition}\label{prop decomp for liftable current}
    Let $A$ be an analytic subset of $X$. Let $k\in \Z^+ \cup \{\infty\}$. Let $T$ be a $k$-liftable current. Then the currents $\mathbf{1}_{A} T$ and $\mathbf{1}_{X\setminus A}$ are $k$-liftable currents. Moreover, we have
    \begin{equation}\label{eq decomp for liftable currents}
    {\chi}_l(T) = {\chi}_{l}(\mathbf{1}_A T) + {\chi}_{l}(\mathbf{1}_{X\setminus A} T) \qquad \text{and} \qquad \overline{\chi}_l(T) \geq \overline{\chi}_{l}(\mathbf{1}_A T) + \overline{\chi}_{l}(\mathbf{1}_{X\setminus A} T),\end{equation}
    for every $l\leq k$.
\end{proposition}

\begin{proof}
    We only need to prove for $k\in \Z^+$. Fix $k\in \Z^+$. Let $T_{[k]}$ be an almost $k$-lifting of $T$. Then, by Lemma~\ref{lemma support lemma for closed current} and Proposition~\ref{prop dominated directed currents}, $\mathbf{1}_{\pi_{k,0}^{-1}(A)} T_{[k]}$ is an almost $k$-lifting of $\mathbf{1}_AT$ and $\mathbf{1}_{\pi_{k,0}^{-1}(X\setminus A)} T_{[k]}$ is an almost $k$-lifting of $\mathbf{1}_{X\setminus A}T$. Therefore, $\mathbf{1}_{A} T$ and $\mathbf{1}_{X\setminus A} T$ are $k$-liftable currents.

    We now prove \eqref{eq decomp for liftable currents}. Let $\epsilon > 0$ and $l\leq k$. By definition, there exists a geometric $l$-lifting $T_{[l]}$ of $T$ such that $\{T_{[l]} \} \cdot u_l \leq - \overline{\chi}_l(T) + \epsilon$ . 
    Since $\mathbf{1}_{\pi_{l,0}^{-1}(A)} T_{[l]}$ is a geometric $l$-lifting of $\mathbf{1}_AT$ and $\mathbf{1}_{\pi_{l,0}^{-1}(X\setminus A)} T_{[l]}$ is a geometric $l$-lifting of $\mathbf{1}_{X\setminus A}T$, we have
    \begin{align*}
        -{\chi}_{l}(\mathbf{1}_A T) - {\chi}_{l}(\mathbf{1}_{X\setminus A} T) &\leq  \{\mathbf{1}_{\pi_{l,0}^{-1}(A)} T_{[l]}\} \cdot u_l + \{\mathbf{1}_{\pi_{l,0}^{-1}(X\setminus A)} T_{[l]}\} \cdot u_l \\
        &=\{T_{[l]} \} \cdot u_l \leq - {\chi}_l(T) +\epsilon.
    \end{align*}
    Thus, we obtain $-{\chi}_{l}(\mathbf{1}_A T) - {\chi}_{l}(\mathbf{1}_{X\setminus A} T) \leq -\chi_l(T)$ by letting $\epsilon \to 0^+$.

    For the reverse inequality, we pick a geometric $l$-lifting $T_{[l]}'$ of $\mathbf{1}_{A} T$ and a geometric $l$-lifting $T_{[l]}''$ of $\mathbf{1}_{X\setminus A} T$ such that
    \[\{T_{[l]}' \} \cdot u_l \leq - \overline{\chi}_l(\mathbf{1}_{A}T) + \epsilon , \qquad \{T_{[l]}'' \} \cdot u_l \leq - \overline{\chi}_l(\mathbf{1}_{X\setminus A}T) +\epsilon.\]
    Since $T_{[l]}' + T_{[l]}''$ is a geometric $l$-lifting of $T$, we get
    \[
        - \overline{\chi}_l(T) \leq \{T_{[l]}' + T_{[l]}''\} \cdot u_l\leq 
        -\overline{\chi}_{l}(\mathbf{1}_A T) - \overline{\chi}_{l}(\mathbf{1}_{X\setminus A} T) + 2\epsilon.
    \]
    Thus, we deduce the reverse inequality by letting $\epsilon \to 0^+$. The proof for $\overline{\chi}_l$ is similar, note that the sum of two cohomological $l$-lifting is a cohomological $l$-lifting.
\end{proof}

We end this section by proving an upper bound for Euler characteristic of liftable currents in projective space, using a dynamical argument. The idea is to move the current by automorphisms and make it degenerate to some curves where the Euler characteristic is easy to compute.

\begin{proposition}\label{prop lower bound in Pn}
    Let $k\in \Z^+ \cup \{\infty\}$. Let $T$ be a $k$-liftable current on $\P^n$ with $n\geq 2$ such that $\|T\|_{\omega_\FS} = 1$. Then $-\chi_l(T) \geq - 2$ and $-\overline{\chi}_l(T) \geq - 2$ for all $l\leq k$.
\end{proposition}

\begin{proof}
    By Proposition~\ref{prop increasing of chi_k}, we only need to prove this proposition for $k=1$. Let $T$ be a $1$-liftable current of mass $1$ with respect to the Fubini-Study metric $\omega_{\FS}$. Define $\sigma_T \colonequals T \wedge \omega_{\FS}$. Let $\mu$ be a smooth probability measure on the dual space $(\P^n)^*$ that parametrizing hyperplanes in $\P^n$. Let $I$ denotes the incidence set in $\P^n \times (\P^n)^*$ that contains points $(x,H)$ such that $x\in H$. By Fubini theorem, we have
    \[\int_{I} d\sigma_T (x)\times d\mu(H)=\int_{(\P^n)^*} \sigma_T(H) d\mu(y) = \int_{\P^n} \mu(\{H: x\in H\}) d\sigma_T(x).\]
    Since the set $\{H: x\in H\}$ is a hyperplane in $(\P^n)^*$, we have $\mu(\{H: x\in H\}) = 0$. This implies that $\sigma_T(H) = 0$ almost everywhere with respect to $\mu$. Thus, we can always choose a hyperplane $H$ such that $T$ puts no mass on $H$.

    Choose a representative $z=[z_0:\cdots:z_n]$ for $\P^n$ such that $T$ puts no mass on the set $A\colonequals \{z_{n-1}=z_n=0\}$.
    For $\lambda \in \C^*$, consider the map:
    \[\varphi_\lambda:\C^{n+1} \to \C^{n+1} : (z_0,\ldots,z_n)\mapsto ( z_0,\ldots,z_{n-2},\lambda z_{n-1},\lambda z_{n}).\]
    This map induces an automorphism of $\P^n$. Let $z\in X\setminus A$, we have 
    \[\varphi_\lambda(z) = [z_0:\cdots:z_{n-2}:\lambda z_{n-1}: \lambda z_{n}] = \Big[\frac{z_0}{\lambda}:\cdots:\frac{z_{n-2}}{\lambda}:z_{n-1}: z_{n}\Big] .\]
    Letting $|\lambda| \to \infty$, we have $\varphi_\lambda(z)$ converging to the point $[0:\cdots:0:z_{n-1}:z_n]$. Therefore, the sequence of positive currents $(\varphi_\lambda)_*(T)$ will converge to the current of integration over the line $ L\colonequals\{z_0=z_1=\cdots = z_{n-2} = 0\}$. By Proposition~\ref{prop liftable preserved} and Proposition~\ref{prop compactness for uniform bound k Euler char}, we have ${\chi}_1(T) = \lim\limits_{\lambda\to \infty} {\chi}_1((\varphi_\lambda)_*(T)) \leq {\chi}_1([L]) = 2$. The result follows.
\end{proof}

\section{Hyperbolic indices: basic properties and examples} \label{sec hyperbolic indices introduce and basic}

We now introduce the notion of hyperbolic indices, using the theory of Euler characteristics of liftable currents developed in the last section. We then study basic properties and compute several examples.

Let $(X,\omega)$ be a compact K\"ahler manifold. For $k\in \Z^+$, define
\begin{equation}\label{eq define geometric hyp k depend metric}
\hyp^{(k)}(X,\omega) \colonequals \inf \Big\{ \frac{-{\chi}_k(T)}{\|T\|_\omega} : T \text{ is }k\text{-liftable current}  \Big\}
\end{equation}
and call it the \textit{geometric $k$-hyperbolic index} of $(X,\omega)$. We also define
\begin{equation}\label{eq define cohomological hyp k depend metric}
\chyp^{(k)}(X,\omega) \colonequals \inf \Big\{ \frac{-\overline{\chi}_k(T)}{\|T\|_\omega} : T \text{ is } k\text{-liftable current}  \Big\},
\end{equation}
and call this number the \textit{cohomological $k$-hyperbolic index} of $(X,\omega)$. If there is no $k$-liftable current on $X$, we put $\hyp^{(k)}(X,\omega) = \chyp^{(k)}(X,\omega) \colonequals \infty$. We define $\hyp^{(\infty)}(X,\omega)\colonequals \sup_k \hyp^{(k)} (X,\omega)$ and $\chyp^{(\infty)}(X,\omega)\colonequals \sup_k \chyp^{(k)} (X,\omega)$. Similarly, we can define $\hyp^{(k)}_*(X,\omega)$ and $\chyp^{(k)}_*(X,\omega)$ using weakly $k$-liftable currents. By definition, we have $\hyp^{(k)}(X,\omega) \geq \hyp_*^{(k)}(X,\omega)$, $\chyp^{(k)}(X,\omega) \geq \chyp_*^{(k)}(X,\omega)$ and the positivity of these numbers does not depend on the choice of the K\"ahler metric $\omega$.

We can also define the \text{geometric and cohomological hyperbolic indices of $X$} without the reference metric $\omega$:
\begin{equation}\label{eq define hyp k not depend metric}
\hyp^{(k)}(X)\colonequals \sup_\omega \Big\{\hyp^{(k)}(X,\omega)\Big\}, \qquad \chyp^{(k)}(X)\colonequals \sup_\omega \Big\{\chyp^{(k)}(X,\omega)\Big\}, 
\end{equation}
where the supremum runs over all K\"ahler form $\omega$ on $X$ such that $\int_X \omega^n =1$. Similarly, we can define $\hyp_*^{(k)}(X)$ and $\chyp_*^{(k)}(X)$.

\begin{remark}\label{remark about pluripotential hyp and alg hyp}
    By Theorem~\ref{theorem coincide with k-Euler characteristic}, we see that if $\hyp^{(\infty)}(X)>0$, then $X$ is algebraically hyperbolic in the sense of Demailly (see Section~\ref{sec complex hyperbolicity}). We will see later that if $\chyp^{(\infty)}(X)>0$, then $X$ is Kobayashi hyperbolic (see Theorem~\ref{theorem chyp > 0 implies Kobayashi hyperbolic}). We can develop a similar theory for complex directed manifolds $(X,V)$ or define the hyperbolic indices for compact subsets on K\"ahler manifolds.
\end{remark}

  We now study the basic properties of hyperbolic indices $\hyp^{(k)}$ and $\chyp^{(k)}$. Most of these statements (except Proposition~\ref{prop basic property of hyp}(iv)) hold for $\hyp^{(k)}_*$ and $\chyp^{(k)}_*$, using the same arguments.

 \begin{proposition}\label{prop basic property of hyp}
     For $k\in \Z^+\cup \{\infty\}$, we have the following statements:
     \begin{itemize}
         \item[(i)] $\hyp^{(k)}(X,\omega) \geq \chyp^{(k)}(X,\omega), \hyp^{(k)}(X) \geq \chyp^{(k)}(X)$ and the equality holds if $k=1$.

         \item[(ii)] $\hyp^{(k)}(X,\omega), \hyp^{(k)}(X),\chyp^{(k)}(X,\omega), \chyp^{(k)}(X)$ are increasing in $k$.
         
         \item[(iii)] 
         $\chyp^{(k)}(X,\omega) >  - \infty$ and $ \chyp^{(k)} (X) > -\infty$.
         
         \item[(iv)]  $\hyp_*^{(1)}(X,\omega)$ and $\hyp_*^{(1)}(X)$ are finite.
         
         \item[(v)] If $\chyp^{(k)}(X,\omega)$ is finite, there exists a $k$-liftable current $T$ such that $\chyp^{(k)} (X,\omega)= -\overline{\chi}_k(T) /\|T\|_\omega$.

         \item[(vi)] If $X$ is a projective manifold, then $\hyp^{(k)}(X,\omega)$ and $\hyp^{(k)}(X)$ are finite.

         \item[(vii)] Let $Y$ be a complex submanifold of $X$. Then $\hyp^{(k)}(Y,\omega|_Y) \geq \hyp^{(k)}(X,\omega)$ and $\chyp^{(k)}(Y,\omega|_Y) \geq \chyp^{(k)}(X,\omega)$.
     \end{itemize}
 \end{proposition}

 \begin{proof}
     Statements (i) and (ii) follow directly from definitions and Proposition~\ref{prop increasing of chi_k}.

     To prove (iii), we proceed similarly to the proof of Proposition~\ref{prop finiteness of Euler characteristic}. It is enough to consider $k = 1$. We assume that there exist $1$-liftable currents on $X$ (otherwise, it is trivial). Suppose that there exists a sequence $(T^{(m)})_{m\in \Z^+}$ of $1$-liftable currents with $\|T^{(m)}\|_\omega = 1$ such that $\lim\limits_{m\to \infty} -\overline{\chi}_1(T^{(m)}) = - \infty$. By Proposition~\ref{prop pick good k-lifing}, for each $m\in \Z^+$, we can pick $1$-liftings $T^{(m)}_{[1]}$ of $T^{(m)}$ such that 
     \[-\overline{\chi}_1(T^{(m)}) = \{T^{(m)}_{[1]}\} \cdot u_{1}.\] By the construction \eqref{eq define Kahler form on lifting} of $\omega_{[1]}$, we have
     \begin{align*}0\leq \{T^{(m)}_{[1]}\} \cdot \{\omega_{[1]}\} &= C_{[1]}\cdot \{T_{[1]}^{(m)}\} \cdot \{\pi_{1,0}^*(\omega)\} +\{T^{(m)}_{[1]}\} \cdot u_{1}\\ &=C_{[1]} -\overline{\chi}_1(T^{(m)}). \end{align*}
     Letting $m\to \infty$ gives a contradiction. We deduce that $\hyp^{(k)}(X,\omega) > - \infty$, and therefore, $\hyp^{(k)}(X) > -\infty$ follows directly.

     Statement (iv) follows from Proposition~\ref{lemma existence liftable current induction level}.

     For (v), we assume for a moment that $k\in \Z^+$. By definition, there exists a sequence of $k$-liftable currents $(T^{(m)})_{m\in \Z^+}$ with $\|T^{(m)}\|_\omega = 1$ such that $\chyp^{(k)}(X,\omega) = \lim\limits_{m\to \infty} - \overline{\chi}_k(T^{(m)})$. Let $\epsilon>0$. Then, we can assume that $-\overline{\chi}_k(T^{(m)}) \leq \chyp^{(k)}(X,\omega) + \epsilon$ for all $m$ sufficiently large. Let $T$ be a weak limit of this sequence. By Proposition~\ref{prop compactness for uniform bound k Euler char}, $T$ is a $k$-liftable current with $\|T\|_\omega = 1$ and $-\overline{\chi}_k(T) \leq \chyp^{(k)}(X,\omega) + \epsilon$. Letting $\epsilon \to 0^+$, we infer that $T$ is the desired current. Suppose that $\chyp^{(\infty)}(X,\omega)$ is finite. Then, for each $k\in \Z^+$, there exists a $k$-liftable current $T^{(k)}$ with $\|T^{(k)}\|_\omega = 1$ such that $\chyp^{(k)}(X,\omega) = -\overline{\chi}_k(T^{(k)})$. Let $T$ be a weak limit of this sequence. Then $T$ is $k$-liftable and $-\overline{\chi}_k(T) \leq \chyp^{(\infty)}(X,\omega)$ for all $k\in \Z^+$. Thus, $-\overline{\chi}_\infty(T) \leq \chyp^{(\infty)}(X,\omega)$. Moreover, by definition, $-\overline{\chi}_k(T) \geq \chyp^{(k)}(X,\omega)$ for all $k\in \Z^+$. This implies that $-\overline{\chi}_\infty(T) = \chyp^{(\infty)}(X,\omega)$. We obtain from this statement a different description of $\chyp^{(\infty)}(X,\omega)$ as follows:
     \[\chyp^{(\infty)}(X,\omega) = \inf \Big\{\frac{-\overline{\chi}_\infty(T)}{\|T\|_\omega}: T \text{ is }k\text{-liftable current}\Big\}.\]

     We now deal with (vi). When $X$ is a projective manifold, we consider a closed curve $\mathscr{C}$ in $X$. The current $[\mathscr{C}]$ is $\infty$-liftable, and we have \[\hyp^{(k)}(X,\omega) \leq \frac{-{\chi}_k([\mathscr{C}])}{\|[\mathscr{C}]\|_\omega} \leq \frac{-{\chi}_\infty([\mathscr{C}])}{\|[\mathscr{C}]\|_\omega}  < \infty.\]
     For $\hyp^{(k)}(X)$, we choose the curve $\mathscr{C}$ as follows. Pick a very ample line bundle $L$ over $X$, then choose a general tuple of $n-1$ different divisors in this class. Let $\mathscr{C}$ be their intersection. Then $[\mathscr{C}]$ is a current in $c_1(L)^{n-1}$ and $[\mathscr{C}]$ is a $\infty$-liftable current. Let $\omega$ be a K\"ahler form such that $\|\omega^n\|_X = 1$, by \cite[Proposition 5.2]{Demailly-numerical-criterion}, we have
    \[c_1(L)^{n-1}\cdot \{\omega\} \geq (c_1(L)^n)^{(n-1)/n} \cdot (\{\omega\}^n)^{1/n} = (c_1(L)^n)^{(n-1)/n} > 0.\]
    This implies that
    \[\hyp^{(k)}(X) \leq \sup_{\omega} \Big\{ \frac{-\chi_k([\mathscr{C}])}{\|\mathscr{C}\|_\omega} \Big\} \leq  \frac{-\chi_\infty(\mathscr{C})}{(c_1(L)^n)^{(n-1)/n}} < \infty,\]
    where the supremum runs over all K\"ahler form $\omega$ such that $\|\omega^n\|_X = 1$.

    The last statement follows from definition and Proposition~\ref{prop do not depend on ambient manifold}. This can be interpreted as saying that $X$ is ``less'' hyperbolic than $Y$. As a direct corollary, if $\hyp^{(k)}(X) > 0$ and $\hyp^{(k)}(Y) \leq 0$, then $Y$ cannot be embedded into $X$.
 \end{proof}

We prove some theorems on the continuity of hyperbolic indices when the K\"ahler form or the manifold varies. By definition, $\hyp^{(k)}(X,\omega)$ and $\chyp^{(k)}(X,\omega)$ only depend on the K\"ahler class $\{\omega\}$. Therefore, they define  functions on the K\"ahler cone.

\begin{proposition}\label{prop Lipschitz over kahler cone}
    Let $k\in \Z^+ \cup \{\infty\}$. Suppose that $\hyp^{(k)}(X,\omega)$ is finite for some K\"ahler form $\omega$. Then the functions $\hyp^{(k)}(X,\cdot)$ and $\chyp^{(k)}(X,\cdot)$ are locally Lipschitz.
\end{proposition}

For the proof of this proposition, we need the following lemma, which is quite similar to \cite[Proposition 2.3]{Tosatti-Cao-regularity-volume}.

\begin{lemma}\label{lemma control Lipschitz}
    Let $V$ be an $n$-dimensional real vector space and $\mathcal{C}\subset V$ be an open convex cone. Let $f:\mathcal{C} \to \R^{\geq 0}$ be a function such that
    \begin{itemize}
        \item[(i)] $f(\lambda \alpha) = \lambda^{-1}f(\alpha)$ for all $\alpha \in \mathcal{C}$ and $\lambda > 0$,
        \item[(ii)] $f(\alpha+\beta) \leq f(\alpha)$ for every $\alpha,\beta \in \mathcal{C}$,
        \item[(iii)] $f$ is locally bounded.
    \end{itemize}
    Then $f$ is locally Lipschitz on $\mathcal{C}$.
\end{lemma}

\begin{proof}
    The proof is almost the same as that of \cite[Proposition 2.3]{Tosatti-Cao-regularity-volume}. We present it here for the reader's convenience. Fix $\{e_i\}_{i=1}^n$ in $\mathcal{C}$ that forms a basis of $V$. We consider the norm $\|\sum_{i=1}^n c_i e_i\| = \sum_{i=1}^n |c_i|$. If $0\in \mathcal{C}$, then $\mathcal{C} = V$ and $f$ is constant by condition (ii). Therefore, we can assume that $0 \notin \mathcal{C}$. Given $\alpha_0 \in \mathcal{C}$, pick $r$ small enough such that $U\colonequals\B(\alpha_0,r) \subset \mathcal{C}$, $\B(\alpha_0/2,r) \subset \mathcal{C}$, and $f$ is bounded on $\B(\alpha_0,r)$. Let $\beta \in \mathcal{C}$, we have $\alpha_0/2 - r \beta/{\|\beta\|} \in \mathcal{C}$
     and thus
     $\alpha_0\|\beta\|/(2r) - \beta \in \mathcal{C}$.
    Let $\alpha \in U$, we have
    \[\alpha - \frac{\alpha_0}{2} = \frac{\alpha_0}{2} + (\alpha - \alpha_0) \in \B(\alpha_0/2,r) \subset \mathcal{C}.\]
    Hence, $\alpha \|\beta\|/r - \alpha_0 \|\beta\|/(2r) \in \mathcal{C}$.
    Suppose that $\|\beta\| < r$, then $\alpha - (\alpha \|\beta\|)/r \in \mathcal{C}$. We have
    \[\alpha - \frac{\alpha_0\|\beta\|}{2r} = \alpha - \frac{\alpha \|\beta\|}{r} + \frac{\alpha \|\beta\|}{r} - \frac{\alpha_0\|\beta\|}{2r} \in \mathcal C,\]
    \[ \alpha - \beta = \alpha  - \frac{\alpha_0\|\beta\|}{2r} + \frac{\alpha_0\|\beta\|}{2r} - \beta \in \mathcal C.\]
    Applying conditions (i) and (ii) gives us:
    \begin{equation}\label{eq lipschitz for hyp, first bound} f(\alpha) \leq f(\alpha - \beta) \leq f\Big(\alpha - \frac{\alpha_0 \|\beta\|}{2r}\Big) \leq f\Big(\alpha - \frac{\alpha \|\beta\|}{r}\Big) \leq \Big(1 - \frac{\|\beta\|}{r}\Big)^{-1} f(\alpha).\end{equation}
    Similarly, we get
    \begin{equation}\label{eq lipschitz for hyp, second bound}
        f(\alpha) \geq f(\alpha + \beta) \geq f\Big(\alpha + \frac{\alpha_0 \|\beta\|}{2r}\Big) \geq f\Big(\alpha + \frac{\alpha \|\beta\|}{r}\Big) \geq \Big(1 + \frac{\|\beta\|}{r}\Big)^{-1} f(\alpha).
    \end{equation}
    Combining \eqref{eq lipschitz for hyp, first bound} and \eqref{eq lipschitz for hyp, second bound}, for $\|\beta\| < r/2$, we have
    \begin{equation}\label{eq lipschitz for hyp, third bound}|f(\alpha) - f(\alpha \pm \beta)| \leq \frac{2\|\beta\|}{r} f(\alpha).\end{equation}
    Let $\alpha_1, \alpha_2 \in \B(\alpha_0,r/6)$ and write $\alpha_1 -\alpha_2 = \sum_{j=1}^n c_j e_j$. Then $\|\alpha_1 - \alpha_2\| = \sum_{j=1}^n |c_j|$. Thus, for $1\leq k\leq n$, $\sum_{j=1}^k |c_j| \leq r/3$. Therefore, we get $\|\alpha_2 + \sum_{j=1}^k c_je_j - \alpha_0 \| < r/3 + r/6 = r/2$. It follows that $\alpha_2 + \sum_{j=1}^k c_je_j \in \mathcal{C}$. We now use \eqref{eq lipschitz for hyp, third bound} to get
    \begin{align*} |f(\alpha_1) - f(\alpha_2)| &= \Big| f\Big(\alpha_2 + \sum_{j=1}^n c_j e_j\Big) -f(\alpha_2)\Big|\\
&\leq \sum_{k=1}^{n} \Big| f\Big(\alpha_2 + \sum_{j=1}^{k} c_j e_j\Big) -f\Big(\alpha_2 + \sum_{j=1}^{k-1} c_j e_j\Big)\Big|\\
&\leq \sum_{k=1}^n \frac{2|c_k|}{r}  \Big| f\Big(\alpha_2 + \sum_{j=1}^{k-1} c_j e_j\Big)\Big| \\
&\leq \sup_U (f) \frac{2}{r} \|\alpha_1 - \alpha_2\|,
\end{align*}
and we get the Lipschitz bound on $\B(\alpha_0, r/6)$.
\end{proof}

\begin{proof}[Proof of Proposition~\ref{prop Lipschitz over kahler cone}]
We define $f(\{\omega\})\colonequals\hyp^{(k)}(X,\{\omega\})$ if $\hyp^{(k)}(X,\omega) \geq 0$ and define $f(\{\omega\})\colonequals-\hyp^{(k)}(X,\{\omega\})$ if $\hyp^{(k)}(X,\omega) \leq 0$. We only need to check that $f$ satisfies all the conditions in Lemma~\ref{lemma control Lipschitz}. By definition, $f$ satisfies (i). Let $\alpha,\beta $ be K\"ahler classes. Since $\alpha + \beta > \alpha$, we have $\{T\} \cdot (\alpha +\beta) \geq  \{T\} \cdot \alpha$. Therefore, $f$ satisfies (ii). Let $\mathcal{B}$ be a compact set in the K\"ahler cone. We only need to show that $f$ is bounded on $\mathcal{B}$. Fix a K\"ahler metric $\omega_X$ on $X$. Since $\mathcal{B}$ is compact, there exists a positive constant $A$ such that $A^{-1} \{\omega_X\} \leq \alpha \leq A\{\omega_X\}$ for every $\alpha \in \mathcal{B}$. This gives us $A^{-1}f(\{\omega_X\}) \leq f(\alpha) \leq A f(\{\omega_X\})$ for every $\alpha \in \mathcal{B}$. The proof for $\chyp^{(k)}(X,\cdot)$ is similar.
\end{proof}

We now consider the case in which the manifold varies. Let $p: \mathcal{X} \to S$ be a proper submersion between a K\"ahler manifold $\mathcal{X}$ and a complex manifold $S$. For $s\in S$, we define $X^{(s)}\colonequals p^{-1}(s)$. Since $\pi$ is proper, $X^{(s)}$ is a compact K\"ahler manifold for all $s\in S$. We assume that $X^{(s)}$ is connected for all $s\in S$. We denote by $(X^{(s)}_k,V^{(s)}_k)$ the Demailly-Semple tower of $X^{(s)}$. Let $\omega_{\mathcal{X}}$ be a K\"ahler form on $\mathcal{X}$ and put $\omega^{(s)}\colonequals \omega_{\mathcal{X}}|_{X^{(s)}}$. We have the following lower semi-continuity property:

\begin{proposition}\label{prop openness of hyp_1 function}
    For $k\in \Z^+\cup \{\infty\}$, $ s \mapsto \chyp^{(k)}(X^{(s)},\omega^{(s)})$ is a lower semi-continuous function with respect to the Euclidean topology on $S$.
\end{proposition}

\begin{proof}
     Since the supremum of lower semi-continuous functions is also a lower semi-continuous function, we only need to prove this proposition for $k\in \Z^+$.

    Fix $k\in \Z^+$ and $s\in S$. Let $(s_m)$ be a sequence of points in $S$ such that $s_m \to s \in S$. We need to show that 
    \[L\colonequals\liminf\limits_{m\to \infty} \chyp^{(k)}(X^{(s_m)},\omega^{(s_m)}) \geq \chyp^{(k)}(X^{(s)},\omega^{(s)}).\]
    If $L = \infty$, then we are done. Assume that $L$ is finite. Let $\epsilon > 0$. Replacing $(s_m)$ by a suitable subsequence, we can assume that $ \chyp^{(k)}(X^{(s_m)},\omega^{(s_m)}) < L+\epsilon$ for all $m$. Then, by Proposition~\ref{prop basic property of hyp}(v), we can pick a $k$-liftable current $T^{(m)}$ on $X^{(s_m)}$ such that $\chyp^{(k)}(X^{(s_m)},\omega^{(s_m)}) = -\overline{\chi}_k(T^{(m)})$ and $\|T^{(m)}\|_{\omega^{(s_m)}} = 1$.

    Pick a compact set $K$ of $S$ such that $s\in K$ and $s_m \in K$ for every $m$. We view $T^{(m)}$ as currents on $\mathcal{X}$. Then the sequence $T^{(m)}$ lies inside the compact set $p^{-1}(K)$ and has mass $1$ with respect to $\omega_{\mathcal{X}}$. Therefore, we can take a weak limit $T$ of this sequence. We see that $T$ is a current on $X^{(s)}$. Put 
    \[(\mathcal{X}_k,\mathcal{V}_k)\colonequals \bigcup_s (X_k^{(s)},V_k^{(s)})\qquad \text{and}\qquad p_{k}:\mathcal{X}_k \to S.\] 
    By Proposition~\ref{prop pick good k-lifing}, for each $m$, we can pick a cohomological $k$-lifting $T^{(m)}_{[k]}$ of $T^{(m)}$ such that 
    \[-\overline{\chi}_k(T^{(m)}) = \{T_{[k]}^{(m)}\} \cdot c_1(\mathcal{O}_{X^{(s_m)}}(1)).\] 
    Note that $\{T_{[k]}^{(m)}\} \cdot c_1(\mathcal{O}_{X^{(s_m)}}(1)) \leq L+\epsilon$ for every $m$. We view $T^{(m)}$ as currents on $\mathcal{X}$ and $T_{[k]}^{(m)}$ as currents on $\mathcal{X}_k$. Then the sequence $(T_{[k]}^{(m)})$ lies inside the compact set $p_{k}^{-1}(K)$. The compactness of the supports allows us to use a similar argument to the proof of Proposition~\ref{prop compactness for uniform bound k Euler char} to show that $T$ is a $k$-liftable current on $X^{(s)}$ with $-\overline{\chi}_k(T) \leq L +\epsilon$ and $\|T\|_{\omega^{(s)}} = 1$. Letting $\epsilon \to 0^+$, we have $-\overline{\chi}_k(T) \leq L$. We deduce that $\chyp^{(k)}(X^{(s)},\omega^{(s)}) \leq -\overline{\chi}_k(T) /\|T\|_{\omega^{(s)}} \leq L$ as desired. 
\end{proof}

We obtain the following direct corollary.

\begin{corollary}\label{cor openness in euclidean}
    For $k\in \Z^+ \cup \{\infty\}$, the set $\{s\in S:\chyp^{(k)}(X^{(s)})> 0\}$ is open in the Euclidean topology on $S$.
\end{corollary}

\begin{remark}\label{remark open of hyp}
    This openness property is similar to that of Kobayashi hyperbolicity. It is known that algebraic hyperbolicity is open in the countable Zariski topology, and there is a question by Demailly (see \cite[Remark 2.5]{demailly-97-algebraic-hyperbolicity}) asking if Kobayashi hyperbolicity is also open in the countable Zariski topology. It is natural to expect that Corollary~\ref{cor openness in euclidean} also holds in this topology. We will discuss a possible approach to this problem in Section~\ref{sec open problem}. It also interesting to know if we can obtain a similar statement for $\hyp^{(k)}$. The difficulty in this problem is the non-closeness of geometric condition.
\end{remark}

We compute hyperbolic indices in some special settings. First, we note that in the $1$-dimensional case, i.e., $X$ is a compact Riemann surface, the construction based on the Demailly-Semple tower does not make sense because of dimensional reasons. However, in this case, all positive $dd^c$-closed currents of bi-dimension $(1,1)$ on $X$ are of the form $c[X]$ for some positive constant $c$ by Lemma~\ref{lemma support lemma for ddc closed current}. Therefore, we can simply put their Euler characteristic as $c$ times the Euler characteristic of $X$. So, we can simply define $\hyp(X,\omega)\colonequals -\chi(X)/ \int_X \omega$.

\begin{proposition}\label{prop 1 dim case hyp compute}
    We have $\hyp(X) = -\chi(X)$. Moreover, if $X = C$ where $C \subset \C \P^2$ is a smooth curve of degree $d$, then $\hyp(C,\omega_{\FS}|_C) = d-3$.
\end{proposition}

\begin{proof}
    The first statement is trivial. For the second statement, by genus-degree formula, we have  $\hyp(C,\omega_{\FS}|_C) = \chi(C)/ \int_C \omega_{\FS} = (d^2-3d)/d = d-3$.
\end{proof}

\begin{remark}
    This proposition is just the classical result that if $X$ has genus greater than $2$, then $X$ is Kobayashi hyperbolic. The second statement says that, with respect to some canonical metrics, hyperbolicity increases as the degree increases. This suggests that we can expect a similar phenomenon for hypersurfaces in higher dimensions. We will give an affirmative answer for this question in Section~\ref{sec quantitative Kobayashi conjecture}.
\end{remark}

Next, we compute hyperbolic indices for complex tori and projective spaces.

\begin{proposition}\label{prop hyp of torus}
    Let $\C^n/\Lambda$ be a complex tori. Then, for all $k\in \Z^+ \cup \{\infty\}$ and K\"ahler metric $\omega$ on $\mathbb{C}^n/\Lambda$, we have $\hyp^{(k)} (\C^n/\Lambda,\omega) = \chyp^{(k)} (\C^n/\Lambda,\omega) = 0$.
\end{proposition}

\begin{proof}
     Put $X\colonequals \C^n/\Lambda$ and consider a K\"ahler form $\omega$ on $X$. Let $\mathscr{C} \subset X$ be an $1$-dimensional sub-torus. By Theorem~\ref{theorem coincide with k-Euler characteristic}, we have $\chi_{k}([\mathscr{C}]) = {\chi}_k(\mathscr{C}) =0 $ for all $k$. Thus, we must have $\hyp^{(k)}(X,\omega) \leq 0$ and $\chyp^{(k)}(X,\omega) \leq 0$ for all $k$. Moreover, as $X$ has trivial tangent bundle, we have $X_1 = X\times \P^{n-1}$. Denote by $\pi_2:X_1 \to \P^{n-1}$ the projection map. Then $u_{1} = \{\pi_2^* (\omega_{\FS})\}$. Let $T$ be a positive closed current of bi-dimension $(1,1)$ on $X_1$. We must have $T \wedge \pi_2^* (\omega_{\FS}) \geq 0$. This implies that $\hyp^{(1)}(X,\omega) \geq 0$. By Proposition~\ref{prop basic property of hyp}(i)-(ii), we must have $\hyp^{(k)} (X,\omega) = \chyp^{(k)} (X,\omega) = 0$ for all $k\in \Z^+\cup \{\infty\}$.
\end{proof}

\begin{proposition}\label{prop hyp of Pn}
    For all $k\in \Z^+\cup\{\infty\}$, we have 
    \[\hyp^{(k)}(\P^n) = \hyp^{(k)}(\P^n,\omega_{\FS}) = \chyp^{(k)}(\P^n) = \chyp^{(k)}(\P^n,\omega_{\FS}) = -2.\]
    Moreover, we have $\chyp^{(k)}(X,\omega_{\FS}|_X) \geq -2$ for all submanifold $X$ of $\P^n$.
\end{proposition}

\begin{proof}
    The case $n=1$ is from Proposition~\ref{prop 1 dim case hyp compute}. Suppose that $n\geq 2$. Let $T$ be a $k$-liftable current of mass $1$ with respect to the Fubini-Study metric $\omega_{\FS}$. Then, by Proposition~\ref{prop lower bound in Pn}, we have $-\overline{\chi}_k(T) \geq -2$. Thus, we obtain $\hyp^{(k)}(\P^n,\omega_\FS) \geq \chyp^{(k)}(\P^n,\omega_\FS) \geq -2$. Moreover, by choosing $T = [\mathscr{C}]$ where $\mathscr{C}$ is a smooth rational curve in $\P^n$, we have $\chyp^{(k)}(\P^n,\omega_\FS) \leq \hyp^{(k)}(\P^n,\omega_\FS) \leq -2$. Thus, $\chyp^{(k)}(\P^n,\omega_\FS) = \hyp^{(k)}(\P^n,\omega_\FS) = -2$. For another K\"ahler metric $\omega$ on $\P^n$, it is known that $\{\omega\} = c \{\omega_\FS\}$ for some positive constant $\mathscr{C}$. The normalized condition implies that $c=1$. Hence, $\chyp^{(k)}(\P^n) = \hyp^{(k)}(\P^n) = -2$. The last statement follows from Proposition~\ref{prop basic property of hyp}(vii).
\end{proof}

We end this section by proving a relation between the automorphism group and the hyperbolic indices.

\begin{proposition}\label{prop hyp of infinite aut mfd}
    Let $X$ be a compact K\"ahler manifold and $k\in \Z^+$. Suppose that there exists a $k$-liftable current $T$ such that $\{T\}- \epsilon \{\omega^{n-1}\}$ is pseudoeffective for some K\"ahler metric $\omega$ and $\epsilon > 0$. Suppose that $\Aut(X)/\Aut_0(X)$ is infinite. Then $\chyp^{(k)}(X) \leq \hyp^{(k)}(X)\leq 0$. In particular, this statement always holds for projective manifolds.
\end{proposition}

\begin{proof}
    Recall that $\Aut(X)$ is a complex Lie group of finite dimension and $\Aut_0(X)$ is the identity connected component. By the Fujiki-Lieberman type theorem (\cite[Theorem 2.1]{Dinh-Hu-Zhang-Fujiki-Lieberman-bigcase}), there exists an element $g$ in $\Aut(X)$ such that $\|(g^m)^* \{T\}\| \to \infty$ as $m\to \infty$. By Proposition~\ref{prop liftable preserved}, the family of $k$-liftable currents $(g^m)^* T$ have the same $\chi_k$. Let $\omega$ be a K\"ahler form on $X$. Letting $m\to \infty$, we obtain \[\hyp^{(k)}(X,\omega) \leq \liminf\limits_{m\to \infty} \frac{-{\chi}_k((g^m)^*(T))}{ \|(g^m)^*( T)\|_\omega} = \liminf\limits_{m\to \infty} \frac{-{\chi}_k(T)}{ \|(g^m)^*( T)\|_\omega} = 0,\]
    as desired. When $X$ is a projective manifold, we consider the curve $\mathscr{C}$ as the intersection of suitable divisors, following the proof of Proposition~\ref{prop basic property of hyp}(v). Then the current $[\mathscr{C}]$ is $\infty$-liftable and $\{\mathscr{C}\} = \{\omega^{n-1}\}$ for some K\"ahler class $\omega$. The result follows.
\end{proof}

\begin{remark}
    In \cite{kobayashi-76-intrinsic}, Kobayashi proved that if $X$ is a compact hyperbolic manifold, then the automorphism group $\Aut(X)$ is finite. Combining with Theorem~\ref{theorem chyp > 0 implies Kobayashi hyperbolic}, this implies that $\chyp^{(k)}(X) \leq 0$. However, it is not clear to us whether one can conclude $\hyp^{(k)}(X)\leq 0$ in this situation.
\end{remark}

\section{Relation with Kobayashi hyperbolicity}\label{sec relation with Kobayashi hyperbolic}

The main goal of this section is to investigate the relationship between hyperbolic indices and Kobayashi hyperbolicity. We first show that if $\chyp^{(\infty)}(X) > 0$, then $X$ is Kobayashi hyperbolic. This is done by showing that Nevanlinna currents are liftable and have non-negative cohomological Euler characteristics. In the remaining part, we prove that if $X$ admits a negative $k$-jet curvature in the sense of Demailly, then $\hyp^{(k)}(X)>0$. Combining these two statements gives a different proof for some classical theorems in complex hyperbolicity. All the theorems in this section hold for $\hyp_*^{(k)}$ and $\chyp_*^{(k)}$, using the same arguments. The manifold $X$ is this section is a compact K\"ahler manifold.

\subsection{Euler characteristic of Nevanlinna currents}

 Let $V$ be a holomorphic subbundle of $T_X$. Let $f:\C \to X$ be a transcendental entire curve in $X$ that is tangent to $V$. Recall that a non-constant entire curve $f: \C \to X$ is called \textit{transcendental} if $\lim\limits_{r\to \infty} T_{f}(r,\omega)/\log r = \infty$. Let $\widetilde{X} = \P(V)$ and $\widetilde{f}$ be the lifting of $f$ on $\widetilde{X}$. Let $\omega$ be a K\"ahler metric on $X$. Then $\omega$ induces a metric $h_1$ on $\mathcal{O}_{\widetilde{X}}(-1)$ (see subsection~\ref{subsec projectivized 1 jet} for details). We recall McQuillan's tautological inequality in \cite{McQuillan}.

 \begin{lemma}\label{lemma limit implies tautological inequality}
     There exists a measurable set $E$ of finite Lebesgue measure such that
     \begin{equation}\label{eq tautological ineq main ineq}\limsup_{r\to \infty,r \notin E} \frac{1}{T_{ f}(r,\omega)} \int_1^r \frac{dt}{t} \int_{\D_t} \widetilde f^* \Theta_{h_1^{-1}}(\mathcal{O}_{\widetilde{X}}(1)) \leq 0.
     \end{equation}
 \end{lemma}

 \begin{proof}
     This is a consequence of the Lelong-Poincar\'e formula, combined with Jensen formula and Lemma~\ref{lemma logarithmic}. See \cite[Lemma 4.1]{VietTuan} or \cite[Lemma 2.1]{deng-degeneracy} for more details.
 \end{proof}

We now consider higher liftings. Let $f:\C\to X$ be a transcendental entire curve in $X$. Let $f_{[k]}:\C \to X_k$ be the $k$-lifting of $f$ to $X_k$ and $\omega_{[k]}$ be the K\"ahler metric on $X_k$ constructed in \eqref{eq define Kahler form on lifting}. We have the following consequence of Lemma~\ref{lemma limit implies tautological inequality}.

\begin{lemma}\label{lemma decreasing of char function}
    For each $k\in \Z^+$, there exists a measurable set $E_k$ of finite Lebesgue measure such that 
    \begin{equation*}\label{eq decreasing of char function}
        \lim_{r\to \infty} \frac{\log r}{T_{f_{[k]}}(r,\omega_{[k]})} = 0 ,\qquad \limsup_{r\to \infty,r\notin E_k}  \frac{T_{f_{[k]}}(r,\omega_{[k]})}{T_{f_{[k-1]}}(r,\omega_{[k-1]})} \leq C_{[k]}.
    \end{equation*}
\end{lemma}

\begin{proof}
    Bounding $\pi_{k,k-1}^* (\omega_{[k-1]}) \leq C \omega_{[k]}$ by some constant $C$, we obtain   
    \[T_{f_{[k-1]}}(r,\omega_{[k-1]}) = T_{f_{[k]}}(r,\pi_{k,k-1}^* \omega_{[k-1]}) \leq C\ T_{f_{[k]}}(r,\omega_{[k]}). \] Thus, by induction, we have 
    \[\lim\limits_{r\to \infty} \frac{\log r}{T_{f_{[k]}}(r,\omega_{[k]})} = 0 \quad \text{for every}\quad k\in \Z^+.\]
     Moreover, since $\omega_{[k]} = C_{[k]}\cdot \pi_{k,k-1}^* (\omega_{[k-1]}) + \Theta_{h_{[k]}}(\mathcal{O}_{X_{k}}(1))$, we can apply Lemma~\ref{lemma limit implies tautological inequality} for $f_{[k-1]}$ instead of $f$ and the complex directed manifold $(X_{k-1},V_{k-1})$ to get the second inequality.
\end{proof}

Let $k \in \Z^+$. We now show that we can always construct a $k$-liftable Nevanlinna current. By Lemma~\ref{lemma Ahlfors} and Lemma~\ref{lemma decreasing of char function}, we can choose a sequence $(r_m)_{m\in \Z^+}$ such that for every $l\leq k$, we have
\begin{equation}\label{eq cond for liftable Nevanlinna}\lim_{m\to \infty} \frac{S_{f_{[l]}}(r_m, \omega_{[l]})}{T_{f_{[l]}}(r_m, \omega_{[l]})} = 0 \qquad \text{and} \qquad \lim_{m\to \infty} \frac{T_{f_{[l]}}(r_m, \omega_{[l]})}{T_{f_{[l-1]}}(r_m, \omega_{[l-1]})} \leq C_{[l]}.\end{equation}
Let $S$ be a Nevanlinna current associated with $f$, constructed by using the sequence $(r_m)$. Then, we have the following:

\begin{lemma}\label{lemma Nevanlinna current liftable}
    $S$ is $k$-liftable with $\overline{\chi}_k(S) \geq 0$.
\end{lemma}

\begin{proof}
    For $l\leq k$, consider the sequence of positive currents 
    \[S_{[l],r_m}\colonequals\frac{1}{T_{f}(r_m,\omega)} \int_0^{r_m} \frac{dt}{t} (f_{[l]})_*([\D_t]).\]
    By \eqref{eq cond for liftable Nevanlinna}, this sequence has uniformly bounded mass with respect to $\omega_{[l]}$. Moreover, any weak limit of this sequence is a positive closed current. Let $S_{[l]}$ be a weak limit of this sequence. Then $S_{[l]}$ is a positive closed current that directed by $V_l$ and satisfies $(\pi_{l,0})_*(S_{[l]}) = S$. Thus, we conclude that $S$ is $k$-liftable.

    For $l\geq 2$, since $f_{[l]} \not \subset D_l$, by Lemma~\ref{lemma Nevanlinna first main theorem}, we have $\{S_{[l]}\}\cdot \{D_l\} \geq 0$. Moreover, by construction, we have $(\pi_{k,l})_*(S_{[k]}) = S_{[l]}$. Therefore, we deduce that 
    \[\{S_{[k]}\} \cdot \pi_{k,l}^* \{D_l\} = \{(\pi_{k,l})_*(S_{[k]})\} \cdot \{D_l\} = \{S_{[l]}\} \cdot \{D_l\} \geq 0.\]
    Thus, $S_{[k]}$ is a cohomological $k$-lifting. Moreover, by Lemma~\ref{lemma limit implies tautological inequality}, we have $\{S_{[k]}\} \cdot u_k \leq 0$. Hence, we obtain $\overline{\chi}_k(S) \geq 0 $ as desired.
\end{proof}

We obtain the following direct corollary:

\begin{theorem}\label{theorem chyp > 0 implies Kobayashi hyperbolic}
    Suppose that $\chyp^{(k)}(X,\omega) > 0$ for some $k\in \Z^+ \cup \{\infty\}$ and K\"ahler form $\omega$. Then $X$ is Kobayashi hyperbolic.
\end{theorem}

\begin{proof}
    By definition, we only need to prove this theorem for $k\in \Z^+$. Assume that $X$ is not Kobayashi hyperbolic. Then, there exists a non-constant entire curve $f:\C \to X$. If $T_f(r,\omega) = O(\log r)$, then by Lemma~\ref{lemma Demailly transcendental}, $f$ extends to a rational curve $\mathscr{C} \subset X$. Hence, we have $-\overline{\chi}_k([\mathscr{C}]) \leq -\chi_k([\mathscr{C}]) \leq -\chi_\infty(\mathscr{C}) = -2$, which contradicts $\chyp^{(k)}(X,\omega) > 0$. If $\lim\limits_{r\to \infty} T_f(r,\omega)/\log r= \infty$, then $f$ is a transcendental entire curve. Let $S$ be the Nevanlinna current of $f$ constructed in Lemma~\ref{lemma Nevanlinna current liftable}. Then $-\overline{\chi}_k(S)\leq 0$, which also contradicts $\chyp^{(k)}(X,\omega) > 0$. The result follows.
\end{proof}

\begin{remark}
    In general, we do not know if the positivity of geometric hyperbolic indices implies Kobayashi hyperbolicity. For this, one need to estimate the geometric Euler characteristics of Nevanlinna currents, which is currently beyond our reach (see Question~\ref{question estimate geometric Euler Nevanlinna current}). However, when $X$ is a projective manifold, we have a partial answer for this question (see Theorem~\ref{theorem hyp big implies hyperbolic in projective case}).
\end{remark}

\subsection{Negativity of jet curvature}\label{sec negative jet curv}

We now investigate the relation of $k$-jet metric and hyperbolic indices. We first recall the notion of $k$-jet metric from \cite[Section 7]{demailly-97-algebraic-hyperbolicity}.

A \textit{singular $k$-jet metric} on $X$ is a singular Hermitian metric $h_k$ on the line bundle $\mathcal{O}_{X_k}(-1)$ over $X_k$. We always assume that the weight function of this metric is a quasi-psh function. Then the curvature current $\Theta_{h_k}(\mathcal{O}_{X_k}(-1))$ can be written in the form $\alpha +\theta + dd^c \varphi$ for some $\theta$-psh function, where $\alpha,\theta$ are closed smooth $(1,1)$-forms and $\alpha+\theta$ represents $c_1(\mathcal{O}_{X_k}(-1))$. Denote $\Sigma_k \subset X_k$ the unbounded locus of $\varphi$, which is the set of $x\in X_k$ such that $\varphi$ is not locally bounded near $x$. We call $\Sigma_k$ the \textit{degeneration set} of $h_k$.

Given a singular $k$-jet metric $h_k$ on $X$. We say that $h_k$ has \textit{negative jet curvature} if there exists $\epsilon > 0$ and a K\"ahler metric $\omega$ on $X_k$ such that:
\begin{equation}\label{eq define negative jet curvature}
    \Theta_{h_k^{-1}}(\mathcal{O}_{X_k}(1)) (\xi,\xi) \geq \epsilon |\xi|^2_{\omega} \text{ for all } \xi \in V_k \subset TX_k
\end{equation} 
in the sense of distributions.

It has been proved in \cite[Theorem 7.8]{demailly-97-algebraic-hyperbolicity} using the Ahlfors-Schwarz lemma that if $X$ has a negative jet curvature $k$-jet metric $h_k$ such that $\Sigma_k \subset X_k^{[\sing]}$, then $X$ is Kobayashi hyperbolic. Our goal in this subsection is to show that this condition also implies that $X$ has positive geometric hyperbolic indices.

\begin{theorem}\label{theorem negative jet curvature implies pluripotential hyperbolic}
    Let $k\in \Z^+$. Suppose that $X$ admits a $k$-jet metric $h_k$ of negative jet curvature such that $\Sigma_k \subset X_k^{[\sing]}$. Then $\hyp^{(k)}(X,\omega)>0$. In particular, if $X$ has ample cotangent bundle or $X$ admits a metric with negative holomorphic sectional curvature, then $\hyp^{(1)}(X,\omega)>0$.
\end{theorem}

This theorem is a consequence of the following more general result.

\begin{lemma}\label{lemma direct current has positive product with good current}
    Let $(X,\omega)$ be a compact K\"ahler manifold of dimension $n$. Let $V$ be a rank $r$ subbundle of $T_X$. Let $S$ be a closed current of bi-degree $(1,1)$ such that $S = \alpha + \theta + dd^c \varphi$, where $\alpha,\theta$ are closed smooth $(1,1)$-forms, and $\varphi$ is a $\theta$-psh function such that the unbounded locus of $\varphi$ lies inside some divisor $D$. Let $T$ be a positive $dd^c$-closed current of bi-dimension $(1,1)$ on $X$ such that $T$ is directed by $V$, puts no mass on $D$, and $\mathbf{1}_{X\setminus D} T\neq 0$. Suppose that there exists a constant $\epsilon >0$ such that $S(\xi,\xi) \geq \epsilon |\xi|^2_{\omega}$ in the sense of distributions, for all $\xi \in V$. Then $\{T\} \cdot \{S\} > 0$.
\end{lemma}

\begin{proof}
    By Demailly's regularization theorem (\cite{Demailly_regula_11current}), we can assume that $\varphi$ is smooth on $X\setminus D$. Choose a local chart $\Omega$ of $X\setminus D$ and a compact subset $\Omega'$ of $\Omega$ such that $\mathbf{1}_{\Omega'} T \neq 0$ (this is always possible since $\mathbf{1}_{X\setminus D} T \neq 0$. Let $\alpha_1,\ldots,\alpha_m$ be holomorphic $1$-forms on $\Omega$ that define the distribution $V$. On $\Omega$, we consider the orthogonal decomposition $T_X = V \oplus V^{\bot}$ with respect to the metric $\omega$. Let $\xi \in T_X$. Then, on $\Omega$, we can decompose $\xi = v + w$ where $v \in V$ and $w \in V^\bot$.

    By assumption, we have $S(v,v) \geq \epsilon|v|^2_\omega$. Moreover, since $\Omega'$ is compact, there exists a constant $C>0$ such that $|S(v,w)| \leq C |v|_\omega|w|_\omega$ and $|S(w,w)| \leq C|w|_\omega^2$ on $\Omega'$. By Young's inequality, we have \[2C|v|_\omega|w|_\omega \leq \frac{\epsilon}{2} |v|_\omega^2 + \frac{2C^2}{\epsilon} |w|^2_\omega.\]
    Write $S(\xi,\xi) = S(v,v) + 2\Re S(v,w) + S(w,w)$ and combine all the inequalities, we obtain
    \begin{equation}\label{eq control S(xi,xi)}S(\xi,\xi) \geq \frac{\epsilon}{2} |v|^2_\omega - C_1 |w|^2_\omega\end{equation}
    on $\Omega'$ for some constant $C_1>0$.

    Define $A\colonequals \sum_{j=1}^m  i \alpha_j \wedge \overline{\alpha_j}$. Then $A$ vanishes on $V$. Moreover, on $\Omega'$, we have $A(w,w) \geq \epsilon' |w|^2_\omega$ for some small constant $\epsilon'$. This combines with \eqref{eq control S(xi,xi)} gives us 
    \[(S+ C_2 A)(\xi,\xi) \geq \frac{\epsilon}{2}|v|_{\omega}^2 + (C_2\epsilon'-C_1)|w|^2_\omega \geq \frac{\epsilon}{2} |\xi|^2_\omega \text{ on }\Omega',\]
    where $C_2>0$ is a big enough constant.

    Hence, we deduce that
    \[\int_{\Omega'} T \wedge S = \int_{\Omega'} T \wedge (S+ C_2A) \geq \int_{\Omega'} T \wedge \frac{\epsilon}{2}\omega > 0.\]
    Since $S$ is smooth on $X\setminus D$, we can define the measure $T\wedge S$ on $X\setminus D$. Thus, we obtain 
    \[\int_{X\setminus D} T \wedge S > 0.\]

    Since $T$ puts no mass on $D$, by Lemma~\ref{lemma puts no mass implies positive cup product}, there exists a density positive measure $R$ between $T$ and $\theta + dd^c \varphi$ such that $\{T\} \cdot \{\theta\} = \|R\|_X$. By Proposition~\ref{prop density = classical local case}, we have \[\mathbf{1}_{X\setminus D} (T \wedge \alpha +R) = \mathbf{1}_{X\setminus D} T \wedge S.\] 
    Since $T$ puts no mass on $D$, $\mathbf{1}_D(T\wedge \alpha) = 0$. Moreover, since $R$ is a positive measure, we have $\mathbf{1}_D R$ is also a positive measure. Hence,
    \begin{align*}
        \{T\} \cdot \{S\} &= \int_X T \wedge \alpha + \{T\} \cdot \{\theta\} \\
        &= \int_{X\setminus D} (T\wedge \alpha + R) + \int_{X} \mathbf{1}_DR \\
        &\geq \int_{X\setminus D} T \wedge S > 0.
    \end{align*}
    The result follows.
\end{proof}

\begin{proof}[Proof of Theorem~\ref{theorem negative jet curvature implies pluripotential hyperbolic}]
    In this situation, $T$ is a geometric $k$-lifting of some $k$-liftable current on $X$ and $D = X_k^{[\sing]}$. We only need to check if $\mathbf{1}_{X\setminus X_k^{[\sing]}} T\neq 0$. It is always true because $\mathbf{1}_{X\setminus X_k^{[\sing]}} T$ is also a $k$-lifting by Proposition~\ref{prop vanish when restrict to singular part}. Thus, we get $\hyp^{(k)}(X,\omega)>0$ by applying Lemma~\ref{lemma direct current has positive product with good current}.

    It is known (see \cite[Remark 7.5]{demailly-97-algebraic-hyperbolicity}) that $X$ admits a smooth $1$-jet metric with negative jet curvature if and only if $X$ admits a metric with negative holomorphic sectional curvature. Moreover, if $X$ as ample cotangent bundle, then $X$ admits a metric with negative holomorphic sectional curvature. The result follows.
\end{proof}

\begin{remark}
    A combination of Theorem~\ref{theorem chyp > 0 implies Kobayashi hyperbolic} and Theorem~\ref{theorem negative jet curvature implies pluripotential hyperbolic} provides a current-theoretic proof for the classical fact: if $X$ admits a metric with negative holomorphic sectional curvature, then $X$ is Kobayashi hyperbolic. We also note that, given $k\in \Z^+$, Demailly provides examples of hyperbolic surfaces (depending on $k$) $X$ that do not admit negative jet curvature $k$-jet metric (see \cite[Theorem 8.2]{demailly-97-algebraic-hyperbolicity}). We note that the manifold in Demailly's example also satisfies $\hyp^{(k)}(X,\omega) \leq 0$ since it contains a curve $\mathscr{C}$ such that $\chi_k(\mathscr{C})\geq 0$. Therefore, there exist hyperbolic surfaces $X$ with $\hyp^{(k)}(X,\omega)\leq 0$ for $k\in \Z^+$.
\end{remark}

\section{Hyperbolic indices of general hypersurfaces}
\label{sec quantitative Kobayashi conjecture}

In this section, we study hyperbolic indices of general hypersurfaces in $\P^n$ and prove Theorem~\ref{theorem quantitative Kobayashi conjecture}. First, we prove a comparison theorem between geometric and cohomological Euler characteristic of liftable currents on projective manifolds. Next, we prove a transcendental version of the fundamental vanishing theorem. The proof of Theorem~\ref{theorem quantitative Kobayashi conjecture} is a combination of this vanishing theorem and the theory of density currents mentioned in Section~\ref{sec background currents}. All the theorems of this section hold in the $dd^c$-closed setting.

\subsection{Comparison theorem} 
Let $X\subset \P^n$ be a projective manifold of dimension greater than $1$, and let $\omega\colonequals \omega_\FS|_X$ be the restriction of the Fubini-Study metric to $X$.

\begin{theorem}\label{theorem comparison in proj mfd}
     Let $k\in \Z^+\cup\{\infty\}$. Let $T$ be a $k$-liftable current on $X$ such that $\|T\|_{\omega} = 1$. Then, for $l\leq k, l\in \Z^+$, we have $|\chi_l(T) - \overline{\chi}_l(T)| \leq - (3^{l-1}-1)\overline{\chi}_k(T) + 2\cdot(3^{l-1}-1)$.
\end{theorem}

\begin{proof}
    Let $\iota$ be the embedding map $\iota: X\hookrightarrow \P^n$. By Proposition~\ref{prop do not depend on ambient manifold}, we have $\chi_l(T) = \chi_l(\iota_*(T))$ and $\overline{\chi}_l(T) \leq  \overline{\chi}_l(\iota_*(T))$. Moreover, since $\overline{\chi}_l\geq \chi_l$, we have 
    \[|\chi_l(T) - \overline{\chi}_l(T)| = \overline{\chi}_l(T) - \chi_l(T) \leq   \overline{\chi}_l(\iota_*(T)) - \chi_l(\iota_*(T)). \]
    Thus, we can assume that $X= \P^n$ and $\|T\|_{\omega_{\FS}} = 1$.

    We have known that $\chi_1 = \overline{\chi}_1$. Therefore, we start with $l=2$. By Proposition~\ref{prop pick good k-lifing}, there exists a cohomological $2$-lifting $T_{[2]}$ of $T$ such that $-\overline{\chi}_2(T) = \{T_{[2]}\} \cdot u_2$. We define $T_{[2]}' \colonequals \mathbf{1}_{X_2^{[\mathrm{reg}]}} T_{[2]}$ and $T_{[2]}''\colonequals \mathbf{1}_{X_2^{[\sing]}}T_{[2]}$. Since $T_{[2]}'$ is a geometric $2$-lifting of $T$, we have $-\chi_2(T) \leq \{T_{[2]}'\}\cdot u_2$. We assume that $T_{[2]}'' \neq 0$. Otherwise, we deduce $\chi_2(T) = -\overline{\chi}_2(T)$. Let $S\colonequals (\pi_{2,1})_* (T_{[2]}'')$. If $S = 0$, then $\{T_{[2]}''\} \cdot u_2 > 0$, which contradicts $-\overline{\chi}_2(T) = \{T_{[2]}\} \cdot u_2$. Therefore, we can assume that $S\neq 0$.

    Let $L_1'$ be a very ample line bundle on $X_1$ of the form \eqref{eq construct ample line bundle on X_k}. Pick a K\"ahler metric $\omega_{L_1'}$ in the first Chern class of $L_1'$ such that $\omega_{L_1'} = \varphi_1^* (\omega_{\FS})$ where $\varphi_1: X_1 \hookrightarrow \P(H^0(X_1,L_1')^*)$ is the embedding map. Since $T_{[2]}''$ is a $1$-lifting of $S$, we deduce from Proposition~\ref{prop lower bound in Pn} and the construction of $L_1'$ that $\{T_{[2]}''\} \cdot u_2 \geq -2 \|S \|_{\omega_{L_1'}} = -2\{S\}\cdot u_1$. We deduce that $-\chi_2(T) \leq -\overline{\chi}_2(T) + 2\{S\} \cdot u_1$.

    Let $L_1$ be the nef line bundle construct on $X_1$ in Proposition~\ref{prop construct nef line bundle on X_k}. Recall that $L_1 = \mathcal{O}_{X_1}(1) \otimes \pi_{1,0}^* \mathcal{O}_X(2)$. Since $L_1$ is nef and $(\pi_{1,0})_* (S) = 0$, by Demailly-P\u{a}un theorem \cite{Demailly-Paun}, we obtain 
    \[\{S\} \cdot u_1 = \{S\} \cdot L_1 \leq \{(\pi_{2,1})_* (T_{[2]})\} \cdot L_1 \leq -\overline{\chi}_2(T) +2 \|T\|_{\omega_{\FS}} \leq -\overline{\chi}_k(T) + 2.\]
    Thus, we conclude that $|\chi_2(T) - \overline{\chi}_2(T)| \leq -2\overline{\chi}_k(T) + 4$.

    We now prove for general $l$. Let $T_{[l]}$ be a cohomological $l$-lifting of $T$ such that $-\overline{\chi}_l(T) = \{T_{[l]} \} \cdot u_l$. We define $T_{[l]}^l \colonequals \mathbf{1}_{D_l} T_{[l]}$. For $2\leq j< l$, we define inductively \[T_{[l]}^j \colonequals \mathbf{1}_{\pi_{l,j}^{-1}(D_j)} (T_{[l]} - T_{[l]}^l - \cdots - T_{[l]}^{j+1}).\]
    By Lemma~\ref{lemma support lemma for closed current}, $T_{[l]}^j$ is a positive closed current on $X_l$. Define $S_{[l]}^j \colonequals(\pi_{l,j-1})_*(T_{[l]}^j)$. Then, we have $(\pi_{j-1,j-1})_* (S_{[l]}^j) = 0$. Since $T_{[l]}^j$ puts no mass on $\pi_{l,p}^{-1}(D_p)$ for $j\leq p\leq l$, it is a geometric $(l-j+1)$-lifting of $S^{j}_{[l]}$. Using the very ample line bundle on $X_{j-1}$ as above and Proposition~\ref{prop lower bound in Pn}, we have 
    $\{T^j_{[l]} \} \cdot u_{l} \geq -2 \{S^j_{[l]} \} \cdot u_{j-1}$.

    Let $L_j$ be the nef line bundle construct on $X_j$ as in Proposition~\ref{prop construct nef line bundle on X_k}. Since $T_{[l]} - \sum_{j=2}^l T_{[l]}^j$ is a geometric $l$-lifting of $T$, by Demailly-P\u{a}un theorem, we obtain 
    \begin{align*}-\chi_l(T) &\leq -\overline{\chi}_l(T) - \sum_{j=2}^l \{T_{[l]}^j\} \cdot u_l \\
    &\leq -\overline{\chi}_l(T) +2 \sum_{j=2}^l \{S_{[l]}^j\} \cdot u_{j-1} \\
    &\leq -\overline{\chi}_l(T) +2 \sum_{j=2}^l \{(\pi_{l,j-1})_*(T_{[l]})\} \cdot L_{j-1} \\
    &\leq -\overline{\chi}_l(T)+2\sum_{j=2}^l\Big(-3^{j-2}\overline{\chi}_l(T)+2 \cdot 3^{j-2}\Big) \\
    &\leq -\overline{\chi}_l(T) - (3^{l-1}-1)\overline{\chi}_k(T) + 2\cdot(3^{l-1}-1).\end{align*}
    Thus, we conclude that $|\chi_l(T) - \overline{\chi}_l(T)| \leq - (3^{l-1}-1)\overline{\chi}_k(T) + 2\cdot(3^{l-1}-1)$ as desired.
\end{proof}

Applying this theorem for liftable Nevanlinna currents, we obtain the following relationship between the positivity of geometric indices and Kobayashi hyperbolicity.

\begin{theorem}\label{theorem hyp big implies hyperbolic in projective case}
    For $k\in \Z^+$, if $\hyp^{(k)}(X,\omega) > 2\cdot (3^{k-1}-1)$, then $X$ is Kobayashi hyperbolic.
\end{theorem}

We also obtain the following comparison theorem between geometric and cohomological hyperbolic indices.

\begin{theorem}\label{theorem compare geometric and cohomological hyp}
For $k\in \Z^+$, we have $ \chyp^{(k)}(X,\omega) \geq \frac{1}{3^{k-1}}\cdot \hyp^{(k)}(X,\omega) - \frac{2\cdot (3^{k-1}-1)}{3^{k-1}}$.
\end{theorem}

\begin{proof}
    By Proposition~\ref{prop basic property of hyp}(v), there exists a $k$-liftable current $T$ with $\|T\|_{\omega} = 1$ such that $\chyp^{(k)}(X,\omega) = -\overline{\chi}_k(T)$. By Theorem~\ref{theorem comparison in proj mfd}, we deduce that
    \begin{align*}\hyp^{(k)}(X,\omega) &\leq \chyp^{(k)}(X,\omega) + (3^{k-1}-1)\chyp^{(k)}(X,\omega) + 2\cdot (3^{k-1}-1)\\
    &= 3^{k-1} \chyp^{(k)}(X,\omega) + 2\cdot (3^{k-1}-1).\end{align*}
    The result follows.
\end{proof}

\subsection{Vanishing theorem}

We recall the fundamental vanishing theorem for entire curves. Let $X$ be a projective manifold, and let $A$ be an ample line bundle over $X$. Assume that there exist positive integers $k$ and $m$ such that 
\[H^0(X_k, \mathcal{O}_{X_k}(m) \otimes \pi_{k,0}^* A^{-1})\]
has non-zero sections $\sigma_1,\ldots,\sigma_N$ and let $Z \subset X_k$ be the base locus of these sections. Then, it is known by \cite{GreenGriffiths-80,demailly-97-algebraic-hyperbolicity,siuYeung-defects-97} that every entire curve $f:\C \to X$ satisfies $f_{[k]}(\C) \subset Z$. This property underlines the use of jet differentials in the study of Kobayashi hyperbolicity and Green-Griffiths algebraic degeneracy.

Let $X$ be a compact K\"ahler manifold of dimension $n$. Fix a cohomology class $\alpha \in H^{1,1}(X,\R)$. We define the cohomology class
\[\Psi_{k,m}^{\alpha} \colonequals c_1(\mathcal{O}_{X_k}(m))- \pi_{k,0}^* (\alpha) \in H^{1,1}(X_k,\R),\]
which plays a similar role as the line bundle $\mathcal{O}_{X_k}(m) \otimes \pi_{k,0}^* A^{-1}$. The positive closed currents in $\Psi_{k,m}^{\alpha}$ will serve as holomorphic sections.

Let $\beta \in H^{n-1,n-1}(X,\R)$. Then $\beta$ is called \textit{movable class} if $\beta$ belongs to the closed convex cone generated by classes of the form $\pi_*(\widetilde{\beta}_1 \cdots  \widetilde{\beta}_{n-1})$ where $\pi:\widetilde{X} \to X$ is a smooth modification and $\widetilde{\beta}_j$ are K\"ahler classes on $\widetilde{X}$ for $j = 1,\ldots,n-1$. This notion was introduced in \cite{BDPP} and was intensively studied by many authors. It is known that if $\beta$ is movable, then $\beta \cdot \alpha \geq 0$ for every pseudo-effective class $\alpha$. There is a conjecture in \cite{BDPP} that these conditions are equivalent (see \cite{WittNystrom-duality} for an affirmative answer in the case of projective manifolds). When $n=2$ the movable cone is the nef cone.

We have the following theorem. Part (i) of this theorem can be viewed as a transcendental version of the fundamental vanishing theorem for entire curves, while part (ii) provides an useful way to estimates Euler characteristic of liftable currents.

\begin{theorem}\label{theorem alpha_km positive implies bound on -chi_k}
    Suppose that there exist $k\in \Z^+, m \in \R^+$ such that  $\Psi^\alpha_{k,m}$ is pseudo-effective. Let $S$ be a positive closed current that represents $\Psi^\alpha_{k,m}$ and $Z\colonequals \{x\in X_k: \nu(S,x) > 0\}$. Let $T$ be a $k$-liftable current on $X$. Then, we have the following statements:
    \begin{itemize}
        \item[(i)] If $\alpha$ is a K\"ahler class and $\overline{\chi}_k(T) \geq 0$ (resp. $\chi_k(T) \geq 0)$, then there exists a cohomological (resp. geometric) $k$-lifting $T_{[k]}$ of $T$ such that $T_{[k]}$ puts mass on $Z$. A similar statement holds when $\alpha$ is a big class and $\{T\}$ is  movable.
        \item[(ii)] Assume that there exists an analytic subset $A$ of $X$ such that $\pi_{k,0}(Z) \subset A$ and $\alpha$ is a K\"ahler class. Then, we have the estimate 
        \[-{\chi}_{k}(\mathbf{1}_{X\setminus A} T) > m^{-1}\cdot \|\mathbf{1}_{X\setminus A}(T)\|_{\omega},\]
        where $\omega$ is a K\"ahler form on $\alpha$.
    \end{itemize}
\end{theorem}

\begin{proof}
    (i) Suppose that $\alpha$ is a K\"ahler class and $\omega\in \alpha$ is a K\"ahler form.
    Let $T_{[k]}$ be a cohomological $k$-lifting of $T$. Suppose that $T_{[k]}$ puts no mass on $Z$. Then, by Lemma~\ref{lemma puts no mass implies positive cup product}, we have $\{T_{[k]}\}\cdot \Psi_{k,m}^\alpha = \{T_{[k]}\} \cdot \{S\} \geq 0$. This is equivalent to $ m\cdot \{T_{[k]}\} \cdot u_k \geq \{T\} \cdot \alpha$. Therefore, if every cohomological $k$-lifting of $T$ puts no mass on $Z$, then we must have $\overline{\chi}_k(S) \leq -m^{-1}\cdot \|T\|_{\omega}$, which is a contradiction with $\overline{\chi}_k(T) \geq 0$. The proof in the situation where $\chi_k(T) \geq 0$ is similar. In the case where $\alpha$ is big and $\{T\}$ is movable, we use the inequality $\{T\} \cdot \alpha \geq 0$.

    (ii) Let $T' \colonequals \mathbf{1}_{X\setminus A} T$. Then, by Proposition~\ref{prop decomp for liftable current}, $T'$ is a $k$-liftable current. Let $T_{[k]}'$ be a geometric $k$-lifting of $T'$. Then $\mathbf{1}_{X_k\setminus \pi_{k,0}^{-1}(A)} T'_{[k]}$ is also a $k$-lifting of $T'$ which puts no mass on $Z$. By a similar argument as in the proof of (i), we only need to show that 
    \[\{\mathbf{1}_{X_k\setminus \pi_{k,0}^{-1}(A)} T'_{[k]}\} \cdot u_k \leq \{T'_{[k]}\} \cdot u_k\]
    or equivalent $\{\mathbf{1}_{\pi_{k,0}^{-1} (A)} T_{[k]}'\} \cdot u_k \geq 0$. This is done by using Proposition~\ref{prop increasing of chi_k} and noting that $(\pi_{k,0})_*(\mathbf{1}_{\pi_{k,0}^{-1} (A)} T_{[k]}') = 0$. The proof is complete. 
\end{proof}

\begin{remark}\label{remark search for more algebraic obstruction}
    We see that this theorem remains true if we replace the cone of directed currents by the larger cone of all positive closed currents on $X_k$. This observation suggests that the jet differential method does not take into account  the directed structure. Therefore, it motivates the search for more delicate algebraic obstructions on jet spaces.
\end{remark}

\subsection{Lower bound of hyperbolic indices}

We now back to the situation where $X$ is a smooth hypersurface of degree $d$ in $\P^{n+1}$. Recall that $K_X \simeq \mathcal{O}_{X}(d-n-2)$. Therefore, for $d \geq n+3$, $K_X$ is an ample line bundle. We recall some fundamental results in \cite{diverio-merker-rousseau-effective} and the refinements in \cite{darondeau-slanted}.

\begin{theorem}\label{theorem non effective alg degeneracy 1}\cite[Theorem 2.4]{diverio-merker-rousseau-effective}
    Let $X$ be a smooth hypersurface of degree $d$ in $\P^{n+1}$. Fix a small enough rational number $\delta > 0$. Then, there exists a positive integer $d_{n,\delta}$ such that 
    \[H^{0} \left(X_n,\mathcal{O}_{X_n}(m) \otimes \pi_{n,0}^* K_X^{-\delta m} \right) \neq 0,\]
    whenever $d\geq d_{n,\delta}$ and $m$ is large enough such that $\delta m$ is an integer.
\end{theorem}

\begin{theorem}\label{theorem non effective alg degeneracy 2}\cite{diverio-merker-rousseau-effective,darondeau-slanted}
    Let $X$ be a general smooth hypersurface of degree $d$ in $\P^{n+1}$. Suppose that $d$ and $m$ are large enough such that
    \[H^{0} \left(X_n,\mathcal{O}_{X_n}(m) \otimes \pi_{n,0}^* K_X^{-\delta m} \right) \neq 0 \quad \text{and} \quad d > \frac{5n+3}{\delta} + n + 2.\]
    For each $p \in \N, p \leq m$, let $Z_p$ be the base locus of the line bundle 
    \[\mathcal{O}_{X_n}(m) \otimes \pi_{n,0}^* \mathcal{O}_{X}\left(-\delta m (d - n - 2)+ p (5n+3) \right).\]
    Let $Z$ be the intersection of all $Z_p$. Then $\pi_{n,0}(Z)$ is a proper analytic subset of $X$.
\end{theorem}

\begin{theorem}\label{theorem effective alg degeneracy}\cite[Theorem 1.1]{diverio-merker-rousseau-effective}
    There exist effective constants $\delta(n)$ and $\lambda(n)$ such that for every $d\geq \lambda(n)$, there exists $m$ large enough such that the conclusion in Theorem~\ref{theorem non effective alg degeneracy 2} holds.
\end{theorem}

We can now prove a Green-Griffiths-Lang type conjecture in our setting.

\begin{theorem}\label{theorem generalized GGL for liftable current}
     Let $X$ be a general smooth hypersurface of degree $d$ in $\P^{n+1}$ and $\lambda(n)$ be the constant in Theorem~\ref{theorem effective alg degeneracy}. Then, if $d\geq \lambda (n)$, there exists a proper analytic subset $Y $ of $X$ such that every $n$-liftable current $T$ with 
     \[\frac{{\chi}_{n}(T)}{\|T\|_{\omega_\FS|_X}} \geq -\delta(d-n-2) + 5n+3\] (and therefore, all $n$-liftable currents with ${\chi}_{n} \geq 0$) puts mass on $Y$. 
\end{theorem}

\begin{proof}
     Let $S_{p,\min}$ be the current of minimal singularities in the first Chern class of the line bundle 
    \[\mathcal{O}_{X_n}(m) \otimes \pi_{n,0}^* \mathcal{O}_{X}\left(-\delta m (d - n - 2)+ p (5n+3) \right).\]
    Recall that a positive closed current in a pseudoeffective class $\alpha$ is said to have
    \textit{minimal singularities} if it is less singular than every other closed positive current in $\alpha$.
    Then, we see that $\{x\in X_n: \nu(S_{p,\min},x) > 0\} \subset Z_p$. Moreover,  observe that 
    \[S_{p,\min} +  (m-p)(5n+3) \pi_{n,0}^* (\omega_{\FS}|_X)\]
    is a positive closed current in the first Chern class of the line bundle
    \[\mathcal{O}_{X_n}(m) \otimes \pi_{n,0}^* \mathcal{O}_{X}\left(-\delta m (d - n - 2)+ m (5n+3) \right),\]
    we infer that $\{x\in X_n: \nu(S_{m,\min},x) > 0\} \subset \{x\in X_n: \nu(S_{p,\min},x) > 0\} \subset Z_p$ for every $p\leq m$. Thus, we have $\{x\in X_n: \nu(S_{m,\min},x) > 0\} \subset Z$.

    Let $T$ be a $n$-liftable current on $X$ such that
    \[\frac{{\chi}_{n}(T)}{\|T\|_{\omega_\FS|_X}} \geq -\delta(d-n-2) + 5n+3\] 
    By Theorem~\ref{theorem non effective alg degeneracy 2} and Theorem~\ref{theorem effective alg degeneracy}, $Y = \pi_{n,0}(Z)$ is an analytic subset of $X$. Suppose that $T$ puts no mass on $Y$. Then $T = \mathbf{1}_{X\setminus Y} T$. Therefore, by Theorem~\ref{theorem alpha_km positive implies bound on -chi_k}(ii), we have
    \[\frac{\chi_n(T)}{\|T\|_{\omega_{\FS}|_X}} < -\delta(d-n-2) + 5n+3,\]
    which is a contradictions. Hence $T$ puts mass on $Y$ as desired.
\end{proof}

We now prove the main result of this section. We need the following result of Riedl-Yang which allows us to control the dimension of Green-Griffiths locus.

\begin{theorem}\cite{riedl-yang-grassmannian-technique}\label{theorem pass from GGL to Kobayashi}
    There exist effective constants $\delta(n)$ and $\lambda(n)$ such that for every $d\geq \lambda(n)$, there exists $m$ large enough such that the conclusion in Theorem~\ref{theorem non effective alg degeneracy 2} holds and the analytic subset $\pi_{n,0}(Z)$ is of dimension $0$.
\end{theorem}

\begin{proof}[Proof of Theorem~\ref{theorem quantitative Kobayashi conjecture}]
Let $T$ be a $n$-liftable current on $X$ with $\|T\|_{\omega_{\FS}|_X} = 1$. Suppose that  $d\geq \lambda(n)$ for $\lambda(n)$ as in the Theorem~\ref{theorem pass from GGL to Kobayashi}. Let $Y \colonequals \pi_{0,2}(Z)$ be the analytic subset of $X$ in Theorem~\ref{theorem pass from GGL to Kobayashi}. Then $Y$ is a set of finite points. Thus, $\mathbf{1}_{X\setminus Y} T = T$ by dimensional reason. By Theorem~\ref{theorem alpha_km positive implies bound on -chi_k}(ii), we have
\[-\chi_n(T) \geq \delta(d-n-2)-(5n+3).\]
Thus, we obtain $\hyp^{(n)}(X,\omega_{\FS}|_X) \geq \delta(d-n-2)-(5n+3)$. By Theorem~\ref{theorem compare geometric and cohomological hyp}, we also duduce that
\[\hyp^{(n)}(X,\omega_{\FS}|_X) \geq  \frac{1}{3^{n-1}}\Big(\delta(d-n-2)-(5n+3)\Big) - \frac{2\cdot (3^{k-1}-1)}{3^{k-1}}. \]
Thus, both functions grow linearly in $d$ as $d\to \infty$. By Theorem~\ref{theorem hyp big implies hyperbolic in projective case}, $X_d$ is Kobayashi hyperbolic for $d$ sufficient large.
\end{proof}

\section{Further developments}\label{sec open problem}

In this final section, we discuss some future directions and propose some open questions related to our theory.

\subsection{An analytic approach to the Kobayashi conjecture} As mentioned in the Introduction, we outline a possible analytic approach to the Kobayashi conjecture. The first step is to build a similar theory in the logarithmic setting.

\begin{problem}\label{problem logarithmic setting}
    Develop a similar theory for the logarithmic setting. Can we prove a similar theorem as Theorem~\ref{theorem quantitative Kobayashi conjecture} in this setting?
\end{problem}

We now discuss this problem. We first note that it is tempting to try to use the logarithmic version of Demailly-Semple jet bundles developed by Dethloff-Lu in \cite{dethloff-lu-log-jet-bundles} to give a new definition for the hyperbolic indices in this setting. However, the geometric meaning in this case remains unclear, and the objects considered in this definition probably differ from the notion of liftable currents. Instead of that, we can try to mimic the notion of algebraic hyperbolicity in the logarithmic setting (see Section~\ref{sec complex hyperbolicity} for details). Let $X$ be a compact K\"ahler manifold and $D$ be an effective divisor. Given a liftable current $T$, we can define $i(T,D)$ which generalizes $i(\mathscr{C},D)$ for curves. To do that, we need to investigate the bi-meromorphic map between $X_k$ and the logarithmic version $X_k(-\log D)$. This has been done in \cite{cadorel-symmetric-diff-hypmfd-cusp} in the case where $k=1$. The geometric $k$-hyperbolic index in the logarithmic setting will be defined by the infimum of $-{\chi}_k(T) + i(T,D)$ where the infimum takes over all $k$-liftable currents that put no mass on $D$. The cohomological analogue will be defined by the infimum of $-\overline{\chi}_k(T) + i(T,D)$ where the infimum takes over all $k$-liftable currents $T$ such that $\{T\} \cdot \{D\} \geq 0$. When $T$ is a Nevanlinna current of a non-constant entire curve $f$, $i(T,D)$ is probably related to the geometric defect of $f$ to $D$ (see \cite{VietTuan} for a discussion in the $1$-jet level).

\begin{problem}\label{problem green theorem}
    (Green's theorem) Are the hyperbolic indices of the complement of union of (enough) hyperplanes in general position in projective space positive?
\end{problem}

The next step is to construct examples of low degree hypersurfaces with positive hyperbolic indices.

\begin{problem}\label{problem seeking example}
    Construct hypersurfaces $X_d$ of low degree $d$ in $\P^n$ such that $\hyp^{(k)}(X_d,\omega_{\FS}|_{X_d})$ or $\chyp^{(k)}(X_d,\omega_{\FS}|_{X_d})$ positive for some $k$.
\end{problem}

The current most effective method in the literature is the deformation method, which deduces the hyperbolicity in the compact case from the logarithmic case (see e.g. \cite{duval-sextique,huynh-examples-lowdim,huynh-examples-alldim}). We can try to use this method to answer Problem~\ref{problem seeking example}. The dynamical method in the proof of Proposition~\ref{prop lower bound in Pn} or Proposition~\ref{prop hyp of infinite aut mfd} may be a useful technique.

\begin{remark}\label{remark deform current}
    To use the deformation method, one needs to use the openness in the Euclidean topology property of Kobayashi hyperbolicity. In our setting, we only have this property for the cohomological notion.
\end{remark}

The last step requires an openness theorem in the (countable) Zariski topology.

\begin{question}\label{question open k-hyperbolic countable zariski}
    For $k\in \Z^+ \cup \{\infty\}$, is the positivity of hyperbolic indices an open property in countable Zariski topology for algebraic families of projective manifolds (or families of compact K\"ahler manifolds)?
\end{question}

We now discuss a possible approach to this question. Let $\pi: \mathcal{X}\to S$ be a holomorphic submersion between a K\"ahler manifold $\mathcal{X}$ of dimension $m+n$ and a connected complex manifold $S$ of dimension $n$. Then each fiber of $\pi$ is a compact $n$-dimensional K\"ahler manifold. We assume they are connected.  Let $T_{\mathcal{X}/S}$ denote the relative tangent bundle with respect to the submersion $\pi$. This is a rank $n$ holomorphic subbundle of $T \mathcal{X}$ on $\mathcal{X}$. For each $s\in S$, put $X_s\colonequals \pi^{-1}(s)$. Then we have $T_{\mathcal{X}/S}|_{\pi^{-1}(s)} = T_{X_s}$. Let $V_{\mathcal{X}/S}$ be a rank $r$ holomorphic subbundle of $T_{\mathcal{X}/S}$ and put $V_s\colonequals V_{\mathcal{X}/S}|_{X_s}$ which is a rank $r$ holomorphic subbundle of $T X_s$. Denote by $\P(V_{\mathcal{X}/S})$ the projectivization of $V_{\mathcal{X/S}}$. Then $\pi$ induces a natural holomorphic submersion map $\widehat \pi: \P(V_{\mathcal{X}/S}) \to S$. Put $Y_s\colonequals \widehat \pi ^{-1}(s) = \P (V_s)$.

By Ehresmann's lemma, $\widehat \pi$ is a locally trivial fibration and therefore the cohomology groups $H^k(Y_s,\C)$. These groups form the Hodge bundle $s\mapsto H^{k}(Y_s,\C)$ over $S$. This bundle carries a natural flat connection $\nabla$, which is known as the \textit{Gauss-Manin connection}. It is known that the Dolbeault cohomology groups $H^{p,q}(Y_s,\C)$ form real analytic subbundles of $H^k(Y_s,\C)$. The Gauss-Manin connection $\nabla$ induces a connection $\nabla^{p,q}$ on the subbundle $H^{p,q}(Y_s,\C)$.

For each $s\in S$, let $\mathscr{T}_s \subset H^{n-1,n-1}(Y_s,\R)$ denote the cone of cohomology classes $\alpha$ containing a positive closed (or $dd^c$-closed) current $T$ that is directed by $\widetilde{V_s}$, where $\widetilde{V_s}$ is defined as in \eqref{eq define tilde V}.

\begin{question}\label{question Demailly-Paun relative setting?}
    Is there a countable union of analytic subsets $S' \subset S$ such that: the cones $\mathscr{T}_s$ are invariant under parallel transport with respect to $\nabla^{n-1,n-1}$ for every $s\in S\setminus S'$?
\end{question}

This question can be viewed as a version of the celebrated Demailly-P\u{a}un theorem (see \cite[Theorem 0.8]{Demailly-Paun}) in this ``relative" setting. Similar to \cite{Demailly-Paun}, this requires a deeper understanding of the structure of the cone $\mathscr{T}_s$. Affirmative answers to Question \ref{question Demailly-Paun relative setting?} will be helpful in the study of Question~\ref{question open k-hyperbolic countable zariski}.

\subsection{Open questions} We collect some open questions related to our theory.

\begin{problem}
    Can we construct a (weakly) $\infty$-liftable current on a general compact complex manifold?
\end{problem}

By \cite[Section 2]{sibony-pfaff}, there always exists a positive $dd^c$-closed current $T_{[k]}$ on $X_k$ such that $T_{[k]}$ is directed by $V_k$. However, we do not know whether $(\pi_{k,0})_*(T_{[k]})\neq 0$. One drawback of the method used in the proof of Proposition~\ref{lemma existence liftable current induction level} is that when the center of the disc in $X_k$ approaches $X_{k}^{[\sing]}$, the area of this disc becomes very large compared to the area of its projection onto $X$.

\begin{question}\label{question estimate geometric Euler Nevanlinna current}
    Let $k\geq 2$ and $S_{[k]}$ be a Nevanlinna current of $f_{[k]}$. Is $S_{[k]}$ puts no mass on the singular locus $X_{k}^{[\sing]}$? If not, can we construct a counterexample?
\end{question}

A negative answer to this question would further illustrate the significant difference between algebraic curves and entire curves, as well as the flexibility inherent in the latter. Another interesting question to ask is whether we can have another proof for Duval's theorem using the liftings of Nevanlinna (or Ahlfors) currents.

 Nevanlinna currents play a key role in \cite{McQuillan}. Therefore, it is natural to ask if one can bring the result of \cite{McQuillan} to our setting.

\begin{problem}
    Let $X$ be a surface of general type such that $c_1^2(X) > c_2(X)$. Is there exists an analytic set $Y\subset X$ such that: if $T$ is an $\infty$-liftable current with $\overline{\chi}_\infty(T) \geq 0$ (or $\chi_\infty(T)$), then $T$ has positive mass on $Y$?
\end{problem}

An important step in \cite{McQuillan} is the blow-up process, which reduces the singularities of the foliation. This leads to the following question.

\begin{question}
    Let $T$ be an $\infty$-liftable current on $X$. Let $\pi:\widetilde{X} \to X$ be the blow up map at a point $x$. Let $\widetilde{T}$ denotes the strict transform of $T$ under $\pi$. Then, it is not hard to see that $\widetilde{T}$ is also an $\infty$-liftable current. Is ${\chi}_\infty(T) = {\chi}_\infty(\widetilde{T})$?
\end{question}

Let $J_k^{[\mathrm{reg}]}X/ \mathrm{Diff}_{k}(1)$ denote the invariant jet bundles. By \cite{demailly-97-algebraic-hyperbolicity}, we know that this space is isomorphic with $X_k^{[\mathrm{reg}]}$, and $X_k$ serves as a compactification of $J_k^{[\mathrm{reg}]}X/ \mathrm{Diff}_{k}(1)$. The current algebraic methods often replace $J_k^{[\mathrm{reg}]}X/\mathrm{Diff}_k(1)$ by different compactifications and use Demailly's Morse inequalities to check the bigness of some tautological bundle (see e.g. \cite{berczi-kirwan-non-reductive-hyperbolicity} or \cite{cadorel-hyperbolicity-viaGG} for a different compactification of Green-Griffiths jet bundles). Let $T_{[k]}$ be a $k$-lifting of a $k$-liftable current $T$. Then $T_{[k]}$ is a current on $X_k$ that puts no mass on $X_k \setminus X_k^{[\mathrm{reg}]}$. If we restrict $T_{[k]}$ to a positive closed current on $X_k^{[\mathrm{reg}]}$, then by the regularization theorem of currents in \cite{DS_regula}, we can extend $T_{[k]}$ trivially to a positive closed current $\widetilde{T_{[k]}}$ on any other compactifications of $X_{k}^{[\mathrm{reg}]}$ (for a proof of this fact, see \cite[Corollary 2.2]{dinh-sibony-entropy-mero-correspondence}). The hyperbolic index $\hyp^{(k)}$ can be defined using the intersection of $\{\widetilde{T_{[k]}}\}$ with the canonical tautological line bundle (depending on the choice of compactification). We expect that $\hyp^{(\infty)}$ does not depend on the choice of compactification. Note that using the compactification in \cite{berczi-kirwan-non-reductive-hyperbolicity} and algebraic methods, one can obtain better estimates for Theorem~\ref{theorem quantitative Kobayashi conjecture}. However, as we pursue a more analytic approach, we choose to use the Demailly-Semple tower because of its' simplicity.

\begin{question}
    Is the definition of $\hyp^{(\infty)}$ depend on the choice of compactification of $J_k^{[\mathrm{reg}]}X/\mathrm{Diff_k}(1)$? What happen if we use $J_k^{\mathrm{GG}}X$?
\end{question}

There is a conjecture by Demailly (see \cite[Conjecture 7.13]{demailly-97-algebraic-hyperbolicity} saying that a compact manifold $X$ is hyperbolic if and only if $X$ admits a negative curvature $k$-jet metric for $k$ large enough. Motivated by Theorem~\ref{theorem negative jet curvature implies pluripotential hyperbolic}, we propose a different version of this conjecture:

\begin{conjecture}\label{conj hyp > 0 equiv Kobayashi}
    A compact K\"ahler manifold $X$ is Kobayashi hyperbolic if and only if $\hyp^{(\infty)}(X)>0$.
\end{conjecture}

\bibliography{biblio_family_MA,biblio_Viet_papers,bib-kahlerRicci-flow,ag}
\bibliographystyle{alpha}

\bigskip

\noindent
\Addresses
\end{document}